\documentclass[aos]{imsart}

\usepackage{amsmath, amsfonts, amsthm, verbatim,
	amssymb,dsfont,color, accents}

\usepackage[utf8x]{inputenc}	
\usepackage{mathtools}

\usepackage{tikz}
\usetikzlibrary{arrows.meta,arrows,decorations.pathreplacing}
\RequirePackage[colorlinks,citecolor=blue,urlcolor=blue]{hyperref}

\startlocaldefs

\newcommand{\bfi}{\begin{fig}}
	\newcommand{\efi}{\end{fig}}
\newcommand{\btab}{\begin{tab}}
	\newcommand{\etab}{\end{tab}}
\newcommand{\barr}{\begin{array}}
	\newcommand{\earr}{\end{array}}
\newcommand{\beq}{\begin{equation}}
\newcommand{\eeq}{\end{equation}}
\newcommand{\bdis}{\begin{displaymath}}
\newcommand{\edis}{\end{displaymath}\noindent}

\newcommand{\bbn}{\mathbb{N}}
\newcommand{\N}{\mathbb{N}}
\newcommand{\bbz}{\mathbb{Z}}
\newcommand{\bbr}{\mathbb{R}}
\newcommand{\R}{\mathbb{R}}

\newcommand{\bbe}{\mathbb{E}}
\newcommand{\E}{\mathbb{E}}

\newcommand{\bbp}{\mathbb{P}}

\newcommand{\bbf}{\mathbb{F}}
\newcommand{\F}{\mathbb{F}}

\newcommand{\bone}{\mathds 1}
\newcommand{\dmid}{\mathrel{\Vert}}

\newcommand{\limp}{\stackrel{\mathbb{P}}{\longrightarrow}}
\newcommand{\limd}{\stackrel{d}{\longrightarrow}}

\newcommand{\nto}{{n\to\infty}}

\newcommand{\lec}{\lesssim}
\newcommand{\halmos}{\quad\hfill $\Box$}

\newcommand{\calc}{{\cal C}}

\newcommand{\calf}{{\cal F}}
\newcommand{\cala}{{\cal A}}

\newcommand{\calv}{{\cal V}}

\newcommand{\calz}{{\cal Z}}
\newcommand{\calw}{{\cal W}}

\newcommand{\al}{{\alpha}}
\newcommand{\la}{{\lambda}}
\newcommand{\La}{{\Lambda}}
\newcommand{\eps}{{\varepsilon}}
\newcommand{\ga}{{\gamma}}
\newcommand{\Ga}{{\Gamma}}

\newcommand{\si}{{\sigma}}
\newcommand{\om}{{\omega}}
\newcommand{\Om}{{\Omega}}

\newcommand{\var}{{\mathrm{Var}}}
\newcommand{\cov}{{\mathrm{Cov}}}

\newcommand{\ov}{\overline}
\newcommand{\un}{\underline}
\newcommand{\wh}{\widehat}
\newcommand{\wt}{\widetilde}

\newcommand{\dd}{\mathrm{d}}

\newcommand{\pd}{\partial}

\newcommand{\di}{\mathrm{dis}}

\newcommand{\tr}{\mathrm{tr}}

\DeclareMathOperator*{\diag}{diag}
\DeclareMathOperator*{\argmin}{argmin}

\newtheorem{Theorem}{Theorem}[section]
\newtheorem{Corollary}[Theorem]{Corollary}
\newtheorem{Lemma}[Theorem]{Lemma}
\newtheorem{Proposition}[Theorem]{Proposition}
\theoremstyle{definition}

\newtheorem{Remark}[Theorem]{Remark}
\newtheorem{Assumption}{Assumption}
\renewcommand{\theAssumption}{\Alph{Assumption}}
\makeatletter
\newcommand{\settheoremtag}[1]{
	\let\oldtheAssumption\theAssumption
	\renewcommand{\theAssumption}{#1}
	\g@addto@macro\endAssumption{
		\global\let\theAssumption\oldtheAssumption}
}
\makeatother

\newcommand{\lra}{\to}
\newcommand{\ind}{\mathbf{1}}

\newcommand{\wti}{\widetilde}

\newcommand{\de}{\delta}

\newcommand{\Den}{\Delta_n}
\newcommand{\DenOneHalf}{\Delta_n^{\frac{1}{2}}}
\newcommand{\DenMinOneHalf}{\Delta_n^{-\frac{1}{2}}}
\newcommand{\DenSquared}{\sqrt{\Delta_n}}

\newcommand{\Deni}{\Delta^n_i}
\newcommand{\uDeni}{\un{\Delta}^n_i}

\newcommand{\LeinsKonv}{\stackrel{L^1}{\Longrightarrow}}
\newcommand{\stabKonv}{\stackrel{\mathrm{st}}{\Longrightarrow}}

\newcommand{\bv}{\big \vert}

\newcommand{\Sig}{\Sigma}

\numberwithin{equation}{section}

\allowdisplaybreaks[3]
\endlocaldefs

\begin{document}

	\begin{frontmatter}
		\title{Rate-optimal estimation of mixed semimartingales}
		\runtitle{Rate-optimal estimation of mixed semimartingales}
		
		\begin{aug}
			\author[A]{\fnms{Carsten H.} \snm{Chong}\ead[label=e1]{carstenchong@ust.hk}},
			\author[B]{\fnms{Thomas} \snm{Delerue}\ead[label=e2]{thomas.delerue@gmail.com}}
			\and
			\author[C]{\fnms{Fabian} \snm{Mies}\ead[label=e3]{f.mies@tudelft.nl}}
			
			\address[A]{Department of Information Systems, Business Statistics and Operations Management,
				The Hong Kong University of Science and Technology,
				\printead{e1}}
			
			\address[B]{Institute of Epidemiology, Helmholtz Munich,
				\printead{e2}}
			
				\address[C]{Delft University of Technology,
				\printead{e3}}
		\end{aug}

\begin{abstract}
Consider the sum $Y=B+B(H)$  of a Brownian motion $B$ and an independent fractional Brownian motion $B(H)$ with Hurst parameter $H\in(0,1)$. Even though $B(H)$ is not a semimartingale, it was shown in [\textit{Bernoulli} \textbf{7} (2001) 913--934]  that $Y$ is a semimartingale if $H>3/4$.  Moreover,  $Y$ is locally equivalent to $B$  in this case, so  $H$ cannot be consistently estimated from local observations of $Y$. This paper pivots on another unexpected feature in this model: if $B$ and $B(H)$ become correlated, then  $Y$ will never be a semimartingale, and $H$ can be identified, regardless of its value. This and other results will follow from a detailed statistical analysis of a  more general class of processes called \emph{mixed semimartingales}, which  are semiparametric extensions of $Y$ with stochastic  volatility in both the martingale and the fractional component. In particular, we derive consistent estimators and feasible central limit theorems  for all   parameters and processes that can be identified from high-frequency observations. We further  show that our estimators achieve optimal   rates   in a minimax sense. 
\end{abstract}

	 
	 
 \end{frontmatter}

\section{Introduction}\label{intro}

Mixed fractional Brownian motions (mfBms) were introduced by \cite{Cheridito01} as the sum
\beq\label{eq:mfBm} Y_t=\si B_t +\rho B(H)_t,\qquad t\geq0,\eeq
of a standard Brownian motion $B$ and an independent fractional Brownian motion $B(H)$ with Hurst parameter $H\in(0,1)\setminus\{\frac12\}$. While this class of processes was originally introduced in mathematical finance to model long memory in asset prices, it poses nonstandard challenges from a statistical perspective: Even though the laws of $B$ and $B(H)$ on a finite time interval are mutually singular, the law of their superposition, $Y$, can be locally equivalent to that of  $\si B$ (if $H>\frac34$) or  $\rho B(H)$ (if $H<\frac14$); see \cite{Cai16,Cheridito01,vanZanten07}. In these cases, either $(\rho,H)$ or $\si$ cannot be consistently estimated on a finite time interval.

In this paper, we are interested in whether these results remain valid if $B$ and $B(H)$ are no longer assumed to be independent. The  answer to this question is  negative.
\begin{Theorem}\label{thm:intro}
	Suppose that $Y$ is given by \eqref{eq:mfBm} with  $\si,\rho>0$ and
	\beq\label{eq:fBm} B(H)_t= K_H^{-1}\int_{-\infty}^t ((t-s)^{H-\frac12}-(-s)_+^{H-\frac12})\,\dd \wt B_s,\qquad t\geq0,\eeq
	is the Mandelbrot--van Ness representation of standard fractional Brownian motion (see e.g., \cite[Chapter 1.3]{Mishura08}),
	where $(B,\wt B)$ is a two-dimensional Brownian motion with $\var(B_1)=\var(\wt B_1)=1$ and $\lambda=\cov(B_1,\wt B_1)\in[0,1]$ and $K_H$ is the normalizing constant given in \eqref{const:C:H}.
	Then the process $Y$ has the same distribution as $\widetilde{Y}$ given by
	\beq\label{eq:3proc} \widetilde{Y}_t = \si W_t +  \sqrt{\frac{2\lambda\rho\sigma}{K_H(H+\frac12)}} W({\ov H})_t+ \rho W(H)_t, \eeq
	where $W$ is a standard Brownian motion, $W(H)$ and $W(\ov H)$ are fractional Brownian motions with  Hurst parameters $H$ and $\ov H=\frac12(H+\frac12)$, respectively, and all three processes are independent. Moreover, if $\la>0$, $Y$ is not a semimartingale and its distribution is locally singular to  both $\si B$ and $\rho B(H)$ for all $H\in(0,1)\setminus\{\frac12\}$. 
\end{Theorem}

The newly emerging fBm with Hurst parameter $\overline{H}$ significantly changes the properties of the model. In particular, for all values of $H$, we have $|\overline{H}-\frac12| < \frac14$, which has two  consequences for the correlated case: first, the mixed process will never be locally equivalent to either of the two pure processes, and second,  $\overline{H}$ (and therefore $H$) is identifiable from high-frequency observations in all cases. Our asymptotic results below show that this also upholds for negative correlation coefficients $\lambda$.

Fractional processes have a long history of applications in fields such as hydrology \cite{hurst1951long,MLS97,pandey1998multifractal}, telecommunications \cite{LTWW94,MRR02}, finance \cite{Cheridito01,CR98,Gatheral18}, turbulence \cite{chen2022preaveraging,Corcuera13} among others.
In these applications, superpositions of fractional processes arise naturally when multiple sources have a cumulative effect. 
For example, \cite{williams2004error} describe a continuous GPS signal affected by both white noise and fractional noise. 
The same phenomenon is found by \cite{xu2017detecting} in a range of astronomical data sets.
Another example is found in hydrology, where fractional Brownian motion is commonly used to model river runoff, with varying Hurst parameters for different rivers \cite{pandey1998multifractal}.
In a system of multiple connected rivers, the runoffs add up downstream, which leads to a superposition of fBms with different Hurst parameters.
Due to the spatial correlation of rainfall, these constituent fBms will be correlated, analogous to the mixed fBm model studied in this paper. 
The results of this paper highlight that correlation between the fBms is not negligible, as it alters the statistical properties significantly. 

Our interest in mixed fractional Brownian motion of the specific form \eqref{eq:mfBm} is motivated by recent applications of mixed processes in financial econometrics. In \cite{Chong21}, it is shown that for a large set of high-frequency stock return data, observed prices are contaminated by microstructure noise that locally resembles fractional Brownian motion with Hurst parameter $H\in(0,\frac12)$. 
Like in most microstructure noise models in the literature, the two innovation processes driving price and noise (which are $B$ and $B(H)$ in \eqref{eq:mfBm}) are assumed to be independent of each other. 
However, both economic theory \cite{Diebold13} and empirical evidence  \cite{Hansen06} suggest that efficient price and microstructure noise should be contemporaneously cross-correlated.

Against this background, the main contribution of this paper is to develop an infill asymptotic theory for semiparametric  extensions of the mfBm model (called \emph{mixed semimartingales}) where   $\si$ and $\rho$ can be stochastic and time-varying. The fundamental  statistical question we address is the following: what parameters can be inferred from local observations of the sum of a martingale  and a correlated fractional process,  and what are the optimal rates of convergence as the sample size increases? 
To this end,  we first derive in Section~\ref{sec:CLT} a stable central limit theorem (CLT) for the empirical autocovariances of the increments of a mixed semimartingale process as the sampling frequency increases to infinity (Theorem~\ref{cor}).
In line with Theorem~\ref{thm:intro}, the population autocovariance consists of three terms with different scaling exponents. 
Most importantly, the leading order term only contains information about the parameter $\sigma$ (if $H>\frac{1}{2}$) or $(\rho,H)$ (if $H<\frac{1}{2})$.
To optimally estimate all parameters, it is thus necessary to utilize the information in the asymptotically smaller contributions to the autocovariance.
In particular, in Section~\ref{sec:est}, we combine the results of Theorem~\ref{cor} with an optimal generalized methods of moments (GMM) procedure to construct consistent and asymptotically mixed normal estimators for all identifiable model parameters in Theorem~\ref{thm:est}. Interestingly, the rates of convergence are non-standard and depend on the parameter value $H$. 
This phenomenon can be traced back to the utilization of autocovariance information of smaller asymptotic order than the leading term.
To make the GMM method feasible, we exhibit in Corollary~\ref{cor:var} consistent estimators of the asymptotic (co-)variances. In Section~\ref{sec:lb}, we return to the parametric setting of mfBm and show in Theorem~\ref{thm:lb} that the rates of our estimators in Theorem \ref{thm:est} are optimal in the minimax sense. Section~\ref{sec:sim} presents Monte Carlo evidence for the estimators from Theorem~\ref{thm:est}. 
Sections~\ref{sec:6} and \ref{app:lb} contain the proofs of our main results, except for Theorem~\ref{cor}, which is proved, in a  more general setting, in Appendices~\ref{app:CLT}--\ref{app:B}. Appendix~\ref{app:sim} contains additional simulation results.

Since our  principal aim is to analyze the impact of cross-correlation, we do not include jumps \cite{AitSahalia09,Jacod14},  irregular observation times \cite{BN05,Chen20,Jacod17} or rounding errors \cite{Delattre97,Li15, Rosenbaum09} in our analysis. 
To simplify the exposition, we further refrain from studying mixed fractional models with more than two components.
In what follows, $C$ denotes a constant in $(0,\infty)$, whose value 
may change from line to line. We also write $A\lec B$ if $A\leq CB$.

\section{Central limit theorem for sample autocovariances}\label{sec:CLT}

On a filtered   space $(\Om, \mathcal{F}, \bbf=(\mathcal{F}_t)_{t \geq 0}, \mathbb{P})$ satisfying the usual conditions, consider a   \emph{mixed semimartingale}
 \begin{equation} \label{mix:SM:mod-2}
 	Y_t = Y_0+\int_0^t a_s\,\dd s + \int_{0}^{t} \si_s \, \dd B_s + \int_{0}^{t} g(t-s) \rho_s \, \dd B_s + \int_{0}^{t} g(t-s) \rho'_s \, \dd B'_s,
 \end{equation}
 where $B$ and $B'$ are two independent one-dimensional standard $\F$-Brownian motions and $a$, $\si$, $\rho$ and $\rho'$ are one-dimensional predictable processes.
Moreover, we assume that  $g \colon \bbr \lra \bbr$ is a kernel of the form
\begin{equation} \label{kernel:g}
g(t) = K_H^{-1} t^{H - \frac{1}{2}}\bone_{\{t>0\}} + g_0(t),\qquad t\in\R,
\end{equation}
where  $H\in(0,1)\setminus\{\frac12\}$,
\begin{equation} \label{const:C:H}
K_H = \sqrt{\frac{1}{2H} + \int_{1}^{\infty}  \bigl( r^{H -\frac{1}{2}} - (r-1)^{H -\frac{1}{2}}  \bigr)^2 \, \dd r}=\frac{\Ga(H+\frac12)}{\sqrt{\sin(\pi H)\Ga(2H+1)}},
\end{equation}
and $g_0 \colon \bbr \lra \bbr$ is a continuously differentiable function with $g_0(x) = 0$ for all $x\leq 0$. By \eqref{kernel:g}, the kernel $g$ behaves as a power-law kernel around $0$, but due to the addition of $g_0$, the behavior of $g$ outside of $0$ is not further specified. In particular, because $g$ is not specified at infinity, the increments of $Y$ may have long or short memory, irrespective of the value of $H$.

Let $\Delta_n$ be a small time step such that $\Delta_n \lra 0$ and define  
\begin{equation} \label{not:tensor-0}
 \Delta_i^n  Y =  Y_{i \Delta_n}-  Y_{(i-1) \Delta_n} 
\end{equation}
(and similarly for other processes). Our goal is to prove a CLT as $\Den\to0$ for the (normalized) autocovariances of the increments of $Y$, given by
\begin{equation*} 
	V^{n}_{r, t}  =  
	\Den^{1-2H} \sum_{i= 1}^{[t/\Den] - r} 
	\Delta_i^n Y\Delta_{i+r}^n Y,\qquad
	\wh{V}^{n}_{r, t}  =
	\sum_{i = 1}^{[t/\Den] - r} 
	\Delta_i^n Y \Delta_{i+r}^n Y
\end{equation*}
for $r\in\N_0$. This will then be used in Section~\ref{sec:est} 
to
derive feasible estimators of $H$ and 
$$ C_T=\int_0^T c_s\,\dd s = \int_0^T \si_s^2\,\dd s,\quad \La_T=\int_0^T\la_s\,\dd s = \int_0^T\si_s\rho_s\,\dd s ,\quad \Pi_T=\int_0^T (\rho_s^2+\rho^{\prime 2}_s)\,\dd s,$$
whenever possible. We impose the following structural assumptions:

\settheoremtag{(CLT)}
\begin{Assumption} \label{Ass:B0} Consider the process $Y$ from \eqref{mix:SM:mod-2} and let
	\beq\label{eq:NH} N(H)= [1/\lvert 2H-1\rvert]. 
	\eeq
	\begin{enumerate}
		\item  The kernel $g$ is of the form \eqref{kernel:g} where $H\in (0,1)\setminus\{\frac12\}$ and $g_0\in C^1(\bbr)$ satisfies $g_0(x)=0$ for all $x\leq 0$. 
		\item  The drift process $a$ is   locally bounded, $\F$-adapted and càdlàg. Moreover, $B$ is a  standard $\F$-Brownian motion.
		\item If $H>\frac12$, the volatility process $\si$ takes the form
		\begin{equation} \label{repr:si-0}
			\si_t = \si^{(0)}_t + \int_{0}^{t} \wti{\si}_s \, \dd B_s+\int_{0}^{t} \wti{\si}'_s \, \dd B'_s+\int_{0}^{t} \wti{\si}''_s \, \dd B''_s, \qquad t \geq 0,
		\end{equation}
		where
		\begin{enumerate}
			\item $\si^{(0)}$ is an $\F$-adapted locally bounded  process such that for all $T>0$, there are  $\ga \in  ( \frac{1}{2}, 1  ]$ and $K_1\in(0,\infty)$ with
			\begin{equation} \label{mom:ass:si:rho:1-0}
				\bbe  [ 1\wedge  \vert \si_t^{(0)} - \si_s^{(0)}   \vert ]  \leq K_1 \vert t - s \vert^{\ga},\qquad s,t\in[0,T];
			\end{equation}
			\item  $\wti{\si}$, $\wti{\si}'$ and $\wti{\si}''$ are   $\F$-adapted locally bounded   processes such that for all $T>0$, there are  $\eps \in(0,1)$ and $K_2\in(0,\infty)$ with
			\begin{equation} \label{reg:ass:si:ti-0}
				\bbe  [ 1\wedge \vert \wti{\si}_t  - \wti{\si}_s  \vert    ]  \leq K_2 \vert t -s \vert^{\eps}, \qquad s,t\in[0,T],
			\end{equation}
		and an analogous bound for $\wti{\si}'$ and $\wti{\si}''$.
			\item $B''$ is a standard $\F$-Brownian motion that is independent of $(B,B')$.
		\end{enumerate}	
		If $H<\frac12$, we have \eqref{repr:si-0} but with $\si$, $\si^{(0)}$, $\wt \si$, $\wti{\si}'$ and $\wti{\si}''$ replaced by  $\rho$ and some processes $\rho^{(0)}$, $\wt\rho$, $\wt\rho'$ and $\wt\rho''$ satisfying   conditions analogous to \eqref{mom:ass:si:rho:1-0} and \eqref{reg:ass:si:ti-0}.
		\item  If $H>\frac12$, the process $\rho$ is   $\F$-adapted and locally bounded. Moreover, for all $T>0$, there is $K_3\in(0,\infty)$ such that
		\begin{equation} \label{reg:cond:rho:B3-0}
			\bbe  [   1\wedge \vert \rho_t - \rho_s   \vert  ]  \leq K_3 \vert t -s \vert^{\frac{1}{2}}, \qquad s,t\in[0,T].
		\end{equation}
		If $H<\frac12$, we have the same condition but with $\rho$ replaced by $\si$.
	\end{enumerate}
\end{Assumption}

These assumptions are fairly standard in high-frequency statistics (cf.\ \cite{AitSahalia14,Jacod12}). As usual, the most restrictive assumptions are \eqref{repr:si-0} and \eqref{reg:cond:rho:B3-0}, which essentially demand that the volatility processes $\si$ and $\rho$ be no rougher than a continuous It\^o semimartingale.   In particular, they do not cover the case of rough volatility (see \cite{Gatheral18}). However,  we conjecture that the CLTs do remain valid even in the presence of rough volatility. This is due to the special structure of quadratic functions, which has been exploited for instance in \cite{CHLRS24} in a slightly different context.

For the following theorem, which  is the main result of this section, we use the notation
\begin{equation} \label{num:Ga}
	\Ga^H_0 = 1 \quad \textrm{and} \quad \Ga^H_r = \frac{1}{2} \Bigl( (r + 1)^{2H} - 2 r^{2H} + (r-1)^{2H} \Bigr), \quad r \geq 1,
\end{equation}
and
\begin{equation} \label{nu:La}
	\Phi^H_0 =  \frac{2K_H^{-1}}{H+\frac12}, \quad
	\Phi^H_r =  \frac{K_H^{-1}}{H+\frac12} \Bigl( (r + 1)^{H+\frac12} - 2 r^{H+\frac12} + (r - 1)^{H+\frac12} \Bigr), \quad r \geq 1.
\end{equation}
Note that $(\Ga^H_r)_{r\geq0}$ is the autocovariance function of fractional Gaussian noise (i.e., of $(B(H)_{n+1}-B(H)_n)_{n\in\N}$) and $\Phi^H_r=\Phi^H_0 \Ga^{\ov H}_r$.
 We also use $\stackrel{\mathrm{st}}{\Longrightarrow}$ to denote functional stable convergence in law in the space of càdlàg functions  equipped with the local uniform topology.

\begin{Theorem}\label{cor}
	Suppose that   Assumption~\ref{Ass:B0} holds  and 	let $  V^n_t = (  V^{n}_{0, t},  \ldots,  V^{n}_{r-1, t})^T$ and $\wh V^n_t = (\wh V^{n}_{0, t},  \ldots,\wh V^{n}_{r-1, t})^T$ for some fixed    $r\in\N$. Further  define $\Ga^H = (\Ga_0^H,\ldots,\Ga_{r-1}^H)^T$, $\Phi^H=(\Phi_0^H,\ldots,\Phi_{r-1}^H)^T$ and $e_1 = (1,0,\ldots,0)^T \in \bbr^{r}$.
	\begin{enumerate}
		\item If $H>\frac12$, then
		\begin{equation} \label{CLT:var:func:lags}
			\Delta_n^{- \frac{1}{2}} \bigg\{ \wh V^{n}_{t} - e_1 C_t- \Den^{H-\frac12} \Phi^H\La_t- \Delta_n^{2H-1}\Ga^H\Pi_t   \bigg\}
			\stackrel{\mathrm{st}}{\Longrightarrow} \mathcal{Z}',
		\end{equation}
		where  
		$\mathcal{Z}' = (\mathcal{Z}'_t)_{t \geq 0}$ is an $\bbr^r$-valued process, defined on a very good filtered extension $(\ov{\Om}, \ov{\mathcal{F}}, (\ov{\mathcal{F}}_t)_{t \geq 0}, \ov{\mathbb{P}})$ of the original probability space $(\Om, \mathcal{F}, (\mathcal{F}_t)_{t \geq 0}, \mathbb{P})$ (see \cite[Chapter~2.1.4]{Jacod12}), that conditionally on  $\mathcal{F}$ is a centered Gaussian process with independent increments and  covariance process   $\calc'_t = (\calc^{\prime ij}_t)_{i,j=0}^{r-1}$   given by
		\begin{equation} \label{eq:cov}
			\calc^{\prime ij}_t=\ov{\bbe}[\calz^{\prime i}_t \calz^{\prime j}_t \mid \calf]  = 2^{\bone_{\{i=0\}}} \int_{0}^{t} \si_s^4 \, \dd s \,\bone_{\{i=j\}}.  
		\end{equation}	
		\item If $H<\frac12$, then
		\begin{align} \label{CLT:var:func:lags-2}
			\begin{split}
				\quad \Delta_n^{- \frac{1}{2}} &\bigg\{V^{n}_{t} - \Ga^H \Pi_t -\Den^{\frac12-H}\Phi^H\La_t- e_1\Delta_n^{1-2H} C_t \bigg\} \\
				=
				\Delta_n^{\frac{1}{2}-2H}& \bigg\{\wh V^{n}_{t} - e_1 C_t- \Den^{H-\frac12} \Phi^H\La_t- \Delta_n^{2H-1}\Ga^H\Pi_t \bigg\}
				\stackrel{\mathrm{st}}{\Longrightarrow} \mathcal{Z},
			\end{split}
		\end{align}
		where  
		$\mathcal{Z} = (\mathcal{Z}_t)_{t \geq 0}$ is an $\bbr^r$-valued process  defined on $(\ov{\Om}, \ov{\mathcal{F}}, (\ov{\mathcal{F}}_t)_{t \geq 0}, \ov{\mathbb{P}})$ that conditionally on  $\mathcal{F}$ is a centered Gaussian process with independent increments and  covariance process   $\calc_t = (\calc^{ij}_t)_{i,j=0}^{r-1}$   given by
		\begin{equation} \label{eq:cov-2}
			\calc^{ij}_t = \ov{\bbe}[\calz^{\prime i}_t \calz^{\prime j}_t \mid \calf]=\calc^{ij} \int_{0}^{t} (\rho^2_s+\rho^{\prime 2}_s)^2 \, \dd s,\qquad \calc^{ij} = v^{H,0}_{ij}+\sum_{k=1}^\infty (v^{H,k}_{ij}+v^{H,k}_{ij}),
		\end{equation}	
		with
		\beq \label{covMat:lags}v^{H,k}_{ij}=	\Ga^H_k\Ga^H_{\lvert i-j+k\rvert} + \Ga^H_{\lvert k-j\rvert}\Ga^H_{k+i}.\eeq
	\end{enumerate}
\end{Theorem}

Theorem~\ref{cor} is a special case of  Theorems~\ref{thm:CLT:mixedSM} and \ref{thm:CLT2}, which are stated and proved in the Appendix.
Note that if $H>\frac34$, the term $\Den^{2H-1}\Ga^H \Pi_t$ in \eqref{CLT:var:func:lags} is dominated by the Gaussian fluctuation process $\Den^{1/2}\calz'$ and can thus be omitted in this case. The same comment applies to \eqref{CLT:var:func:lags-2}, where the term $e_1 C_t$ can be dropped if $H<\frac14$.

\section{Semiparametric estimation of mixed semimartingales}\label{sec:est}

In order to  construct rate-optimal estimators of $\Theta_T=(H,C_T,\La_T,\Pi_T)$, we combine Theorem~\ref{cor} with a GMM approach. The idea is to choose $r$ lags and, at stage $n$, a symmetric positive weight matrix $\calw_n \in \R^{r\times r}$ (which can be a random statistic) and  to obtain an estimator of $\Theta_T$ by solving the minimization problem
\begin{equation}\label{eq:optprob} \begin{split}
 \argmin_{\theta=(H,C,\La,\Pi)} \Bigl\lVert \mathcal{W}_n^{1/2} \Bigl(\wh V^{n}_{T} - e_1 C- \Den^{H-\frac12} \Phi^H\La- \Delta_n^{2H-1}\Ga^H\Pi\Bigr) \Bigr\rVert_2^2,\end{split}
\end{equation}
where 
$\lVert \cdot\rVert_2$ is the Euclidean norm. 
More precisely, we construct an estimator $	\wh\Theta^n_T=(\wh H^n,\wh C^n_T,\wh \La^n_T,\wh\Pi^n_T)$ of $\Theta_T$ by solving the \emph{estimating equation}
\begin{equation}\label{eq:esteq} 
	F_n(\theta)=0,
\end{equation}
where 
$
	F_n(\theta) = \nabla_\theta \lVert \mathcal{W}_n^{ 1/2}(\widehat{V}^n_T-\mu_n(\theta))\rVert^2_2 = -2D_\theta \mu_n(\theta)^T \mathcal{W}_n  ( \widehat{V}^n_T - \mu_n(\theta)  )   \in\R^4
$
and
$\mu_n(\theta)=\mu_n(H,C,\La,\Pi)=e_1 C + \Delta_n^{H-{1}/{2}} \Phi^H\Lambda +\Delta_n^{2H-1} \Gamma^H\Pi \in \R^r$.
Using the 
theory of estimating equations (see \cite{JS18,MP22}), we can derive the asymptotic properties of this estimator.

\begin{Theorem}\label{thm:est}
	Suppose that conditions 2--5 in Assumption~\ref{Ass:B} hold for the processes specified in \eqref{eq:spec}. Further assume that  $r\geq5$ and that $\calw$ and $\calw_n$ are (possibly random) symmetric  positive definite matrices in $\R^{r\times r}$ such that $\calw_n\limp \calw$. 
	If $H<\frac12$, suppose that $\Pi_T>0$ almost surely; if $H>\frac12$, suppose that $\La_T\neq0$ almost surely.
	\begin{enumerate}
		\item If $H\in(\frac14,\frac12)$, there exists a sequence $	\wh\Theta^n_T=(\wh H^n,\wh C^n_T,\wh \La^n_T,\wh\Pi^n_T)$ of estimators of $\Theta_T=(H,C_T,\La_T,\Pi_T)$ such that $\mathbb{P}(F_n(\wh\Theta^n_T)=0)\limp1$ and  
		\begin{equation}\label{eq:CLT-est-1} 
			D_n^{-1} (\wh\Theta^n_T-\Theta_T) 	\stackrel{\mathrm{st}}{\longrightarrow}  E^{-1} (\partial_H \Gamma^H \Pi_T,e_1, \Phi^H,\Ga^H)^T \mathcal{W}\calz_T,
		\end{equation}
	where
	\begin{equation}\label{eq:ABCW-1} 
		\begin{split}
			D_n &= \begin{pmatrix} \Delta_n^{\frac{1}{2}}&0&0&0 \\ 0&\Delta_n^{2H-\frac{1}{2}} &0&0 \\ 0&0&\Delta_n^{H}&0 \\ 2\Delta_n^{\frac{1}{2}}\log\Delta_n^{-1}\Pi_T&0&0& \Delta_n^{\frac{1}{2}} \end{pmatrix},\\
			E&= (\partial_H\Gamma^H \Pi_T,e_1 , \Phi^H , \Gamma^H )^T \mathcal{W} (\partial_H\Gamma^H \Pi_T ,e_1 , \Phi^H, \Gamma^H) \in\R^{4\times 4},
		\end{split}
	\end{equation}
	and $\calz$ is the same process as in Theorem~\ref{cor}.  The matrix $E$ in the last display is regular.  Moreover, the sequence $\wh\Theta^n_T$ is \emph{locally unique} in the sense that any other sequence $\wt\Theta^n_T$ of estimators such that $\mathbb{P}(F_n(\wt\Theta^n_T)=0)\limp1$ and $\mathbb{P}(\lVert \wt\Theta^n_T-\Theta_T\rVert_2\leq 1/(\log\Den^{-1})^2)\limp1$   satisfies $\mathbb{P}(\wt\Theta^n_T=\wh \Theta^n_T)\limp1$.
	\item If $H\in(0,\frac14)$, define $F^{(1)}_n(\theta^{(1)})= \nabla_\theta \lVert \mathcal{W}_n^{ 1/2}(\widehat{V}^n_T-\mu^{(1)}_n(\theta^{(1)}))\rVert^2_2$, where $\mu^{(1)}_n(\theta^{(1)})=\mu_n^{(1)}(H,\La,\Pi)= \Delta_n^{H-{1}/{2}} \Phi^H\Lambda +\Delta_n^{2H-1} \Gamma^H\Pi$. Then there is  a locally unique sequence $	\wh\Theta^{n, (1)}_T=(\wh H^{n,(1)},\wh \La^{n,(1)}_T,\wh\Pi^{n,(1)}_T)$ of estimators of $\Theta^{(1)}_T=(H,\La_T,\Pi_T)$  with the property   $\mathbb{P}(F^{(1)}_n(\wh\Theta^{n,(1)}_T)=0)\limp1$   such that \eqref{eq:CLT-est-1} remains valid for $\wh\Theta^{n, (1)}_T- \Theta^{(1)}_T$   with the second row and column (out of four) deleted from all vectors and matrices appearing in \eqref{eq:CLT-est-1} and \eqref{eq:ABCW-1}. 
	\item If $H\in(\frac12,\frac34)$,  Part 1 of the theorem remains true if \eqref{eq:CLT-est-1} and \eqref{eq:ABCW-1} are replaced by
	\begin{equation}\label{eq:CLT-est-2} 
			D_n^{-1} (\wh\Theta^n_T-\Theta_T) 	\stackrel{\mathrm{st}}{\longrightarrow} E^{-1}   (\partial_H \Phi^H \Lambda_T,e_1,\Phi^H, \Ga^H)^T  \mathcal{W}\calz'_T,
	\end{equation}
and
	\begin{equation}\label{eq:ABCW-2}		\begin{split}
			D_n &= \begin{pmatrix} \Delta_n^{1-H}&0&0&0 \\ 0&\Delta_n^{\frac{1}{2}} &0&0 \\ \Delta_n^{1-H} \log(\Delta_n^{-1}) \Lambda_T &0&\Delta_n^{1-H}&0 \\ 0&0&0& \Delta_n^{\frac{3}{2}-2H} \end{pmatrix},\\
			E&=  (\partial_H\Phi^H  \Lambda_T,e_1 ,  \Phi^H ,\Gamma^H )^T \mathcal{W} (\partial_H\Phi^H \Lambda_T, e_1,  \Phi^H , \Gamma^H ) \in\R^{4\times 4},
		\end{split}
\end{equation}
respectively, 	and $\calz'$ is the same process as in Theorem~\ref{cor}.
\item If $H\in(\frac34,1)$,  define $F^{(2)}_n(\theta^{(2)})= \nabla_\theta \lVert \mathcal{W}_n^{1/2}(\widehat{V}^n_T-\mu^{(2)}_n(\theta^{(2)}))\rVert^2_2$, where $\mu^{(2)}_n(\theta^{(2)})=\mu_n^{(2)}(H,C,\La)= e_1C+\Delta_n^{H-{1}/{2}} \Phi^H\Lambda$. Then there is  a locally unique sequence $	\wh\Theta^{n, (2)}_T=(\wh H^{n,(2)},\wh C^{n,(2)}_T,\wh\La^{n,(2)}_T)$ of estimators of $\Theta^{(2)}_T=(H,C_T,\La_T)$ satisfying $\mathbb{P}(F^{(2)}_n(\wh\Theta^{n,(2)}_T)=0)\limp1$   such that \eqref{eq:CLT-est-2}  remains valid for $\wh\Theta^{n, (2)}_T- \Theta^{(2)}_T$ with the last row and column (out of four) deleted from all vectors and matrices appearing in  \eqref{eq:CLT-est-2} and \eqref{eq:ABCW-2}. 
\item In the setting of Part 2 (resp., Part 4), if a sequence of estimators $\wh\Theta^n_T=(\wh H^n,\wh C^n_T,\wh \La^n_T,\wh\Pi^n_T)$ satisfies $\mathbb{P}(F_n(\wh\Theta^n_T)=0)\to1$, then the weak convergence statements in Part 2 (resp., Part 4) remain valid with $\ov \Theta^n_T=(\wh H^n,\wh\La^n_T,\wh\Pi^n_T)$ instead of $\wh \Theta^{n,(1)}_T$ (resp., $\ov \Theta^n_T=(\wh H^n,\wh C^n_T,\wh\La^n_T)$ instead of $\wh \Theta^{n,(2)}_T$).
\end{enumerate}
\end{Theorem}

The proof will be given in Section~\ref{sec:6}. Note that a nondiagonal rate matrix also occurs in similar situations where a self-similarity parameter is estimated; see \cite{brouste2018local,chigansky2022a,MP22}, for example. 

In the case $H\in (\frac14, \frac34)$, all parameters of the model are identifiable, and Parts 1 and 3 of Theorem \ref{thm:est} describe how the exact value of $H$ affects the asymptotic behavior of the estimators. 
In the case $H\notin (\frac14, \frac34)$, the model is only partially identifiable:
if $H<\frac14$ (Part 2), we cannot consistently estimate $C_T$, while if $H>\frac34$ (Part 4), we cannot consistently estimate $\Pi_T$. 
Parts 2 and 4 of Theorem \ref{thm:est} state that in these partially identifiable cases, one may obtain asymptotically normal estimators by reducing the GMM equations to only include identifiable parameters.
However, these estimators are infeasible if the regime of $H$ is not known.
Fortunately, by Part 5, the feasible GMM estimator \eqref{eq:optprob}
can still be employed in the regime $H\notin(\frac14, \frac34)$ to derive asymptotically normal estimators for all identifiable parameters.

\begin{Remark}\label{rem:ident}
 The assumption $\La_T\neq0$ is important in Theorem~\ref{thm:est} if $H>\frac12$ to ensure identifiability of all parameters. For example, in the case of an mfBm as in \eqref{eq:mfBm}, if $\la=0$, then by \cite{vanZanten07} there is no way to consistently estimate $H$ and $\rho$ if $H>\frac34$. Moreover, by Theorem~\ref{thm:intro}, it will not be possible to asymptotically distinguish the model $(H,\si,\lambda,\rho)=(H_0,\si_0,\lambda_0,\rho_0)\in (\frac34,1)\times(0,\infty)\times(0,1]\times(0,\infty)$ from the model
		$(H,\si,\lambda,\rho)=(H_1,\si_1,0,\rho_1)$ if $H_1=\frac12(\frac12+H_0)$, $\si_1=\si_0$ and $\rho_1=(2\la_0\rho_0\si_0/(K_{H_0}(H_0+\frac12)))^{1/2}$. To see this, note that the process $Y$ in the model $(H_0,\si_0,\la_0,\rho_0)$ has the same law as $\si_0 W+\rho_1W(H_1)+\rho_0W(H_0)$ by Theorem~\ref{thm:intro}. Moreover, by \cite{vanZanten07}, the laws of $\si_0W+\rho_0W(H_0)$  and $\si_0W$ are locally equivalent, so we deduce that the laws of  $\si_0 W+\rho_0W(H_0)+\rho_1W(H_1)$  and $\si_0 W+\rho_1W(H_1)$ are equivalent by noting that convolution of measures preserves equivalence. However, since $\si$ remains the same in both models, this local equivalence has no consequence for  estimating $\si$, which  often (e.g., in econometrics)   is the main parameter of   interest. We also note that if $\frac12<H<\frac34$ and $\La_T=0$, then the rate of convergence for estimating $H$ will be slower (equal to $\Delta_n^{3/2-2H}$, see \cite{Dozzi15}), as one can no longer rely on the fictitious fBm for inferring $H$.
\end{Remark}

In order to make the CLTs of Theorem~\ref{thm:est} feasible, we adapt the results of \cite{LX16} to construct consistent estimators of the involved asymptotic covariance matrices. This further allows us to  choose an optimal weight matrix $\calw_n$.  We choose two integer sequences $k_n$ and $\ell_n$ and define
\begin{equation}\label{eq:acov} \begin{split}
		\wh \Sig_n &= \wh\Sig^{(0)}_n + \sum_{\ell=1}^{\ell_n} w(\ell,\ell_n)(\wh \Sig^{(\ell)}_n + (\wh\Sig^{(\ell)}_n)^T),  \\
		\wh\Sig^{(\ell)}_n &= \Den\sum_{i=\ell+1}^{[T/\Den]-r+1} \psi^{(i)}(\psi^{(i-\ell)})^T \in\R^{ r \times r},\qquad \psi^{(i)}=(\psi^{(i)}_0,\dots,\psi^{(i)}_{r-1})^T,\\
		\psi^{(i)}_j &= \Delta^n_i Y \Delta^n_{i+j}Y - \wh m^{n,j}_i,\qquad \wh m^{n,j}_i=\frac1{k_n}\sum_{k=0}^{k_n-1} \Delta^n_{i+k} Y\Delta^n_{i+k+j}Y,\\
		\wh\eta_n& = (\Den^{2\wh H^n}\wh\Pi^n_T (\partial_H \Ga^{\wh H^n}-2(\log \Den^{-1}) \Ga^{\wh H^n}) + \Den^{\wh H^n+1/2}\wh \La^n_T (\partial_H \Phi^{\wh H^n} - (\log \Den^{-1})\Phi^{\wh H^n}),\\
		&\qquad\qquad\Den e_1,\Den^{\wh H^n+1/2}\Phi^{\wh H^n},\Den^{2\wh H^n}\Ga^{\wh H^n}) \in \R^{r\times 4}
	\end{split}
\end{equation}
for some deterministic weight function $w$.

\begin{Corollary}\label{cor:var} Assume the conditions of (any part of) Theorem~\ref{thm:est} and suppose that $k_n,\ell_n\to\infty$ with $\ell_n/\sqrt{k_n}\to0$ and $\ell_n\sqrt{k_n\Den}\to0$. Further assume that $w$ is uniformly bounded and satisfies $w(\ell,\ell_n)\to1$ for every $\ell\geq1$.
 If we denote the diagonal entries of 
	\begin{equation}\label{eq:avar} 
	\mathcal{V}_n=\Den	(\wh \eta_n^T  {  \mathcal{W}}_n \wh \eta_n)^{-1}\wh\eta_n^T  { \mathcal{W}}_n\wh\Sig_n  { \mathcal{W}}_n \wh\eta_n(\wh \eta_n^T  {  \mathcal{W}}_n \wh \eta_n)^{-1} \in \R^{4\times 4}
\end{equation}
by $\mathcal{V}_n^H$, $\mathcal{V}_n^C$, $\calv_n^\La$ and  $\mathcal{V}_n^\Pi$, then  asymptotic $\ga$-confidence intervals for $H$, $C_T$ (if $H>\frac14$), $\La_T$ and $\Pi_T$ (if $H<\frac34$) for  $\ga\in(0,1)$ are given by
\begin{align*} 
	&[\wh H^n \pm \Phi^{-1}((1-\ga)/2)\sqrt{\mathbb{V}_n^H}],& [\wh C^n_T \pm \Phi^{-1}((1-\ga)/2)\sqrt{\mathbb{V}_n^C}],\\
	& [\wh \La^n_T \pm \Phi^{-1}((1-\ga)/2)\sqrt{\mathbb{V}_n^\La}], &[\wh \Pi^n_T \pm \Phi^{-1}((1-\ga)/2)\sqrt{\mathbb{V}_n^\Pi}],
\end{align*}
respectively, where $\Phi$ is the cumulative distribution function of the standard normal law and $	\wh\Theta^n_T=(\wh H^n,\wh C^n_T,\wh \La^n_T,\wh\Pi^n_T)$ is the solution to \eqref{eq:esteq}.
\end{Corollary}

\begin{Remark}[Optimal GMM]\label{rem:opt}
	If we choose $\calw_n=\wh\Sigma_n^{-1}$ and we have $H\in(\frac14,\frac12)$ (resp., $H\in(\frac12,\frac34)$), then the right-hand side of \eqref{eq:CLT-est-1} (resp., \eqref{eq:CLT-est-2}) has a centered Gaussian distribution with mean $0$ and covariance matrix $E^{-1}$ with $\calw=\calc_T^{-1}$ (resp., $\calw=(\calc'_T)^{-1}$). Analogous  statements hold if $H\in(0,\frac14)$ and $H\in(\frac34,1)$.
\end{Remark}

\section{Statistical lower bounds}\label{sec:lb}

To derive a statistical lower bound, we consider the parametric setup of an mfBm
\begin{equation}
	Y_t = \int_0^t \sigma\, \dd B_s + \rho \int_{-\infty}^t h_H(t,s) \, \dd B_s + \rho'\int_{-\infty}^t h_H(t,s)\, \dd B'_s, \label{eqn:cormixed-fbm}\end{equation}
where $h_H(t,s) =K_H^{-1} [ (t-s)_+^{H-{1}/{2}}  - (-s)_+^{H-{1}/{2}}  ]$, $\sigma>0$, $\rho, \rho'\in\R$, and $B$ and $B'$ are two independent standard Brownian motions. 
Note that this model is a special case of \eqref{mix:SM:mod} but with $a_s=(H-\frac12)K_H^{-1}\int_{-\infty}^0 (s-r)^{H-3/2} (\rho\,\dd B_r+\rho'\,\dd B'_r)$, which is unbounded  as $s\downarrow 0$ if $H<\frac12$. Nevertheless, one can show that any small neighborhood around $0$ only has a negligible impact on the asymptotics in $V^n_t$, so Theorem~\ref{cor} remains valid.

\begin{table}[t]
	\caption{Rates of convergence of the estimators presented in Section \ref{sec:est}.}
	\label{tab:rates}
	{\def\arraystretch{2}\tabcolsep=10pt
	\begin{tabular}{c|c|c}
Parameter		& $H\in(0,\frac{1}{2})$ & $H\in(\frac{1}{2},1)$   \\   \hline 
		$H$  & $\Den^{\frac{1}{2}}$ & $\Den^{1-H}$  \\ 
		$ \sigma^2$   & $\Den^{2H-\frac{1}{2}}$ (if $H>\frac14$) & $\Den^{\frac{1}{2}}$   \\ 
		$\Lambda$ & $\Den^{H}$  & $\Den^{1-H}\lvert \log\Den\rvert$ \\ 
	$\Pi$   & $\Den^{\frac{1}{2}} \lvert \log \Den\rvert$ & $\Den^{\frac{3}{2}-2H}$ (if $H<\frac34$)    
	\end{tabular}
}
\end{table}

The methods presented in Section \ref{sec:est} therefore yield estimators of the parameters $H$, $\sigma^2$, $ \Pi = \rho^2 + \rho'^{2}$ and $\Lambda = \rho\sigma$,
with rates given in Table \ref{tab:rates}.
We show that these rates are optimal, by establishing matching minimax lower bounds.
To this end, consider model \eqref{eqn:cormixed-fbm} with parameter vector $\theta = (H, \sigma^2,\Lambda,\Pi)\subset \Theta$ for the open parameter set 
\begin{align*}
	\Theta = \{ (H,\sigma^2,\Lambda, \Pi) \in\R^4 : H\in(0,1)\setminus\{\tfrac12\}, \, \sigma^2 >0,\, \Lambda\neq 0,\, \Pi >0,\, \La^2< \si^2\Pi \}.
\end{align*}
We use the notation $H(\theta)=\theta_1$, $\si^2(\theta)=\theta_2$ etc.
For $\theta_0\in\Theta$, define the local parameter set 
\begin{align*}
	\mathcal{D}_n(\theta_0) = \{ \theta \in\Theta: \|\theta-\theta_0\|\leq 1/\lvert \log \Den\rvert\}
\end{align*}
and the rate vector 
\begin{align*}
	R_n(\theta_0) =  \begin{cases}
		 ( \Den^{\frac{1}{2}}, \Den^{2H-\frac{1}{2}}, \Den^{H}, \Den^{\frac{1}{2}}/\lvert \log\Den\rvert  ) &\text{if } H(\theta_0)<\frac{1}{2}, \\
		 (\Den^{1-H}, \Den^{\frac{1}{2}}, \Den^{1-H}/\lvert \log\Den\rvert, \Den^{\frac{3}{2}-2H}   ) & \text{if } H(\theta_0)>\frac{1}{2}.
	\end{cases} 
\end{align*}
For the regime $H\in(0,\frac{1}{4})$, the parameter $\sigma^2$ is not identifiable, as evidenced by the deteriorating rates of convergence. 
The same holds for $\Pi$ in the regime $H\in(\frac{3}{4},1)$. 

\begin{Theorem}\label{thm:lb}
	Let $\theta_0\in\Theta$ be such that $\Lambda(\theta_0)\neq 0$.
	If $H(\theta_0)\in(\frac{1}{4}, \frac{3}{4})\setminus\{\frac{1}{2}\}$, there exists some $c>0$ such that
	\begin{equation}
		\limsup_{n\to\infty} \inf_{\widehat{\theta}_n} \sup_{\theta\in\mathcal{D}_n(\theta_0)} \bbp_\theta\Bigl( | (\widehat{\theta}_n-\theta)_k |\geq c\, R_n(\theta_0)_k \Bigr) >0, \quad k=1,2,3,4, \label{eqn:lb-identified}
	\end{equation}
where the infimum is taken among all measurable functions $\wh{\theta}_n$ of $\{Y_{ {i}{\Den}}: i=1,\ldots, [1/\Den]\}$.
	If $H\in(0,\frac{1}{4}]$ (resp., $H\in[\frac{3}{4},1)$), then \eqref{eqn:lb-identified} remains true for $k=1,3,4$ (resp., $k=1,2,3$).
\end{Theorem}

While our current methods do not allow us to determine the sharp value of the constant $c$ in \eqref{eqn:lb-identified}, we conjecture that the GMM estimators from the previous section (if the weight matrix is chosen  as in Remark~\ref{rem:opt}) are close to being  asymptotically efficient. Figure~\ref{fig} plots the value of the constant $\sqrt{C(H)}$ as a function of $H$, where $C(H)$ is the asymptotic variance of the optimal GMM estimator $\wh H^n$ from Theorem~\ref{thm:est} in the mfBm model \eqref{eq:mfBm}, that is, $C(H)$ is the constant that satisfies $\Delta_n^{-1/2}(\wh H^n-H) \limd N(0,C(H)/T)$ if $H<\frac12$ and $\Delta_n^{H-1}(\wh H^n-H) \limd N(0,C(H)(\frac\si{\la\rho})^2/T))$ if $H>\frac12$.
\begin{figure}[htb]
	\centering
	\includegraphics[width=0.9\linewidth]{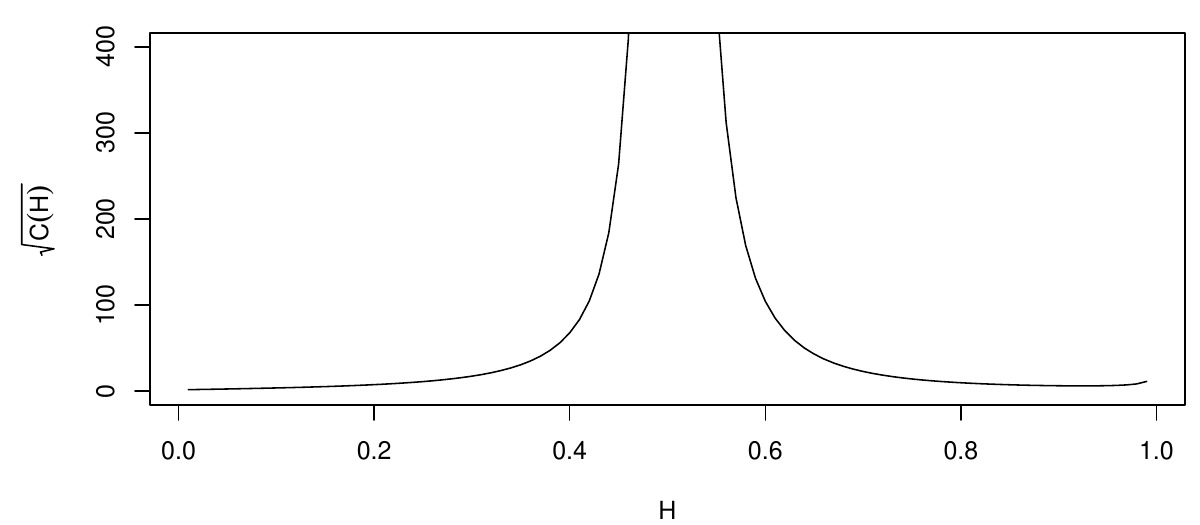}
	\caption{Constant in the asymptotic variance of the optimal GMM estimator of $H$ in an mfBm model.}\label{fig}
\end{figure}

As we can see, while the rate of convergence of $\wh H^n$ for $H>\frac12$ improves as $H\downarrow \frac12$, the associated constant deteriorates. This happens because in the limit as $H\downarrow \frac12$,  mfBm converges in distribution to Brownian motion, with the consequence that the fBm and the BM parts can no longer be separately identified. The same argument clearly applies when $H\uparrow \frac12$, except that here the convergence rate itself does not change with $H$. The divergence of $C(H)$ as $H\to\frac12$ will have an impact on the finite-sample performance of our estimators, as we shall see momentarily in a Monte Carlo simulation.
Besides the singular behavior at $H\approx \frac12$ we also find the asymptotic variance to be very large in absolute terms. This demonstrates the intrinsic difficulty of the statistical problem.

\section{Simulation study}\label{sec:sim}

In order to evaluate the performance of the estimators from Theorem~\ref{thm:est}, we simulate $\{Y_{i\Den}:i=1,\dots,[T/\Den]\}$ from the mfBm model \eqref{eq:mfBm} with $\Den=1/23{,}400$ and $T=20$, which in a typical financial context corresponds to sampling every second and aggregating one month of data. We consider $H\in\{0.1,0.2,0.3,0.4, 0.6, 0.7, 0.8, 0.9\}$ and $\la\in\{-0.9,-0.5,0,0.5,0.9\}$ and further take $\si=0.02$. In order to simulate from a setting that is representative of the magnitude of noise in high-frequency financial data (see e.g., \cite{AY09}), we fix the signal-to-noise ratio to $2:1$, that is, we assume that  the increments of Brownian motion are responsible for $2/3$ of the variance of $\Delta^n_i Y$ and compute $\rho$ accordingly.
We also include $H=0.5$, in which case we let $\rho=\la=0$. 
In Appendix~\ref{app:sim}, we present additional simulation results where we fix $\si$ and $\rho$ instead of the signal-to-noise ratio.

Regarding the tuning parameters, we choose $r=31$ and  $k_n=300 \approx 2\Den^{-1/2}$ in \eqref{eq:acov}, which corresponds to considering autocorrelations up to half a minute and computing the local autocovariance estimates $\wh m^{n,j}_i$ over 5-minute blocks (if $\Den$ has the meaning of one second).  We also experimented with $r=16$ and obtained similar results. In \eqref{eq:acov}, we further choose the Parzen kernel $w(\ell,\ell_n)=w(\ell/(\ell_n+1))$, where $w(x)=(1-6x^2+6x^3)\bone_{\{x\leq 1/2\}} + 2(1-x)^3\bone_{\{x>1/2\}}$ and the sequence $\ell_n$ is chosen according to the optimal procedure by \cite{NW94} (with the details given in Table I B and Table II C of the reference). With this choice, $\wh \Sigma_n$ is positive semidefinite in finite samples and $\ell_n$ is of order $\Den^{-1/5}$ and hence satisfies $\ell_n/\sqrt{k_n}\to0$ and $\ell_n\sqrt{k_n\Den}\to0$ if $k_n$ is order $\Den^{-1/2}$. 

For every simulated path, we first use a classical Ljung--Box test statistic $Q^n_T=\lfloor T/\Den\rfloor(\lfloor T/\Den\rfloor+2)\sum_{\ell=1}^{r-1} (R^n_\ell)^2/(\lfloor T/\Den\rfloor-\ell)$, where $R^n_\ell$ is the sample autocorrelation coefficient of $\{\Delta^n_i Y:i=1,\dots, \lfloor T/\Den\rfloor\}$ at lag $\ell$, to discriminate whether $H=\frac{1}{2}$ or not. Indeed, if $H=\frac{1}{2}$, we have $S^n_T \limd \chi^2_{r-1}$, while $S^n_T\limp \infty$ in the case $H\neq \frac{1}{2}$ because $R^n_\ell\limp \Ga^H_\ell\neq0$.
	If the Brownian case $H=\frac{1}{2}$ is not rejected\color{black}, we let $(H_n,C_n,\La_n,\Pi_n)=(\frac12,\wh V^n_{0,T},0,0)$ be the estimated parameter vector. This initial test is necessary because as $H\to\frac12$, the mfBm model collapses to Brownian motion. Thus, if $H=\frac12$ (or close to $\frac12$), simply minimizing \eqref{eq:optprob} typically produces a value of $H$ which is very close to $\frac12$ and arbitrary splits the total variance among $C$, $\Pi$ and $\La$.
	This is in line with the behavior of the asymptotic variance as $H\to \frac{1}{2}$ depicted in Figure \ref{fig}. Therefore, the initial test we perform can be seen as shrinking our estimators towards the Brownian model to circumvent the weakly identified regime. 
	
	If $H=\frac12$ is rejected, we let $(H_n,C_n,\La_n,\Pi_n)$ be a numerical solution to \eqref{eq:optprob} found in the following way: For each candidate $H$, we first
	we run a weighted linear regression of $\calw_n^{1/2}\wh V^n_T$ on $\calw_n^{1/2}e_1$, $\Den^{H-1/2}\calw_n^{1/2}\Phi^H$ and $\Den^{2H-1}\calw_n^{1/2}\Ga^H$ (with intercept forced to be $0$ and the optimal weight matrix $\calw_n=\wh \Sigma_n^{-1}$) with the constraint that $C$, $\La$ and $\Pi$ must satisfy $\Pi \geq0$,  $C \geq0$ and $\La ^2\leq\Pi C$. This is to reflect the fact that $\rho^2\geq0$, $\si^2\geq0$ and $\La_T^2\leq \Pi_TC_T$.  Denoting the resulting coefficients by $C_n(H)$, $\La_n(H)$ and $\Pi_n(H)$, we construct  $H_n$ by minimizing the objective function $H\mapsto \mathrm{score}(H)$ on the interval $(0,1)$, where $\mathrm{score}(H)$ is the sum of squared residuals in the regression analysis associated with $H$ but set to $\mathrm{score}(H)=\infty$ if  $C_n(H)/(\Pi_n(H)\Den^{2H-1}+\Phi^H_0\La_n(H)\Den^{H-1/2}+C_n(H))>0.99$. In the latter case, we consider the fractional part as practically absent if the Brownian motion part accounts for more than 99\% of the variance of increments. Indeed, if $H$ is close to or equal to $\frac12$, it can happen that adding a tiny fractional component with $H$ very close to $0$ can achieve a higher score value, which  is undesirable in practice.
	
 Finally, we define the remaining estimators by $C_n=C_n(H_n)$, $\La_n=\La_n(H_n)$ and $\Pi_n=\Pi_n(H_n)$. If $\mathrm{score}(H)=\infty$ for all $H$, we let $H_n=0.5$ and $(H_n,C_n,\La_n,\Pi_n)=(\frac12,\wh V^n_{0,T},0,0)$. Comparing with Theorem~\ref{thm:est}, we note that $(H_n,C_n,\La_n,\Pi_n)$ is  equal to $(\wh H^n,\wh C^n_T,\wh \La^n_T,\wh\Pi^n_T)$ unless one of the constraints or exceptions above occur (which happen with asymptotically vanishing probability). Having obtained these estimators using $T=20$ days of data, we further estimate integrated volatility on the last day, that is, $C_{20}-C_{19}=\si^2$, by rerunning the weighted and constrained linear regression mentioned above, but using only data of the last day and with $H$ fixed at $H_n$ (which was previously obtained using 20 days of data). We denote this estimator by $C^{[19,20]}_n$. In a financial context, this mimics daily estimation of integrated volatility based on an estimate of $H$ obtained from a moving window of one month.

\begin{table*}
	\caption{Median and interquartile range of $H_n$ based on 1{,}000 simulated paths.}
	\label{tab:1}
	\begin{tabular}{@{}cccccc@{}}
		& \multicolumn{5}{c}{$\la$}\\
		\hline 
		$H$ & \multicolumn{1}{c}{$-0.9$}&\multicolumn{1}{c}{$-0.5$}
		& \multicolumn{1}{c}{0} & \multicolumn{1}{c}{$0.5$}&\multicolumn{1}{c}{$0.9$}
		\\
		\hline
		$0.1$  & 0.0995 &  						0.0989&   					0.0960   & 					0.0976& 					0.1492  \\
		~ &  [0.0939, 0.1058]  & [0.0860, 0.1133] & [0.0467,  0.1364]& [0.0582, 0.1951]& [0.1123,   0.2499] \smallskip \\
		$0.2$    & 0.2000 &						0.1996   & 					0.1975    &					0.2104&  				0.2590  \\
		~ &[0.1959,     0.2038]&[0.1868,  0.2121]&[0.1139,  0.2443]&[0.1769,  0.3203]&[0.2237,   0.3646] \smallskip \\
		$0.3$  &   0.2999   &   				0.2992   &   				0.2902  &    					0.3292  &  					0.3603   \\
		~& [0.2953,  0.3044]  &[0.2806,  0.3166]  & [0.2071, 0.3615]  & [0.2887, 0.4081] & [0.3184, 0.4198] \smallskip \\
		$0.4$  &   0.3994  &  					0.3981   &    				0.3789  &  					0.4019 &   						0.4148 \\
		~& [0.3818,  0.4084]  & [0.3192, 0.4157] & [0.3493, 0.4287]  & [0.3699,  0.4389]  & [0.3775,  0.4679]\smallskip \\
		$0.5$ &   &  & 0.5000 &  & \\
		~&   &  & [0.5000,0.5000] &  &\smallskip \\
		$0.6$  & 0.5999 &  						0.5986&   					0.5954   & 					0.5977& 					0.5818  \\
		~ &  [0.5942, 0.6133]  & [0.5888, 0.6441] & [0.5726,  0.6463]& [0.5663, 0.6255]& [0.5428,   0.6151] \smallskip \\
		$0.7$    & 0.6994 &						0.6991   & 					0.6983    &					0.7044&  				0.6629  \\
		~ &[0.6922,     0.7071]&[0.6888,  0.7107]&[0.6739,  0.7336]&[0.6298,  0.7332]&[0.6012,   0.6962] \smallskip \\
		$0.8$  &   0.7978   &   				0.7975   &   				0.7981  &    					0.7975  &  					0.7879   \\
		~& [0.7916,  0.8055]  &[0.7897,  0.8072]  & [0.7856, 0.8126]  & [0.7688, 0.8364] & [0.7354, 0.8081] \smallskip \\
		$0.9$  &   0.8918  &  					0.8913   &    				0.8905  &  					0.8895 &   						0.8885 \\
		~& [0.8819,  0.9058]  & [0.8801, 0.9062] & [0.8778, 0.9074]  & [0.8727,  0.9118]  & [0.8635,  0.9067]\smallskip \\
		\hline
	\end{tabular}
\end{table*}

Table~\ref{tab:1} summarizes the results for $H_n$. If $\la=-0.9$ or $\la=-0.5$, the estimator $H_n$ is centered around the true value of $H$ with only low to moderate dispersion uniformly for all considered values of $H$. At $\la=0$ or $\la=0.5$, $H_n$ is still relatively centered around its true value but there is a noticeable increase in the variability of the estimates. At $\la=0.9$, the estimator $H_n$   exhibits a clear upward (resp., downward) bias for $H<0.5$ (resp., $H>0.5$). It is interesting that this bias together with an increase in the spread only appears at the positive end but not at the negative end of $\la$. In fact, additional plots  (not shown here to save space) reveal that for negative values of $\la$, the empirical distribution of $H_n$ has a symmetrical bell shape around the true value of $H$, while for positive values of $\lambda$, the empirical distribution of $H_n$ becomes bimodal, with one local maximum around $H$ and another one at some value not far from $\frac12(\frac12+H)$. We believe that this is due to the fact that the original fBm with Hurst index $H$ and the fictitious one with Hurst parameter $\frac12(H+\frac12)$ in the case of positive (resp., negative) $\la$ introduce return autocorrelations with the same (resp., opposite) sign, making it harder (resp., easier) to separate them. 

\begin{table*}
	\caption{Median and interquartile range  of $C^{[19,20]}_n/\si^2$ based on 1{,}000 simulated paths.}
	\label{tab:2}
	\begin{tabular}{@{}cccccc@{}}
		& \multicolumn{5}{c}{$\la$}\\
		\hline 
		$H$ & \multicolumn{1}{c}{$-0.9$}&\multicolumn{1}{c}{$-0.5$}
		& \multicolumn{1}{c}{0} & \multicolumn{1}{c}{$0.5$}&\multicolumn{1}{c}{$0.9$}
		\\
		\hline
	$0.1$  & 1.0005 &  						1.0009&   					1.0064   & 					1.0199& 					1.0229  \\
~ &  [0.9614, 1.0390]  & [0.9539, 1.0461] & [0.9585,  1.0553]& [0.9728, 1.0738]& [0.9698,   1.0940] \smallskip \\
$0.2$    & 0.9994 &						0.9984   & 					1.0103    &					1.0442&  				1.0465  \\
~ &[0.9532,     1.0491]&[0.9325,  1.0753]&[0.9506,  1.1020]&[0.9774,  1.1547]&[0.9604,   1.1942] \smallskip \\
$0.3$  &   1.0004   &   				1.0029   &   				1.0302  &    					1.0759  &  					1.0722   \\
~& [0.9232,  1.0882]  &[0.8592,  1.1809]  & [0.9414, 1.2249]  & [0.9576, 1.3396] & [0.9246, 1.4233] \smallskip \\
$0.4$  &   0.9092  &  					0.8077   &    				1.0487  &  					1.1634 &   						1.2228 \\
~& [0.4252,  1.3912]  & [0.5862, 1.4440] & [0.8745, 1.4470]  & [0.9696,  1.6896]  & [0.9840,  2.0461]\smallskip \\
$0.5$ &   &  & 1.0007 &  & \\
~&   &  & [0.9942,1.1068] &  &\smallskip \\
	$0.6$  & 0.8722 &  						0.8084&   					1.0727   & 					1.1545& 					1.1850  \\
~ &  [0.3602, 1.7096]  & [0.4232, 1.5999] & [0.8752,  1.6210]& [0.9570, 1.5893]& [0.9590,   1.7388] \smallskip \\
$0.7$    & 1.0109 &						1.0158   & 					1.0128    &					1.0727&  				1.0874  \\
~ &[0.7115,     1.3561]&[0.7783,  1.2695]&[0.9008,  1.2201]&[0.9859,  1.2302]&[0.9627,   1.3477] \smallskip \\
$0.8$  &   1.0240   &   				1.0058   &   				0.9995  &    					1.0215  &  					1.0297   \\
~& [0.8769,  1.1752]  &[0.8721,  1.1419]  & [0.9220, 1.0882]  & [0.9783, 1.0730] & [0.9862, 1.0846] \smallskip \\
$0.9$  &   1.0305  &  					1.0005   &    				0.9879  &  					1.0035 &   						1.0208 \\
~& [0.9144,  1.1296]  & [0.8994, 1.1160] & [0.9251, 1.0768]  & [0.9647,  1.0469]  & [0.9875,  1.0553]\smallskip \\
		\hline
	\end{tabular}
\end{table*}

The situation is different for the volatility estimator $C^{[19,20]}_n$, as Table~\ref{tab:2} reveals. Here the results vary much less with $\la$ but mainly with the value of $H$ itself. Both the median bias and the interquartile range tend to increase as $H$ gets closer to but does not reach $\frac12$. This is because the fBm, the fictitious fBm and the Brownian motion collapse  to one in the limit as $H\to\frac12$, making it harder to distinguish the three as $H$ approaches $\frac12$ (see the discussion at the end of   Section~\ref{sec:lb}). It is interesting to note that the volatility estimator performs quite well even if $H<\frac14$, even though in theory volatility is not identifiable in this case. This, of course, is due to our simulation setup, in which we fix the signal-to-noise ratio rather than $\si$ and $\rho$. Indeed, if we fix $\si$ and $\rho$ instead, then the signal-to-noise ratio decreases sharply with $H$ (e.g., if $\si=\rho$ and $\Delta_n=23{,}400$, the signal-to-noise ratio will be in the range 1:3000 to 1:3200, depending on $\la$, at $H=0.1$). In additional simulations in Appendix~\ref{app:sim}, where we fix $\si$ and $\rho$, we do see that the variance of $C_n^{[19,20]}$ increases as $H$ becomes smaller.

\begin{table*}
	\caption{Median and interquartile range   of $\La_n/\sqrt{C_n\Pi_n}$ (defined as $0$ if $\Pi_n=0$) based on 1{,}000 simulated paths.}
	\label{tab:3}
	\begin{tabular}{@{}cccccc@{}}
		& \multicolumn{5}{c}{$\la$}\\
		\hline 
		$H$ & \multicolumn{1}{c}{$-0.9$}&\multicolumn{1}{c}{$-0.5$}
		& \multicolumn{1}{c}{0} & \multicolumn{1}{c}{$0.5$}&\multicolumn{1}{c}{$0.9$}
		\\
		\hline
		$0.1$  & -0.9005 &  						-0.5007&   					0.0104   & 					0.5690& 					0.5619  \\
		~ &  [-0.9069, -0.8949]  & [-0.5248, -0.4749] & [-0.1297,  0.2866]& [0.0162, 0.9707]& [-0.0461,   0.7417] \smallskip \\
		$0.2$    & -0.9001 &						-0.4999   & 					0.0113    &					0.5377&  				0.4909  \\
		~ &[-0.9050,     -0.8945]&[-0.5382,  -0.4569]&[-0.1839,  0.4411]&[-0.1937,  0.6355]&[-0.3396,   0.6186] \smallskip \\
		$0.3$  &   -0.8996   &   				-0.4970   &   				0.0162  &    					0.4539  &  					0.4700   \\
		~& [-0.9089,  -0.8903]  &[-0.5864,  -0.3905]  & [-0.4076, 0.6350]  & [-0.4873, 0.5873] & [-0.4875, 0.6017] \smallskip \\
		$0.4$  &   -0.9040 &  					-0.4844   &    				0.5423  &  					0.5490 &   						0.5130 \\
		~& [-0.9265,  -0.8240]  & [-0.6839, 0.8631] & [-0.4863, 0.6871]  & [-0.3783,  0.6134]  & [-0.5045,  0.9777]\smallskip \\
		$0.6$  & -0.9209 &  						-0.5482&   					-0.1145   & 					0.6914& 					0.6545  \\
		~ &  [-0.9997, -0.9035]  & [-0.6673, 0.9978] & [-0.4754,  0.8584]& [-0.3639, 0.7694]& [-0.3771,   0.9811] \smallskip \\
		$0.7$    & -0.9020 &						-0.5082   & 					-0.0196    &					0.6289&  				0.4570  \\
		~ &[-0.9171,     -0.8824]&[-0.5955,  -0.3894]&[-0.2687,  0.3519]&[-0.2415,  0.8973]&[-0.4257,   0.8507] \smallskip \\
		$0.8$  &   -0.9046   &   				-0.5117   &   				-0.0089  &    					0.4797  &  					0.8148   \\
		~& [-0.9166,  -0.8866]  &[-0.5536,  -0.4564]  & [-0.1019, 0.1021]  & [0.2148, 1.0000] & [0.2354, 0.9998] \smallskip \\
		$0.9$  &   -0.9071  &  					-0.5226   &    				-0.0317  &  					0.4544 &   						0.8219 \\
		~& [-0.9174,  -0.8929]  & [-0.5552, -0.4771] & [-0.0952, 0.0486]  & [0.3233,  0.6273]  & [0.5653,  1.0000]\smallskip \\
		\hline
	\end{tabular}
\end{table*}

Next, we report simulation results for $\La_n/\sqrt{C_n\Pi_n}$ as an estimator of $\lambda$ in Table~\ref{tab:3}. Here the picture is closer to what we observed for $H_n$. For nonpositive values of $\lambda$, the estimator performs relatively well (with exceptions), while biases start to show up as $\la$ moves into the positive range. In fact, comparing the results between $\la=0.5$ and $\la=0.9$, it seems  that the estimator $\La_n/\sqrt{C_n\Pi_n}$ has a hard time distinguishing between these two cases. We  do not have a plausible explanation for this  behavior.

Finally, we consider the simulation results for $\Pi_n$, which can be found in Table~\ref{tab:4}. Here the impact of $\la$ is particularly striking: the estimator shows a relatively good performance for negative values of $\la$ and is practically useless for positive values of $\la$. We conjecture that the cause is the same as before: with positive values of $\lambda$, both the original and the fictitious fBm components lead to autocorrelations of the same sign, making it difficult to separate them. 

\begin{table*}
	\caption{Median and interquartile range  of $\Pi_n/(20\rho^2)$ based on 1{,}000 simulated paths.}
	\label{tab:4}
	\begin{tabular}{@{}cccccc@{}}
		& \multicolumn{5}{c}{$\la$}\\
		\hline 
		$H$ & \multicolumn{1}{c}{$-0.9$}&\multicolumn{1}{c}{$-0.5$}
		& \multicolumn{1}{c}{0} & \multicolumn{1}{c}{$0.5$}&\multicolumn{1}{c}{$0.9$}
		\\
		\hline
		$0.1$  & 0.9868 &  						0.9726& 					  0.9040   & 				0.8409& 					3.8838  \\
		~ &  [0.8626, 1.1509]  & [0.7019, 1.3861] & [0.2543,  2.5844]& [0.2994, 14.5667]& [1.5175,   84.0271] \smallskip \\
		$0.2$    & 0.9998 &						0.9858   & 					0.9354    &				1.1647&  				6.7238  \\
		~ &[0.9008,     1.1021]&[0.7025,  1.3560]&[0.0912,  3.5754]&[0.4954,  44.1307]&[2.5029,   323.7322] \smallskip \\
		$0.3$  &   0.9954   &   			0.9829   &   				0.8002  &    				2.0526  &  					8.1592   \\
		~& [0.8723,  1.1385]  &[0.5547,  1.6563]  & [0.0457, 8.3227]  & [0.6617, 81.7545] & [2.5609, 228.6105] \smallskip \\
		$0.4$  &   0.9831  &  					0.9140   &    				0.1906  &  				0.8187 &   						2.8029 \\
		~& [0.4771,  1.4473]  & [0.0331, 2.0469] & [0.0851, 4.9195]  & [0.2751,  13.0497]  & [0.6495,  97.6562]\smallskip \\
			$0.6$  & 1.0076 &  						1.0299& 					  1.1053   & 				0.8441& 					1.5601  \\
		~ &  [0.8997, 1.0551]  & [0.7188, 1.1265] & [0.5266,  1.5510]& [0.6166, 3.2883]& [1.1024,   6.5762] \smallskip \\
		$0.7$    & 0.9935 &						0.9992   & 					1.0033    &				1.0671&  				1.3798  \\
		~ &[0.9484,     1.0526]&[0.9406,  1.0637]&[0.9400,  1.0914]&[0.8951,  1.2127]&[1.0304,   1.9491] \smallskip \\
		$0.8$  &   0.9669   &   			0.9662   &   				0.9681  &    				0.9780  &  					0.9236   \\
		~& [0.8898,  1.0709]  &[0.8805,  1.0889]  & [0.8674, 1.1212]  & [0.8290, 1.1729] & [0.8048, 1.0568] \smallskip \\
		$0.9$  &   0.8405  &  					0.8333   &    				0.8348  &  				0.8388 &   						0.8157 \\
		~& [0.6926,  1.1298]  & [0.6825, 1.1369] & [0.6746, 1.1419]  & [0.6578,  1.1673]  & [0.6453,  1.1054]\smallskip \\
		\hline
	\end{tabular}
\end{table*}

In summary, we draw the following conclusions from the simulation study for a scenario where the signal-to-noise ratio is fixed:
\begin{itemize}
	\item The statistical estimation of mixed semimartingale models is a hard task in general. Even in a parametric mfBm model, the behavior of the estimators $H_n$, $\Pi_n$, $C_n^{[19,20]}$ and $\La_n/\sqrt{\Pi_nC_n}$ is quite  different for different true parameter values.
	\item If $\la$ is negative and $H$ is bounded away from $\frac12$, all estimators $H_n$, $\Pi_n$, $C_n^{[19,20]}$ and $\La_n/\sqrt{\Pi_nC_n}$ perform relatively well. This is exactly the case where the statistical properties of the fractional component,  the fictitious fractional component and the Brownian motion component are sufficiently distinct to  separate  them from each other.
	\item In all other cases, at least two of three are statistically similar, making it intrinsically difficult  to disentangle them. 
\end{itemize}

\section{Proof of Theorem \ref{thm:intro}, Theorem~\ref{thm:est}, Corollary~\ref{cor:var} and Theorem \ref{thm:lb}}\label{sec:6}

\begin{proof}[Proof of Theorem \ref{thm:intro}]
One could derive the first statement of the theorem from general results about multivariate fractional Brownian motion \cite{amblard2012basic}. For the sake of completeness, we give a short direct proof here.
Since both  $Y$ and $\widetilde{Y}$ are centered Gaussian processes with stationary increments, it suffices to study the variance of increments.
On the one hand, 
$
	\E [(\widetilde{Y}_{t+h}-\widetilde{Y}_t)^2 ]
	= \sigma^2 h +  2\lambda\rho\sigma(K_H(H+\frac12))^{-1} h^{2\overline{H}} + \rho^2 h^{2H};
$
on the other hand, we can compute $\E[ (Y_{t+h}-Y_t)^2]$ via Itô's isometry.
Writing $h_H(t,s) =K_H^{-1} [ (t-s)^{H- {1}/{2}}_+- (-s)^{H- {1}/{2}}_+ ]$, we have that $B(H)_t=\int_{-\infty}^t h_H(t,s) \, \dd \wt B_s$ and 
\begin{equation} \label{eqn:increments} 
		\E[ (Y_{t+h}-Y_t)^2 ]
		 = \sigma^2 h + \rho^2 h^{2H} + 2 \lambda \rho\sigma \int_{t}^{t+h} [h_H(t+h,s) - h_H(t,s)] \, \dd s,   
 \end{equation}
where the last integral equals
$
  \int_{0}^{h} [h_{H}(h,s)-h_{H}(0,s)]\, \dd s   =  K_H^{-1}\int_0^h (h-s)^{H-\frac{1}{2}}\, \dd s   =  {h^{2\overline{H}}}/(K_H(H+\frac12)). 
$
Thus, $\E[ (Y_{t+h}-Y_t)^2 ] = \E [(\widetilde{Y}_{t+h}-\widetilde{Y}_t)^2]$ and $(Y_t)_{t\geq 0} \overset{d}{=} (\widetilde{Y}_t)_{t\geq 0}$.

It remains to show that  $Y$ is not a semimartingale for $\lambda>0$.
To this end, we may employ the same arguments as \cite{Cheridito01}:
If $H<\frac{1}{2}$, the process has an infinite quadratic variation, which contradicts the semimartingale property.  If $H\in(\frac{3}{4}, 1)$, then the law of $\sigma W + \rho W(H)$ is locally equivalent to that of $\sigma W$. Thus,  $Y$ is locally equivalent to $\sigma W + (2\lambda\rho\sigma / (K_H(H+\frac12)))^{1/2} W(\overline{H})$. 
Since $\overline{H}\in(\frac{1}{2}, \frac{3}{4})$,   \cite[Theorem 1.7]{Cheridito01} shows that $Y$ is not a semimartingale. If $H\in(\frac{1}{2}, \frac{3}{4}]$, then $Y$ is the sum of three independent (fractional) Brownian motions. 
The proof given in \cite[Section 4]{Cheridito01},  which is presented for two components, straightforwardly generalizes to the case of three fractional Brownian motions, showing
that $Y$ is not a semimartingale.
\end{proof}

\begin{proof}[Proof of Theorem~\ref{thm:est}]
	We want to apply \cite[Theorem A.2]{MP22}, so we verify the associated conditions called  (E.1) and (E.2)'. Note that the latter theory is formulated for classical weak convergence but readily extends to stable convergence. 
	We only consider the case $H\in(\frac14,\frac12)\cup(\frac12,\frac34)$ as the other two cases $H\in(0,\frac14)$ and $H\in(\frac34,1)$ are analogous. For (E.1), define the matrices $A_n=A_n(\theta), B_n=B_n(\theta)\in\R^{4\times 4}$ as
	\begin{align*}
		\begin{dcases}
		A_n 
		= \Delta_n^{\frac{1}{2}-2H} B_n, \quad B_n=
		\begin{pmatrix} \Delta_n^{1-2H}&0&0&2\Delta_n^{1-2H}\log(\Delta_n^{-1})\Pi \\ 0&1&0&0 \\ 0&0&\Delta_n^{\frac{1}{2}-H}&0 \\ 0&0&0& \Delta_n^{1-2H} \end{pmatrix} & \text{ for }H<\tfrac{1}{2}, \\
		A_n 
		= \Delta_n^{-\frac{1}{2}}B_n, \quad B_n=
		\begin{pmatrix} \Delta_n^{\frac{1}{2}-H}&0&\Delta_n^{\frac{1}{2}-H}\log (\Delta_n^{-1}) \Lambda&0 \\ 0&1&0&0 \\ 0&0&\Delta_n^{\frac{1}{2}-H}&0 \\ 0&0&0& \Delta_n^{1-2H} \end{pmatrix} & \text{ for }H>\tfrac{1}{2}. \end{dcases}
	\end{align*}
	Since $
		D_\theta \mu_n(\theta) =( \Delta_n^{H-\frac{1}{2}}(\partial_H \Phi^H-\Phi^H\log \Den^{-1})\Lambda + \Delta_n^{2H-1} (\partial_H\Gamma^H-2\Ga^H\log \Den^{-1})  \Pi , e_1,\linebreak \Delta_n^{H-\frac{1}{2}} \Phi^H, \Delta_n^{2H-1} \Gamma^H)\in\R^{r\times 4}$,
		we have by Theorem \ref{cor} that 
	\[ \begin{cases} A_n(\Theta_T)F_n(\Theta_T)	\stackrel{\mathrm{st}}{\longrightarrow} -2 ( \partial_H \Gamma^H \Pi, e_1, \Phi^H, \Ga^H)^T \mathcal{W}\calz_T &\text{for } H\in(\frac14,\frac12),\\ A_n(\Theta_T)F_n(\Theta_T)	\stackrel{\mathrm{st}}{\longrightarrow} -2(\partial_H \Phi^H \Lambda, e_1, \Phi^H, \Ga^H )^T\mathcal{W} \calz'_T &\text{for } H\in(\frac12,\frac34). \end{cases}\]
		 This proves (E.1) in \cite{MP22}.
	
	For (E.2)', note that  continuous differentiability of $F_n(\theta)$ around $\Theta_T$ is clear. Next, we observe that
	\begin{align*}
		D_\theta F_n(\theta) = -2D^2_\theta \mu_n(\theta)^T \mathcal{W}_n ( \widehat{V}^n_T-\mu_n(\theta)  ) +2 D_\theta \mu_n(\theta)^T\mathcal{W}_nD_\theta \mu_n(\theta).
 	\end{align*}
	A straightforward computation shows that
	$2B_n(\Theta_T) [ D_\theta \mu_n(\theta)^T\mathcal{W}_n D_\theta \mu_n(\theta) ]B_n(\Theta_T)^T \limp 2E$, locally uniformly in a shrinking neighborhood of size $r_n=1/(\log\Delta_n^{-1})^2$ around $\Theta_T$. Applying Theorem~\ref{cor}, we further have    $B_n  (\Theta_T) D^2_\theta\mu_n(\theta)^T\mathcal{W}_n (\widehat{V}^n_T-\mu_n(\theta))  B_n(\Theta_T)^T \limp 0$ locally uniformly. Hence,
\begin{equation}\label{eq:aux} 
		\sup_{\theta\,:\, \lvert\theta-\Theta_T\rvert\leq 1/(\log\Delta_n^{-1})^2} \lVert B_n(\Theta_T) D_\theta F_n(\theta)B_n(\Theta_T)^T - 2E\rVert \limp0. 
\end{equation}
	Finally, by \eqref{eq:ABCW-1} and \eqref{eq:ABCW-2}, we can  check that $\lVert B_n(\Theta_T)^T\rVert\lVert B_n(\Theta_T)A_n(\Theta_T)^{-1}\rVert/r_n\limp0$ in the range of $H$ we consider, which  proves (E.2)' in 
	\cite{MP22}, with $C_n=B_n^T$ and $W=2E$.
	
	The matrix $E$ is regular because the vectors $e_1, \partial_H \Gamma^H, \Phi^H, \Gamma^H\in\R^{r}$ are linearly independent. For $r\to \infty$, this is evident as all four vectors have different decay rates. For $r\geq 5$ fixed, we can check that the $3\times 3$ submatrix consisting of the entries two, three, and five (i.e.\ lags $1,2,4$) of $\partial_H \Gamma^H, \Phi^H, \Gamma^H$ has a non-zero determinant. We have verified that this is the case for $H\neq \frac{1}{2}$, using a computer algebra system.
	Analogously, we have verified the regularity of the same matrix based on $\partial_H \Phi^H, \Phi^H, \Gamma^H$.
	Thus, \cite[Theorem A.2]{MP22} yields, for $H\in (\frac{1}{4},\frac{1}{2})$,
	\begin{align*}
		A_n(\Theta_T) B_n(\Theta_T)^{-1} E (B_n(\Theta_T)^{-1})^T  ( \wh{\Theta}_T^n - \Theta_T  ) \stackrel{\mathrm{st}}{\longrightarrow} (\partial_H \Gamma^H \Pi_T, e_1, \Phi^H, \Ga^H)^T \mathcal{W}\calz_T,
	\end{align*}
	which is equivalent to \eqref{eq:CLT-est-1}. For $H\in(\frac{1}{2},\frac{3}{4})$, we obtain \eqref{eq:CLT-est-2}. 
	
	Part 2 (resp., Part 4) of the theorem can be derived along the same lines, but with the second (resp., fourth) row and column (out of four) deleted from all vectors and matrices. For Part 5, we restrict ourselves to the setting of Part 3, where $H\in(\frac34,1)$ (the proof in the setting of Part 2 is similar). On the event $F_n(\wh \Theta^n_T)=0$, which happens with probability converging to $1$, we have
	\[ 0=F_n(\wh \Theta^n_T) = -2D_\theta\mu_n(\wh\Theta^n_T)^T\calw_n(\wh V^n_T-\mu_n(\wh\Theta^n_T)). \]
	Next, we introduce the matrix $G_n(\theta)=((\partial_H\Ga^H-2\Ga^H\log\Den^{-1})\Den^{2H-1}\Pi,0,0)\in\R^{r\times 3}$ and denote the restriction of a matrix $M\in\R^{4\times 4}$  to the upper left $3\times 3$-corner by  $M^{(2)}$ and the restriction of a vector $v\in \R^4$ to the first three entries by $v^{(2)}$. Now, applying $(\cdot)^{(2)}$ to both sides of the previous display and  multiplying the result by $A^{(2)}_n(\Theta^{(2)}_T)$ (recall the definition of  $A_n(\theta)$ and $B_n(\theta)$ from the beginning of this proof), we obtain
	\begin{align*}
	0	&= - 2A^{(2)}_n(\Theta^{(2)}_T)(D_{\theta^{(2)}}\mu^{(2)}_n(\ov\Theta^n_T)+G_n(\wh\Theta^n_T))^T\calw_n(\wh V^n_T-\mu_n^{(2)}(\ov\Theta^n_T)-\Den^{2\wh H^n-1}\Ga^{\wh H^n}\wh\Pi^n_T) \\
		&= - 2A^{(2)}_n(\Theta^{(2)}_T)D_{\theta^{(2)}}\mu^{(2)}_n(\ov\Theta^n_T)^T\calw_n(\wh V^n_T-\mu_n^{(2)}(\ov\Theta^n_T)) + o_\bbp(1)\\
		&=A^{(2)}_n(\Theta^{(2)}_T)F_n^{(2)}(\ov\Theta^n_T)+ o_\bbp(1).
	\end{align*}
	For the second equality, note that $H\in(\frac34,1)$ and that $A^{(2)}_n(\Theta^{(2)}_T)$, $A^{(2)}_n(\Theta^{(2)}_T)G_n(\wh\Theta^n_T)^T$ and $D_{\theta^{(2)}}\mu^{(2)}_n(\ov\Theta^n_T)$ have matrix norms of order $\Den^{-1/2}$, $\Den^{H-1}\log \Den^{-1}$ and $1$, respectively, while $\wh V^n_T-\mu_n^{(2)}(\ov\Theta^n_T)=O_\bbp(\Den^{1/2})$ by Theorem~\ref{cor}. With high probability, $F_n^{(2)}(\wh\Theta^{n,(2)}_T)=0$, so in this case, Taylor's theorem gives us some $\wt\Theta^n_T$ between $\ov\Theta^n_T$ and $\wh\Theta^{n,(2)}_T$ such that
	\begin{align*}
	0	&=A^{(2)}_n(\Theta^{(2)}_T)D_{\theta^{(2)}}F_n^{(2)}(\wt \Theta^n_T)(\ov \Theta^n_T-\wh\Theta^{n,(2)}_T)+o_\bbp(1) \\
		&=A^{(2)}_n(\Theta^{(2)}_T)B^{(2)}_n(\Theta^{(2)}_T)^{-1}[B^{(2)}_n(\Theta^{(2)}_T)D_{\theta^{(2)}}F_n^{(2)}(\wt \Theta^n_T)B^{(2)}_n(\Theta^{(2)}_T)^T]\\
		&\quad\qquad\times(B^{(2)}_n(\Theta^{(2)}_T)^T)^{-1}(\ov \Theta^n_T-\wh\Theta^{n,(2)}_T)+o_\bbp(1)
	\end{align*}
	As in \eqref{eq:aux}, one can show that $B^{(2)}_n(\Theta^{(2)}_T)D_{\theta^{(2)}}F_n^{(2)}(\wt \Theta^n_T)B^{(2)}_n(\Theta^{(2)}_T)^T\limp 2E^{(2)}$. Since $A^{(2)}_n(\Theta^{(2)}_T)B^{(2)}_n(\Theta^{(2)}_T)^{-1}=\Den^{-1/2}\mathrm{Id}_3$, $\Den^{1/2}B^{(2)}_n(\Theta^{(2)}_T)=D_n^{(2)}$ and $E^{(2)}$ is regular, we conclude that $(D_n^{(2)})^{-1}(\ov \Theta^n_T-\wh\Theta^{n,(2)}_T)\limp0$.
\end{proof}

\begin{proof}[Proof of Corollary~\ref{cor:var}]
As before, we only consider one case, namely when $H\in(\frac12,\frac34)$. The other cases are similar. Let $$\ov m^{n,j}_i = \si^2_{(i-1)\Den}\bone_{\{j=0\}} + \si_{(i-1)\Den}\rho_{(i-1)\Den}\Phi^H_j\Den^{H-1/2} + (\rho_{(i-1)\Den}^2+\rho^{\prime 2}_{(i-1)\Den})\Ga^H_j\Den^{2H-1}$$
and  $\ov \psi_j^{(i)}$ (resp., $\ov \Sigma^{(\ell)}_n$, $\ov \Sigma_n$) be defined in the same way as $\psi_j^{(i)}$ (resp., $\Sigma^{(\ell)}_n$, $\Sigma_n$) but with $\ov m^{n,j}_i$ (resp., $\ov\psi^{(i)}$, $\ov\Sigma^{(\ell)}_n$) substituted for $\wh m^{n,j}_i$ (resp., $\psi^{(i)}$, $\wh\Sigma^{(\ell)}_n$). 
	Then, because $\si$, $\rho$ and $\rho'$are at least $\frac12$-H\"older continuous in $L^2$, one can borrow from classical results concerning spot volatility estimation (e.g., \cite[Chapter~13.3]{Jacod12}) to show that $\Den^{-1}(  m^{n,j}_i -\ov m^{n,j}_i ) = O_\bbp(k_n^{-1/2}\vee \sqrt{k_n\Den})$ uniformly in $j$, which implies $\Den^{-2}(\wh\Sigma_n-\ov\Sigma_n) = O_\bbp(\ell_n(k_n^{-1/2}\vee \sqrt{k_n\Den}))\limp 0$ by our assumptions on $\ell_n$ and $k_n$. Next, again because $\si$, $\rho$ and $\rho'$ are $\frac12$-H\"older continuous in $L^2$, we have   $\Den^{-2}\ov\Sigma_n = \int_0^T [c_n^{(0)}(s) + \sum_{\ell=1}^{\ell_n} w(\ell,\ell_n)(c_n^{(\ell)}(s)+c_n^{(\ell)}(s)^T)]\,\dd s + O_\bbp(\ell_n\sqrt{\Den})$, where 
	\begin{align*}
	&c_n^{(\ell)}(s)_{jj'}	=\si^4_s(2\bone_{\{\ell=j=j'=0\}}+\bone_{\{\ell=0,j=j'>0\}}) + \si^2\rho^2(\Phi^H_\ell\Phi^H_{\lvert j'-j-\ell\rvert} + \Phi^H_{\lvert j'-\ell\rvert}\Phi^H_{j+\ell})\Den^{2H-1} \\
		&\quad+(\rho^2_s+\rho^{\prime 2}_s)^2(\Ga^H_\ell\Ga^H_{\lvert j'-j-\ell\rvert} + \Ga^H_{\lvert j'-\ell\rvert}\Ga^H_{j+\ell})\Den^{4H-2} \\
		&\quad+ \si^3_s\rho_s(\Phi^H_{\lvert j'-j\rvert} (\bone_{\{\ell=0\}}+\bone_{\{\ell=j'-j\}}) + \Phi^H_{j+j'}(\bone_{\{\ell=j'\}}+\bone_{\{\ell=j=0\}}))\Den^{H-1/2}\\
		&\quad+ \si^2_s(\rho^2_s+\rho^{\prime 2}_s)(\Ga^H_{\lvert j'-j\rvert} (\bone_{\{\ell=0\}}+\bone_{\{\ell=j'-j\}}) + \Ga^H_{j+j'}(\bone_{\{\ell=j'\}}+\bone_{\{\ell=j=0\}}))\Den^{2H-1}\\
		&\quad+ \si_s\rho_s(\rho^2_s+\rho^{\prime 2}_s)(\Phi^H_{\ell} \Ga^H_{\lvert j'-j-\ell\rvert}+ \Ga^H_{\ell}\Phi^H_{\lvert j'-j-\ell\rvert}+\Phi^H_{\lvert j'-\ell\rvert}\Ga^H_{j+\ell}+\Ga^H_{\lvert j'-\ell\rvert}\Phi^H_{j+\ell})\Den^{3H-1/2}.
	\end{align*}
Note that all  terms
 defining $c_n^{(\ell)}(s)_{jj'}$ are  summable in $\ell$ because $H<\frac34$ and $\Ga^H_r  = O(r^{2H-2})$ and $\Phi^H_r = O(r^{H-3/2})$ as $r\to\infty$. Therefore, we have $\Den^{-2}\ov\Sigma_n = \diag(2, 1,\dots,1) \int_0^T \si^4_s \,\dd s + O_\bbp(\ell_n\sqrt{\Den} \vee \Den^{H-1/2}) \limp \calc'_T$. Since $\Den^{-3/2}\wh\eta_n D_n  \limp (\partial_H\Phi^H \La_T, e_1,\Phi^H, \Ga^H)$, we obtain
 \begin{align*}
 	D_n^{-1}\mathcal{V}_n(D_n^{-1})^T& =\Den(D_n^T\wh\eta_n^T  {  \mathcal{W}}_n\wh\eta_nD_n)^{-1}D_n^T\wh\eta_n^T  {  \mathcal{W}}_n\wh\Sig_n  {  \mathcal{W}}_n\wh\eta_nD_n(D_n^T\wh\eta_n^T   {  \mathcal{W}}_n\wh\eta_nD_n)^{-1} \\
 	& \limp E^{-1}(\partial_H\Phi^H \La_T, e_1,\Phi^H, \Ga^H)^T \mathcal{W} \calc'_T \mathcal{W}(\partial_H\Phi^H \La_T, e_1,\Phi^H, \Ga^H)E^{-1}.
 \end{align*}
 Recalling \eqref{eq:CLT-est-2}, we have shown that $\mathcal{V}_n$ consistently estimates the asymptotic covariance matrix of $\wh \Theta^n_T$, which is the claim of the corollary.
\end{proof}

For the proof of Theorem~\ref{thm:lb}, we assume $\Den=\frac1n$
to simplify notation.
Since $Y_0 =\wt Y_0=0$, observing $\{Y_{i/n}:i=1,\dots,n\}$ is equivalent to observing the increments $\{\Delta_i^n  Y:i=1,\dots,n\}$.
These increments constitute a stationary centered  Gaussian time series with some covariance matrix $\wt{\Sigma}_n(\theta)\in\R^{n\times n}$.
Noting that \eqref{eqn:increments} is also valid for $\lambda<0$, we find that
\begin{equation}\label{eq:sig}
	\wt{\Sigma}_n(\theta) = \sigma^2 \, n^{-1} I_n + \Pi n^{-2H} \Sigma_n(H) + \Lambda b(H) n^{-2\overline{H}}  \Sigma_n(\overline{H}),
\end{equation}
where $	\overline{H} = \frac12(H+\frac{1}{2})$, $b(H) = 2/\Gamma(H+\tfrac{3}{2})$ and
$\Sigma_n(H)=(\Ga^H_{\lvert i-j\rvert})_{j,k=1}^n$ is the covariance matrix of  $n$ consecutive normalized increments of a fractional Brownian motion with Hurst parameter $H$. Given $\theta_0\in\Theta$ and four nonnegative sequences $r_{1,n},r_{2,n},r_{3,n},r_{4,n}\to0$, we define $\theta_n\in\Theta$ by
\begin{equation}\label{eqn:theta-n}\begin{aligned}
		H(\theta_n) &= H+r_{1,n}, 
		&\quad \sigma^2(\theta_n) &= \sigma^2 + r_{2,n}, \\
		\Lambda(\theta_n) &= \Lambda \frac{b(H)}{b(H+r_{1,n})} n^{r_{1,n}}  (1+r_{3,n}) , 
		&\quad \Pi(\theta_n) &= \Pi n^{2r_{1,n}}  (1+ r_{4,n}),
\end{aligned}
\end{equation}
abbreviating $(H,\sigma,\Lambda,\Pi)=(H,\sigma,\Lambda,\Pi)(\theta_0)$.
The parameter $\theta_n$ is chosen carefully such that
\beq\label{eq:wtsigma}\begin{split}
	\wt{\Sigma}_n(\theta_n) 
	&= (\sigma^2+r_{2,n}) n^{-1} I_n + (1+r_{4,n}) \Pi n^{-2H} \Sigma_n(H+r_{1,n}) \\
	&\quad + (1+r_{3,n}) \Lambda b(H) n^{-2\overline{H}}  \Sigma_n(\overline{H} + \tfrac{r_{1,n}}{2}).
\end{split}\eeq
For now, we only assume that $r_{i,n}\to 0$ as $n\to\infty$ for $i=1,2,3,4$. 

Following the general approach outlined in \cite[Chapter~2]{Tsybakov2008}, we shall prove Theorem~\ref{thm:lb} by deriving sharp KL divergence estimates. Recall that for two covariance matrices $\Sigma_1,\Sigma_2\in\R^{n\times n}$, the KL divergence of the corresponding centered Gaussian distributions is given by
\begin{align*}
	\mathrm{KL}(\Sigma_1\dmid\Sigma_2)=\mathrm{KL}  ( \mathcal{N}(0, \Sigma_1) \dmid \mathcal{N}(0,\Sigma_0)  )
	&= \frac{1}{2} \biggl\{ \tr(\Sigma_0^{-1}\Sigma_1)-n+\log\frac{\det \Sigma_0}{\det \Sigma_1} \biggr\}.
\end{align*}
In the next proposition, which is the main technical estimate in the proof of Theorem~\ref{thm:lb}, we establish an upper bound on the KL divergence
\begin{align*}
	\mathrm{KL}(\theta_n\dmid\theta_0)=\mathrm{KL}(\wt{\Sigma}_n(\theta_n)\dmid \wt{\Sigma}_n(\theta_0)).
\end{align*}
We give the proof in Section~\ref{app:lb}.

\begin{Proposition}\label{lem:KL-complete}
	Suppose that $H=H(\theta_0)\in(0, \frac{1}{2})$.
	For any $\delta>0$ sufficiently small, there exists a $C=C(\theta_0,\delta)$ such that
	\begin{align*}
		\mathrm{KL}(\theta_n \dmid \theta_0) \leq C \begin{cases}
			r_{2,n}^2\,n + (r_{1,n}^2+r_{3,n}^2)\, n^{3-4\overline{H}}  + r_{4,n}^2\, n^{\delta} &\text{if } H\in[\tfrac{3}{4}, 1), \\
			r_{2,n}^2\,n + (r_{1,n}^2+r_{3,n}^2)\, n^{3-4\overline{H}}  + r_{4,n}^2 n^{3-4H} & \text{if }H\in(\tfrac{1}{2}, \tfrac{3}{4}), \\
			r_{2,n}^2\, n^{4H-1} + (r_{1,n}^2+r_{4,n}^2)\, n + r_{3,n}^2 \, n^{1-4(\overline{H}-H)}  &\text{if } H\in(\tfrac{1}{4}, \tfrac{1}{2}), \\
			r_{2,n}^2\, n^\delta + (r_{1,n}^2+r_{4,n}^2)\, n + r_{3,n}^2 \, n^{1-4(\overline{H}-H)} & \text{if }H\in(0,\tfrac{1}{4}]. 
		\end{cases}
	\end{align*}
\end{Proposition}

\begin{table}[t]
	\caption{Choice of $r_{1,n},\ldots, r_{4,n}$ in the proof of Theorem~\ref{thm:lb}.}
	\label{tab:rates-theta-n}
	{\def\arraystretch{2}\tabcolsep=10pt
		\begin{tabular}{c|c|c|c|c}
			& $H\in(0,\tfrac{1}{4})$ & $H\in (\tfrac{1}{4}, \tfrac{1}{2})$ & $H\in (\tfrac{1}{2}, \tfrac{3}{4})$ & $H\in (\tfrac{3}{4}, 1)$ \\ \hline 
			$r_{1,n}$ & $n^{-\frac{1}{2}}$ & $n^{-\frac{1}{2}}$ & $n^{2\overline{H}-\frac{3}{2}}$  & $n^{2\overline{H}-\frac{3}{2}}$  \\
			$r_{2,n}$ & $0$ & $n^{\frac{1}{2}-2H}$ & $n^{-\frac{1}{2}}$  & $n^{-\frac{1}{2}}$ \\
			$r_{3,n}$ & $n^{2(\overline{H}-H)-\frac{1}{2}}$ & $n^{2(\overline{H}-H)-\frac{1}{2}}$  & $n^{2\overline{H}-\frac{3}{2}}$ & $n^{2\overline{H}-\frac{3}{2}}$ \\
			$r_{4,n}$ & $n^{-\frac{1}{2}}$ & $n^{-\frac{1}{2}}$ & $n^{2H-\frac{3}{2}}$ & $0$ 
		\end{tabular}
	}
\end{table}

\begin{proof}[Proof of Theorem~\ref{thm:lb}]
For any   $\theta_0 = (H,\sigma^2,\Lambda,\Pi)\in\Theta$, we define $\theta_n$ as in \eqref{eqn:theta-n} and
$r_{i,n} = r_0 n^{\alpha_i(H)}$ as shown in Table \ref{tab:rates-theta-n}.
In view of Proposition \ref{lem:KL-complete}, these rates are chosen  such that $\mathrm{KL}(\theta_n \dmid \theta_0) \leq r_0^2 C(\theta_0)$ . 
Upon setting $r_0$ small enough, we find that
\begin{align}
	\mathrm{KL}(\theta_n \dmid \theta_0) \leq \tfrac{1}{9}. \label{eqn:KL-bounded}
\end{align}
Moreover, from \eqref{eqn:theta-n}, it is simple to derive a lower bound on the errors $\theta_n-\theta_0$, component by component. 
In particular, since $r_{1,n}=o( 1/\log n)$, we have 
\begin{align*}
	|\Lambda(\theta_n)-\Lambda(\theta_0)| = \Omega( r_{1,n} \log n + r_{3,n}), 
	\qquad |\Pi(\theta_n)-\Pi(\theta_0)| = \Omega(r_{1,n} \log n+r_{4,n}),
\end{align*}
where $\Omega$ denotes an asymptotic lower bound, that is, $a_n  = \Omega(b_n)$ if $b_n =  {O}(a_n)$.
The resulting bounds on $\theta_n-\theta_0$  are exactly the rates $R_n(\theta)$ of Theorem \ref{thm:lb} (listed in Table~\ref{tab:rates}).

In order to translate these KL estimates into statistical lower bounds, we follow  \cite[Chapter 2]{Tsybakov2008}. Intuitively speaking,
if   $\theta_n$ satisfies \eqref{eqn:KL-bounded}, we cannot consistently decide whether $\theta_n$ or $\theta_0$ is the true parameter. 
Hence, no estimator can converge towards $\theta_0$ faster than $\theta_n$. 
To make this mathematically precise, let $\widehat{\theta}_n$ be any measurable function of $\{Y_{{i}/{n}}:i=1,\ldots, n\}$. 
Note that $\theta_n\in \mathcal{D}_n(\theta_0)$ for $n$ large enough.
Then, for any $c>0$, 
\begin{align*}
	&\sup_{\theta\in\mathcal{D}_n(\theta_0)} \bbp_\theta\Bigl( |(\widehat{\theta}_n-\theta)_k| \geq c  |(\theta_n-\theta_0)_k| \Bigr) \\
	&\qquad\geq \tfrac{1}{2} \bbp_{\theta_0}\Bigl( |(\widehat{\theta}_n-\theta_0)_k| \geq c |(\theta_n-\theta_0)_k| \Bigr)
	 +  \tfrac{1}{2} \bbp_{\theta_n}\Bigl( |(\widehat{\theta}_n-\theta_n)_k| \geq c |(\theta_n-\theta_0)_k| \Bigr)\\
	&\qquad\geq \frac{1}{2} \bbp_{\theta_0}\Bigl( |(\widehat{\theta}_n-\theta_0)_k| \geq c |(\theta_n-\theta_0)_k|\Bigr)
	 +   \tfrac{1}{2} \bbp_{\theta_0}\Bigl( |(\widehat{\theta}_n-\theta_n)_k| \geq c |(\theta_n-\theta_0)_k| \Bigr) - \tfrac{1}{3}. 
\end{align*}
In the last step, we used that the fact that the total variation distance between $\bbp_{\theta_n}$ and $\bbp_{\theta_0}$ is upper bounded by $\sqrt{\mathrm{KL}(\theta_n \dmid \theta_0)/2}\leq\sqrt{\mathrm{KL}(\theta_n \dmid \theta_0)} \leq \frac{1}{3}$.
Now use the union bound to obtain
\begin{align*}
	&  \bbp_{\theta_0}\Bigl( |(\widehat{\theta}_n-\theta_0)_k| \geq c |(\theta_n-\theta_0)_k| \Bigr)
 +   \bbp_{\theta_0}\Bigl( |(\widehat{\theta}_n-\theta_n)_k| \geq c |(\theta_n-\theta_0)_k| \Bigr) \\
	&\qquad\geq \bbp_{\theta_0} \Bigl(   \lvert(\widehat{\theta}_n-\theta_0)_k\rvert\vee \lvert(\widehat{\theta}_n-\theta_n)_k\rvert  \geq c |(\theta_n-\theta_0)_k|  \Bigr) \\
	&\qquad\geq \bbp_{\theta_0} \bigl( |(\theta_n-\theta_0)_k| \,\geq\, 2c |(\theta_n-\theta_0)_k|  \bigr) = 1,
\end{align*}
where the last equality holds for $c\leq \frac{1}{2}$.
Hence,
\begin{align*}
	\sup_{\theta\in\mathcal{D}_n(\theta_0)} \bbp_\theta \Bigl(|(\widehat{\theta}_n-\theta)_k| \geq c  |(\theta_n-\theta_0)_k| \Bigr) \geq \tfrac{1}{6}.
\end{align*}
Because $|(\theta_n-\theta_0)_k| =\Omega( R_n(\theta_0)_k)$, this establishes the claimed minimax rates.
\end{proof}

\section{Proof of Proposition~\ref{lem:KL-complete}}\label{app:lb}

Let $A(h) = h \wt{\Sigma}_n(\theta_n) + (1-h) \wt{\Sigma}_n(\theta_0)$ and   $\delta_n=\wt{\Sigma}_n(\theta_n)-\wt{\Sigma}_n(\theta_0)$.  By
Taylor expansion,
\begin{align*}
	&\mathrm{KL}(A(1) \dmid A(0)) \\
	&\qquad= \frac{d}{dh} \mathrm{KL}(A(h) \dmid A(0)) \Big|_{h=0} + \frac{1}{2} \int_0^1 \int_0^s \frac{d^2}{dh^2} \mathrm{KL}(A(h) \dmid A(0)) \Big|_{h=v}\, \dd v\,\dd s,
\end{align*}
where
\begin{align*}
	\frac{d}{dh} \mathrm{KL}(A(h) \dmid A(0)) 
	&= \frac{1}{2} \tr  ( A(0)^{-1} \delta_n  ) - \frac{1}{2} \tr  ( A(h)^{-1} \delta_n  ) = \frac{1}{2} \tr  ( [A(0)^{-1}-A(h)^{-1}]   \delta_n   ), \\ 
	\frac{d^2}{dh^2} \mathrm{KL}(A(h) \dmid A(0))
	&= \frac{1}{2} \tr  ( A(h)^{-1} \delta_n A(h)^{-1} \delta_n  ).
\end{align*}
Hence,
\begin{equation} \label{eqn:kl-trace}\begin{split}
		\mathrm{KL}(A(1) \dmid A(0)) 
		&= \frac{1}{4} \int_0^1 \int_0^s \tr  ( A(h)^{-1} \delta_n A(h)^{-1} \delta_n  )\, \dd h\, \dd s   \\
		&\leq \frac{1}{4} \sup_{h\in[0,1]} \tr  ( A(h)^{-1} \delta_n A(h)^{-1} \delta_n  ).\end{split}
\end{equation}
In order to find an upper bound for \eqref{eqn:kl-trace}, we will use the following technical lemma.

\begin{Lemma}\label{lem:trace-inequality}
	Let $B$ be a symmetric matrix and $A$ and $A_0$ be symmetric positive semidefinite matrices such that $A-A_0$ is positive semidefinite, too.
	Then
	\begin{align*}
		\tr(A_0 B A_0 B) \leq \tr (A B A_0 B) = \tr(A_0 B A B) \leq \tr(A B A B).
	\end{align*}
\end{Lemma}
\begin{proof} 
	Denote $C=BA_0 B$, which is symmetric positive semidefinite. 
	Von Neumann's trace inequality yields $\tr ( (A-A_0) C) \geq \sum_{i=1}^n a_i b_{n-i+1}$, where $a_i$ and $b_i$ are the descending eigenvalues of $A-A_0$ and $C$, respectively. 
	Since $a_i, b_i\geq 0$, we conclude that $\tr(A_0 C) \leq \tr(A C)$, which proves the first inequality and, consequently, the
	second.
\end{proof}

To bound \eqref{eqn:kl-trace}, we therefore need a lower bound on  $A(h)$ and an upper bound on $\delta_n$.

\begin{Lemma}\label{lem:Ah-lb}
	For any $\theta_0\in\Theta$ and  $\delta>0$, there exists $c=c(\theta_0,\delta)>0$ such that
	\begin{equation}\label{eq:Ah}
		A(h) \geq c  [ n^{-1} I_n + n^{-2H} \Sigma_n(H-\delta)  ], \qquad h\in[0,1].
	\end{equation}
	Here, $\geq$ denotes the  Loewner partial order on the cone of positive semidefinite matrices.
\end{Lemma}
\begin{proof}[Proof of Lemma \ref{lem:Ah-lb}]
	Since $A(h)$ is a convex combination, it suffices to establish the lower bounds for $A(1)$ and $A(0)$. Note that by definition,
	\begin{align*}  &\sqrt{\Pi\sigma^2}b(H)n^{-2\ov H} \Sigma_n(\ov H)_{jk}\\
		&\quad=   \cov\biggl(\lvert \si\rvert\frac{\Delta^n_j B}{ \Den^{1/2}},\sqrt{\Pi}\frac{\Delta^n_k B(H)}{\Den^H}\biggr)+\cov\biggl(\lvert \si\rvert\frac{\Delta^n_k B}{\Den^{1/2}},\sqrt{\Pi}\frac{\Delta^n_j B(H)}{\Den^H}\biggr),
	\end{align*}
	where $B(H)_t=\int_{-\infty}^t h_H(t,s)\,\dd B_s$. Since $XY^T+YX^T\leq XX^T+YY^T$ for $X,Y\in\R^{n\times n}$, it follows that $\sqrt{\Pi\sigma^2}b(H)n^{-2\ov H} \Sigma_n(\ov H)\leq  \si^2n^{-1}I_n+\Pi n^{-2H}\Sigma_n(H)$ and therefore, by \eqref{eq:sig}, 
	$$\wt\Sig_n(\theta) \geq \biggl(1-\frac{\lvert \La\rvert}{\sqrt{\Pi\sigma^2}}\biggr)(\si^2n^{-1}I_n+\Pi n^{-2H}\Sigma_n(H)).$$
	If we apply this  to $\theta=\theta_0$, we immediately obtain \eqref{eq:Ah} for $A(0)=\wt\Sigma_n(\theta_0)$. To derive the estimate for $A(1)=\wt\Sigma_n(\theta_n)$, we apply the above to $\theta=\theta_n$, which yields
	\begin{align*}
		A(1)\geq \frac12\biggl(1-\frac{\lvert \La\rvert}{\sqrt{\Pi\sigma^2}}\biggr)(\si^2n^{-1}I_n+\Pi n^{-2H}\Sigma_n(H+r_{1,n}))
	\end{align*}
	for all sufficiently large $n$. Since $\La^2<\Pi\si^2$, this implies \eqref{eq:Ah} for $h=1$.
\end{proof}

We proceed to finding an upper bound on $\delta_n$.

\begin{Lemma}\label{lem:Delta-ub}
	For any $\theta_0\in\Theta$ and  $\delta>0$, there exists  $C=C(\theta_0,\delta)>0$ such that $	\delta_n \leq C B_n(\theta_0,\delta)$, where
	\begin{equation*}
		B_n(\theta_0,\delta)  = r_{2,n}n^{-1} I_n + (r_{1,n}+r_{4,n})n^{-2H} \Sigma_n(H+\delta)  + (r_{1,n}+r_{3,n}) n^{-2\overline{H}} \Sigma_n(\overline{H}+\delta).
	\end{equation*}
\end{Lemma}
\begin{proof}[Proof of Lemma \ref{lem:Delta-ub}]
	By \cite[Proposition 7.2.9]{samorodnitsky1994stable}, we have that
	\begin{equation*}
		\Sigma_n(H) = T_n(f_H),
	\end{equation*}
	where 
	\beq\label{eq:spectral} T_n(f)_{j,k}=\int_{-\pi}^\pi f(\lambda) e^{- {i} \lambda \lvert j-k\rvert}\, \dd\lambda\eeq
	and the spectral density $f_H$ is given by
	\begin{equation*}
		f_H(\lambda) 
		=  \frac{\Gamma(2H+1)\sin(\pi H)}{\pi}(1-\cos\lambda) \sum_{k\in\bbz} |\lambda + 2k\pi|^{-2H-1}.
	\end{equation*}
	In particular,  by \eqref{eq:wtsigma}, 
	\begin{multline*}
		\delta_n=\wt{\Sigma}_n(\theta_n) - \wt{\Sigma}_n(\theta) 
		= r_{2,n} n^{-1} I_n 
		+ \Pi n^{-2H}  [(1+r_{4,n})\Sigma_n(H+r_{1,n}) - \Sigma_n(H) ] \\
		+ \Lambda b(H) n^{-2\overline{H}}\, [(1+r_{3,n}) \Sigma_n(\overline{H} + \tfrac{r_{1,n}}{2}) -  \Sigma_n(\overline{H}) ]= T_n(\wt{g}_n),
	\end{multline*}
	where 
	\begin{align*}
		\wt{g}_n 
		&=   r_{2,n} \tfrac{n^{-1}}{2\pi} +  \Pi n^{-2H}[(1+r_{4,n})f_{H+r_{1,n}} -f_{H} ]  	+ \Lambda b(H) n^{-2\overline{H}} [ (1+r_{3,n}) f_{\overline{H} + \frac{r_{1,n}}{2}} - f_{\overline{H}}] \\
		&\leq \mathtoolsset{multlined-width=0.9\displaywidth} \begin{multlined}[t]  C(\theta_0) \Bigl[ r_{2,n} n^{-1} + n^{-2H} r_{4,n} |f_{H+r_{1,n}}| + n^{-2H} |f_{H+r_{1,n}} -f_{H}| \\
			+n^{-2\overline{H}} r_{3,n} |f_{\overline{H} + \frac{r_{1,n}}{2}}| + n^{-2\overline{H}} |f_{\overline{H} + \frac{r_{1,n}}{2}} - f_{\overline{H}}| \Bigr]\end{multlined}
	\end{align*}
	and $C(\theta_0)$ may change its value from line to line.
	For large $n$, we have that
	\begin{align*}
		|f_{H+r_{1,n}}(\lambda)| 
		&\leq C(\theta_0) |\lambda|^{1-2H-2r_{1,n}}
		\leq C(\theta_0) |\lambda|^{1-2H-2\delta}, \\
		|f_{\overline{H} + \frac{r_{1,n}}{2}} (\lambda)|
		&\leq C(\theta_0) |\lambda|^{1-2\overline{H}-r_{1,n}}
		\leq C(\theta_0) |\lambda|^{1-2\overline{H}-2\delta},\\
		|f_{H+r_{1,n}}(\lambda) - f_{H}(\lambda) | 
		&\leq C(\theta_0) r_{1,n} |\lambda|^{1-2H - 2r_{1,n}} \lvert \log{\lvert \la\rvert}\rvert
		\leq C(\theta_0) r_{1,n} |\lambda|^{1-2H - 2\delta} , \\
		|f_{\overline{H} + \frac{r_{1,n}}{2}} (\lambda) - f_{\overline{H}} (\lambda)|
		&\leq C(\theta_0) r_{1,n} |\lambda|^{1-2\overline{H} - r_{1,n}}\lvert \log{\lvert \la\rvert}\rvert
		\leq C(\theta_0) r_{1,n} |\lambda|^{1-2\overline{H} - 2\delta}
	\end{align*}
	for all $\lambda\in[-\pi,\pi]$.
	Hence, $\wt{g}_n(\lambda) \leq C(\theta_0)\widehat{g}_n(\lambda)$ for all $\lambda\in[-\pi,\pi]$, where
	\begin{align*}
		\wt{g}_n\leq C(\theta_0)\Bigl[ r_{2,n}\tfrac{n^{-1}}{2\pi} + (r_{1,n}+r_{4,n}) n^{-2H} f_{H+\delta} + (r_{1,n} + r_{3,n})n^{-2\overline{H}} f_{\overline{H}+\delta}\Bigr],
	\end{align*}
	which yields the claim.
\end{proof}

\begin{Lemma}\label{lem:trace-limit}
	Let $g\colon[-\pi,\pi]\to \R$ be a symmetric function such that $g(\lambda) =  {O}(|\lambda|^{-\beta})$ as $\la\to0$ for some $\beta\in[0,1)$.
	Then, for all $\delta>0$ and  $H\in(0,1)$,
	\begin{align*}
		\tr( T_n(g) T_n(g) ) &= {O} (  n \vee  n^{2\beta + \delta}  ),\\
		\tr\bigl( \Sigma_n(H)^{-1} T_n(g) \Sigma_n(H)^{-1} T_n(g) \bigr)
		& = {O} ( n\vee n^{2(\beta-2H+1) + \delta}   ).
	\end{align*}
\end{Lemma}
\begin{proof}[Proof of Lemma \ref{lem:trace-limit}] We apply \cite[Theorem 5 in the full version]{Lieberman2012a} to 
	the spectral densities $f=f_H$ and $g$ with $\alpha = \alpha(\theta)= 2H-1$ and $\beta$ as above. 
\end{proof}

\begin{proof}[Proof of Proposition \ref{lem:KL-complete}]
	By \eqref{eqn:kl-trace} and Lemmas \ref{lem:Ah-lb} and \ref{lem:Delta-ub}, we obtain for   $\delta>0$,
	\begin{align*}
		\mathrm{KL} (\theta_n \dmid \theta_0) \leq  \begin{cases}C(\theta_0)
			n^2 \tr( B_n B_n ) & \text{if } H>\frac{1}{2}, \\
			C(\theta_0)	n^{4H} \tr\bigl( \Sigma_n(H-\delta)^{-1}B_n \Sigma_n(H-\delta)^{-1} B_n  \bigr) & \text{if } H<\frac{1}{2},
		\end{cases}
	\end{align*}
	where $B_n=B_n(\theta_0, \delta)$ as in Lemma \ref{lem:Delta-ub}.
	
	If $H>\frac{1}{2}$, the Cauchy--Schwarz inequality yields
	\begin{align*}
		\tr(B_n B_n) 
		&\leq 3 r_{2,n}^2 n^{-2} \tr ( I_n^2 )  + 3 (r_{1,n}+r_{4,n})^2 n^{-4H} \tr(\Sigma_n(H+\delta)^2 ) \\
		&\quad+ 3 (r_{1,n}+r_{3,n})^2 n^{-4\overline{H}} \tr( \Sigma_n(\overline{H}+\delta)^2 ).
	\end{align*}
	Clearly, $\tr(I_n^2)= n$. Moreover, since
	$\Sigma_n(\overline{H}+\delta)$ and $\Sigma_n(H+\delta)$ satisfy the conditions of Lemma \ref{lem:trace-limit} with $\beta = 2\overline{H}+2\delta-1$ and $\beta = 2H+2\delta-1$, respectively, we have that  
	\begin{equation*}
		\tr( \Sigma_n(\overline{H}+\delta)^2 )
		=  {O} (  n\vee n^{4\overline{H}+5\delta -2}) ,\qquad
		\tr( \Sigma_n(H+\delta)^2 )
		=  {O} (  n\vee n^{4H+5\delta -2}) .
	\end{equation*}
	Since $\delta>0$ was arbitrary, we find that $\tr(B_n B_n) =  {O}(z_n)$, where
	\begin{equation*}
		z_n 
		= r_{2,n}^2 n^{-1} + (r_{1,n}+r_{4,n})^2   (n^{1-4H} \vee  n^{\delta-2}  )  + (r_{1,n}+r_{3,n})^2   ( n^{1-4\overline{H}}\vee n^{\delta-2}  ).
	\end{equation*}
	This establishes the  KL upper bound for the cases $H\in(\frac{1}{2}, \frac{3}{4})$ and $H\in[\frac{3}{4}, 1)$.
	
	If $H<\frac{1}{2}$,
	we may again apply Lemma \ref{lem:trace-limit} to find that
	\begin{align*}
		\tr\bigl[  (\Sigma_n(H-\delta)^{-1} I_n  )^2 \bigr] 
		&=  {O} (  n\vee n^{2-4H+5\delta}), \\
		\tr\bigl[  ( \Sigma_n(H-\delta)^{-1} \Sigma_n(\overline{H}+\delta)  )^2 \bigr] 
		&=  {O} (  n\vee  n^{4(\overline{H}-H) +5\delta } ), \\
		\tr\bigl[  ( \Sigma_n(H-\delta)^{-1} \Sigma_n(H+\delta)  )^2 \bigr] 
		&=  {O} (  n\vee n^{5\delta}).
	\end{align*}
	Since $\delta>0$ was arbitrary, we find that $\tr[ \Sigma_n(H-\delta)^{-1}B_n \Sigma_n(H-\delta)^{-1} B_n  ] = {O}(w_n)$, where
	\begin{align*}
		w_n 
		&= r_{2,n}^2   (n^{-1}\vee n^{-4H+\delta} )  + (r_{1,n}+r_{3,n})^2   ( n^{1-4\overline{H}}\vee n^{-4H+\delta}  ) \\
		&\quad	+ (r_{1,n}+r_{4,n})^2   ( n^{1-4H}\vee n^{ - 4H+\delta}  ).
	\end{align*}
	This yields the KL upper bound for the remaining cases $H\in(\frac{1}{4},\frac{1}{2})$ and $H\in(0,\frac{1}{4}]$.
\end{proof}

\begin{appendix}

\section{Central limit theorem for general variation functionals}\label{app:CLT}

In this section, we state and prove a CLT for general variation functionals of multivariate mixed semimartingale processes. To this end, we consider a $d$-dimensional mixed semimartingale of the form
\begin{equation} \label{mix:SM:mod}
	Y_t = Y_0+\int_0^t a_s\,\dd s + \int_{0}^{t} \si_s \, \dd B_s + \int_{0}^{t} g(t-s) \rho_s \, \dd B_s,
\end{equation}
where now $B$ is a $d'$-dimensional Brownian motion, $a$ (resp., $\si$ and $\rho$) is   $\R^d$-valued (resp., $\bbr^{d \times d'}$-valued) and predictable and  $g \colon \bbr \lra \bbr$ is given by \eqref{kernel:g}. Since $d'$ may be larger than $d$, \eqref{mix:SM:mod} includes the case where the martingale and the fractional part of $Y$ are driven by correlated (and not necessarily identical) Brownian motions. In fact, 
the situation considered in Section~\ref{sec:CLT}  can be embedded into the current setting by defining $d'=2$ and  
\beq\label{eq:spec}
B^{\text{\eqref{mix:SM:mod}}}=(B,B'),\quad	\si^{\text{\eqref{mix:SM:mod}}} = (\si,0),\quad 	\rho^{\text{\eqref{mix:SM:mod}}} = (\rho,\rho')
\eeq
(the superscript stands for ``from Equation~\eqref{mix:SM:mod}''). In particular, Theorem~\ref{cor}  then becomes a special case of Theorems~\ref{thm:CLT:mixedSM} and \ref{thm:CLT2} below.

Let $f \colon \bbr^{d \times L} \lra \bbr^M$ for some $L, M \in \bbn$. For $Y$ and similarly for other $d$-dimensional processes, we define
\begin{equation} \label{not:tensor}
	\Delta_i^n  Y =  Y_{i \Delta_n}-  Y_{(i-1) \Delta_n} \in \bbr^d,\qquad
	\un{\Delta}^n_i  Y = (\Delta_i^n  Y, \Delta_{i+1}^n  Y, \ldots, \Delta_{i+L-1}^n  Y) \in \bbr^{d \times L}.
\end{equation}
Our goal is to formulate and prove a CLT for normalized variation functionals of the form
\beq\label{eq:Vnf} V^n_f(Y,t)=\Den\sum_{i=1}^{[t/\Den]-L+1} f\biggl(\frac{\un\Delta^n_i Y}{\Den^{H\wedge (1/2)}}\biggr),\qquad t\geq0,\eeq
where $[\cdot]$ denotes the floor function.
In what follows, $\Vert \cdot \Vert$ denotes the Euclidean norm in $\bbr^n$ for any $n \in \bbn$. Also, if $z$ is some matrix in $\bbr^{n \times m}$ for any $n,m \in \bbn$, then $\Vert z \Vert$ is defined by viewing $z$ as a vector in $\bbr^{nm}$. We introduce the following assumptions. 

\settheoremtag{(CLT')}
\begin{Assumption} \label{Ass:B} Consider the process $Y$ from \eqref{mix:SM:mod} and recall $N(H)$ from \eqref{eq:NH}.
	We make the following assumptions:
	\begin{enumerate}
		\item
		The function $f \colon \bbr^{d \times L} \lra \bbr^M$ is even (i.e., satisfies $f(-x)=f(x)$ for all $x$) and belongs to $C^{2N(H)+1}(\bbr^{d \times L}, \bbr^M)$, with all partial derivatives  up to order $2N(H)+1$ (including $f$ itself) being of polynomial growth. 
		\item  The kernel $g$ is of the form \eqref{kernel:g} where $H\in (0,1)\setminus\{\frac12\}$ and $g_0\in C^1(\bbr)$ satisfies $g_0(x)=0$ for all $x\leq 0$. 
		\item  The drift process $a$ is $d$-dimensional, locally bounded, $\F$-adapted and càdlàg. Moreover, $B$ is a $d'$-dimensional standard $\F$-Brownian motion.
		\item If $H>\frac12$, the volatility process $\si$ takes the form
		\begin{equation} \label{repr:si}
			\si_t = \si^{(0)}_t + \int_{0}^{t} \wti{\si}_s \, \dd B_s, \qquad t \geq 0,
		\end{equation}
		where
		\begin{enumerate}
			\item $\si^{(0)}$ is an $\F$-adapted locally bounded $\bbr^{d \times d'}$-valued process such that for all $T>0$, there are  $\ga \in  ( \frac{1}{2}, 1  ]$ and $K_1\in(0,\infty)$ with
			\begin{equation} \label{mom:ass:si:rho:1}
				\bbe  [ 1\wedge  \Vert \si_t^{(0)} - \si_s^{(0)}   \Vert ]  \leq K_1 \vert t - s \vert^{\ga},\qquad s,t\in[0,T];
			\end{equation}
			\item  $\wti{\si}$ is an $\F$-adapted locally bounded $\bbr^{d \times d' \times d'}$-valued process such that for all $T>0$, there are  $\eps \in(0,1)$ and $K_2\in(0,\infty)$ with
			\begin{equation} \label{reg:ass:si:ti}
				\bbe  [ 1\wedge \Vert \wti{\si}_t  - \wti{\si}_s  \Vert    ]  \leq K_2 \vert t -s \vert^{\eps}, \qquad s,t\in[0,T].
			\end{equation}
		\end{enumerate}	
		If $H<\frac12$, we have \eqref{repr:si} but with $\si$, $\si^{(0)}$ and $\wt \si$ replaced by  $\rho$ and some processes $\rho^{(0)}$ and $\wt\rho$ satisfying   conditions analogous to \eqref{mom:ass:si:rho:1} and \eqref{reg:ass:si:ti}.
		\item  If $H>\frac12$, the process $\rho$ is   $\F$-adapted, locally bounded and $\bbr^{d \times d'}$-valued. Moreover, for all $T>0$, there is $K_3\in(0,\infty)$ such that
		\begin{equation} \label{reg:cond:rho:B3}
			\bbe  [   1\wedge \Vert \rho_t - \rho_s   \Vert  ]  \leq K_3 \vert t -s \vert^{\frac{1}{2}}, \qquad s,t\in[0,T].
		\end{equation}
		If $H<\frac12$, we have the same condition but with $\rho$ replaced by $\si$.
	\end{enumerate}
\end{Assumption}

To state the CLT, we need more additional notation.
For  suitable $v=(v_{k \ell, k' \ell'})_{k, k',\ell,\ell' = 1}^{d,d,L,L},q=(q_{k \ell, k' \ell'})_{k, k',\ell,\ell' = 1}^{d,d,L,L} \in\R^{(d\times L)\times(d\times L)}$, define
\begin{equation} \label{def:uRhof1f2}\begin{split}
		\mu_f(v)&=(\E[f_{m}(Z)])_{m=1}^M \in\R^M,\\	\ga_f(v,q) &= (\cov ( f_{m}(Z), f_{m'}(Z')  ))_{m,m'=1}^M \in\R^{M\times M}, \quad\ga_f(v) =\ga_f(v,v),  \end{split}
\end{equation}
where $(Z,Z')\in (\R^{d\times L})^2$ is multivariate normal  with mean $0$ and $\cov(Z_{k \ell},Z_{k' \ell'})=\cov(Z'_{k \ell},Z'_{k' \ell'})=v_{k \ell, k' \ell'}$ and $\cov(Z_{k \ell},Z'_{k' \ell'})=q_{k\ell,k'\ell'}$. Furthermore, given a multi-index $\chi = (\chi_{k \ell, k' \ell'})_{k, k',\ell,\ell' = 1}^{d,d,L,L} \in 
\bbn_0^{(d\times L)\times(d\times L)}$, we define
\begin{equation} \label{mult:ind:1}
	\vert \chi \vert = \sum_{k,k' = 1}^d \sum_{\ell,\ell' = 1}^{L} \chi_{k\ell,k\ell'}, \quad \chi! = \prod_{k,k' = 1}^d \prod_{\ell,\ell' = 1}^{L} \chi_{k\ell,k'\ell'}!,\quad
	v^{\chi} = \prod_{k,k' = 1}^d \prod_{\ell,\ell' = 1}^{L} v_{k\ell,k'\ell'}^{\chi_{k\ell, k' \ell'}}
\end{equation}
and the partial derivatives 
\begin{equation} \label{mult:ind:2}
	\pd^{\chi} \mu_f(v)  = 
	\frac{\pd^{\vert \chi \vert} \mu_f}{\pd v_{11,11}^{{\chi}_{11,11}} \cdots \pd v_{dL,dL}^{{\chi}_{dL,dL}}}(v)\in\R^M.
\end{equation}
For $s\geq0$, we also define 
\begin{equation} \label{mat:lim:thms}
	\begin{split}
		c(s)_{k \ell, k' \ell'} & = (\si_s \si_s^T)_{k k'} \ind_{\{\ell= \ell'\}},  \\
		\pi_r(s)_{k \ell, k' \ell'}  &= (\rho_s \rho_s^T)_{k k'} \Ga^H_{\vert \ell - \ell'+r \vert},\quad \pi(s)_{k \ell, k' \ell'} =\pi_0(s)_{k \ell, k' \ell'},\\
		\la(s)_{k \ell, k' \ell'}  &= 2^{-\bone_{\{\ell=\ell'\}}} (\si_s   \rho_s^T\bone_{\{\ell\leq \ell'\}}+  \rho_s\si_s^T\bone_{\{\ell\geq \ell'\}})_{k k'} \Phi^H_{\vert \ell - \ell' \vert}
	\end{split}
\end{equation}
for all $k, k' \in \{1, \ldots, d\}$ and $\ell,\ell' \in \{1, \ldots, L\}$, where $\Ga^H_r$ and $\La^H_r$ are defined in \eqref{num:Ga} and \eqref{nu:La}, respectively.

\begin{Theorem} \label{thm:CLT:mixedSM}
	If Assumption~\ref{Ass:B} holds with $H>\frac12$, then
	\begin{equation} \label{CLT:1}
		\Delta_n^{- \frac{1}{2}} \Bigl\{ V^n_f(Y,t)
		- V_f(Y,t)- \cala^{\prime n}_t \Bigr \} 
		\stackrel{\mathrm{st}}{\Longrightarrow} \mathcal{Z}',
	\end{equation}
	where $
	V_f(Y,t)=	\int_{0}^{t} \mu_f (c(s)) \, \dd s $
	and
	\begin{equation}\label{eq:Ant}
		\cala^{\prime n}_t=  
		\sum_{j=1}^{N(H)} 
		\sum_{\lvert \chi \rvert = j} \frac{1}{\chi!} 
		\int_{0}^{t} \pd^{\chi} \mu_f (c(s)) \Bigl( \Den^{H-\frac12} \la(s) 
		+ \Den^{2H-1} \pi(s) \Bigr)^{\chi} \, \dd s 
	\end{equation}
	and
	$\mathcal{Z}' = (\mathcal{Z}'_t)_{t \geq 0}$ is an $\bbr^M$-valued process, defined on a very good filtered extension $(\ov{\Om}, \ov{\mathcal{F}}, (\ov{\mathcal{F}}_t)_{t \geq 0}, \ov{\mathbb{P}})$ of the original probability space $(\Om, \mathcal{F}, (\mathcal{F}_t)_{t \geq 0}, \mathbb{P})$  that conditionally on  $\mathcal{F}$ is a centered Gaussian process with independent increments and  covariance function  
	\begin{equation} \label{cov:fct:lim:CLT}
		\calc'_t = (\ov{\bbe}[\calz^{\prime m}_t \calz^{\prime m'}_t \mid \calf])_{m,m'=1}^M	  =\int_{0}^{t}  \ga_f (c(s)) \, \dd s.
	\end{equation}
\end{Theorem}

\begin{Theorem}\label{thm:CLT2}
	If Assumption~\ref{Ass:B} holds with $H<\frac12$, then
	\begin{equation} \label{CLT:2}
		\Delta_n^{- \frac{1}{2}} \biggl\{ V^n_f(Y,t)
		- \int_{0}^{t} \mu_f (\pi(s)) \, \dd s - \cala^{n}_t \biggr \} 
		\stackrel{\mathrm{st}}{\Longrightarrow} \mathcal{Z},
	\end{equation}
	where
	\begin{equation}\label{eq:Ant-2}
		\cala^{n}_t=  
		\sum_{j=1}^{N(H)} 
		\sum_{\lvert \chi \rvert = j} \frac{1}{\chi!} 
		\int_{0}^{t} \pd^{\chi} \mu_f (\pi(s)) \Bigl( \Den^{\frac12-H} \la(s) 
		+ \Den^{1-2H} c(s) \Bigr)^{\chi} \, \dd s 
	\end{equation}
	and
	$\mathcal{Z} = (\mathcal{Z}_t)_{t \geq 0}$ is an $\bbr^M$-valued process  defined on $(\ov{\Om}, \ov{\mathcal{F}}, (\ov{\mathcal{F}}_t)_{t \geq 0}, \ov{\mathbb{P}})$ that conditionally on  $\mathcal{F}$ is a centered Gaussian process with independent increments and  covariance function  
	\begin{equation} \label{cov:fct:lim:CLT-2}\begin{split}
			\calc_t &= (\ov{\bbe}[\calz^{m}_t \calz^{m'}_t \mid \calf])_{m,m'=1}^M\\	 
			& =\int_{0}^{t}  \biggl\{\ga_f (c(s))+\sum_{r=1}^\infty (\ga_f(\pi(s),\pi_r(s)) + \ga_f(\pi(s),\pi_r(s))^T)\biggr\} \, \dd s.\end{split}
	\end{equation}
\end{Theorem}

The proof of the two results is given in Appendices \ref{sec:proof}--\ref{app:B}. 
Apart from the term $\cala^{\prime n}_t$,  Theorem~\ref{thm:CLT:mixedSM} is exactly the CLT of semimartingale variation functionals (see \cite{Jacod12}). Similarly, if we ignore $\cala^n_t$, Theorem~\ref{thm:CLT2} is the CLT for variation functionals of a fractional process with roughness parameter $H$ (cf.\ \cite{BN11,Corcuera13}). The two processes $\cala^{\prime n}_t$ and $\cala^{n}_t$ respectively play the role of higher-order bias terms (for the estimation of the integrated volatility or noise volatility functional). However,  these processes are also key to  identifying  all other quantities of interest in $Y$. If $\la\equiv0$ and $H<\frac12$, Theorem~\ref{thm:CLT2} reduces to \cite[Theorem~3.1]{Chong21}. If $\la \not\equiv0$, additional terms involving $\la$ appear in both $\cala^{\prime n}_t$ and $\cala^{n}_t$.

\end{appendix}

\section*{Acknowledgments}

We   would like to thank the Associate Editor and four referees for their careful reading of the paper. Their constructive comments have  led to significant improvements of the paper. CC is partially supported by ECS project 26301724.
TD is partially supported by the DFG, project number KL 1041/7-2.

\bibliographystyle{abbrv}
\bibliography{biblio}

\begin{thebibliography}{10}

\bibitem{AitSahalia09}
Y.~A\"{\i}t-Sahalia and J.~Jacod.
\newblock Testing for jumps in a discretely observed process.
\newblock {\em Ann. Statist.}, 37(1):184--222, 2009.

\bibitem{AitSahalia14}
Y.~A{\"i}t-Sahalia and J.~Jacod.
\newblock {\em High-Frequency Financial Econometrics}.
\newblock Princeton University Press, Princeton, 2014.

\bibitem{AY09}
Y.~A\"{\i}t-Sahalia and J.~Yu.
\newblock High frequency market microstructure noise estimates and liquidity
  measures.
\newblock {\em Ann. Appl. Stat.}, 3(1):422--457, 2009.

\bibitem{amblard2012basic}
P.-O. Amblard, J.-F. Coeurjolly, F.~Lavancier, and A.~Philippe.
\newblock Basic properties of the multivariate fractional {B}rownian motion.
\newblock In {\em Self-similar processes and their applications}, volume~28 of
  {\em S\'emin. Congr.}, pages 63--84. Soc. Math. France, Paris, 2013.

\bibitem{BN11}
O.~Barndorff-Nielsen, J.~Corcuera, and M.~Podolskij.
\newblock Multipower variation for {B}rownian semistationary processes.
\newblock {\em Bernoulli}, 17(4):1159--1194, 2011.

\bibitem{BN05}
O.~E. Barndorff-Nielsen and N.~Shephard.
\newblock Power variation and time change.
\newblock {\em Teor. Veroyatn. Primen.}, 50(1):115--130, 2005.

\bibitem{brouste2018local}
A.~Brouste and M.~Fukasawa.
\newblock Local asymptotic normality property for fractional {{Gaussian}} noise
  under high-frequency observations.
\newblock {\em Ann. Statist.}, 46(5):2045--2061, 2018.

\bibitem{Cai16}
C.~Cai, P.~Chigansky, and M.~Kleptsyna.
\newblock Mixed {G}aussian processes: a filtering approach.
\newblock {\em Ann. Probab.}, 44(4):3032--3075, 2016.

\bibitem{chen2022preaveraging}
D.~Chen, Y.~Cheng, C.~H. Chong, P.~Gentine, W.~Jia, B.~Monier, and S.~Shen.
\newblock Pre-averaging fractional processes contaminated by noise, with an
  application to turbulence.
\newblock {\em arXiv:2212.00867}, 2022.

\bibitem{Chen20}
D.~Chen, P.~A. Mykland, and L.~Zhang.
\newblock The five trolls under the bridge: principal component analysis with
  asynchronous and noisy high frequency data.
\newblock {\em J. Amer. Statist. Assoc.}, 115(532):1960--1977, 2020.

\bibitem{Cheridito01}
P.~Cheridito.
\newblock Mixed fractional {B}rownian motion.
\newblock {\em Bernoulli}, 7(6):913--934, 2001.

\bibitem{chigansky2022a}
P.~Chigansky and M.~Kleptsyna.
\newblock Estimation of the {{Hurst}} parameter from continuous noisy data.
\newblock {\em arXiv:2205.11092}, 2022.

\bibitem{Chong20}
C.~Chong.
\newblock High-frequency analysis of parabolic stochastic {PDE}s.
\newblock {\em Ann. Statist.}, 48(2):1143--1167, 2020.

\bibitem{Chong20a2}
C.~Chong.
\newblock High-frequency analysis of parabolic stochastic {PDE}s with
  multiplicative noise.
\newblock {\em arXiv:1908.04145}, 2020.

\bibitem{Chong21}
C.~H. Chong, T.~Delerue, and G.~Li.
\newblock When frictions are fractional: Rough noise in high-frequency data.
\newblock {\em arXiv:2106.16149}, 2022.

\bibitem{CHLRS24}
C.~H. Chong, M.~Hoffmann, Y.~Liu, M.~Rosenbaum, and G.~Szymanski.
\newblock Statistical inference for rough volatility: {C}entral limit theorems.
\newblock {\em Ann. Appl. Probab.}, 34(3):2600--2649, 2024.

\bibitem{CR98}
F.~Comte and E.~Renault.
\newblock Long memory in continuous-time stochastic volatility models.
\newblock {\em Math. Finance}, 8(4):291--323, 1998.

\bibitem{Corcuera13}
J.~M. Corcuera, E.~Hedevang, M.~S. Pakkanen, and M.~Podolskij.
\newblock Asymptotic theory for {B}rownian semi-stationary processes with
  application to turbulence.
\newblock {\em Stochastic Process. Appl.}, 123(7):2552--2574, 2013.

\bibitem{Delattre97}
S.~Delattre and J.~Jacod.
\newblock A central limit theorem for normalized functions of the increments of
  a diffusion process, in the presence of round-off errors.
\newblock {\em Bernoulli}, 3(1):1--28, 1997.

\bibitem{Diebold13}
F.~X. Diebold and G.~Strasser.
\newblock On the correlation structure of microstructure noise: A financial
  economic approach.
\newblock {\em Rev. Econ. Stud.}, 80(4):1304--1337, 03 2013.

\bibitem{Dozzi15}
M.~Dozzi, Y.~Mishura, and G.~Shevchenko.
\newblock Asymptotic behavior of mixed power variations and statistical
  estimation in mixed models.
\newblock {\em Stat. Inference Stoch. Process.}, 18(2):151--175, 2015.

\bibitem{Gatheral18}
J.~Gatheral, T.~Jaisson, and M.~Rosenbaum.
\newblock Volatility is rough.
\newblock {\em Quant. Finance}, 18(6):933--949, 2018.

\bibitem{Hansen06}
P.~R. Hansen and A.~Lunde.
\newblock Realized variance and market microstructure noise.
\newblock {\em J. Bus. Econom. Statist.}, 24(2):127--218, 2006.

\bibitem{hurst1951long}
H.~E. Hurst.
\newblock Long-term storage capacity of reservoirs.
\newblock {\em Trans. Am. Soc. Civil Eng.}, 116(1):770--799, 1951.

\bibitem{Jacod17}
J.~Jacod, Y.~Li, and X.~Zheng.
\newblock Statistical properties of microstructure noise.
\newblock {\em Econometrica}, 85(4):1133--1174, 2017.

\bibitem{Jacod12}
J.~Jacod and P.~Protter.
\newblock {\em Discretization of Processes}, volume~67 of {\em Stochastic
  Modelling and Applied Probability}.
\newblock Springer, Heidelberg, 2012.

\bibitem{JS18}
J.~Jacod and M.~S{\o}rensen.
\newblock A review of asymptotic theory of estimating functions.
\newblock {\em Stat. Inference Stoch. Process.}, 21(2):415--434, 2018.

\bibitem{Jacod14}
J.~Jacod and V.~Todorov.
\newblock Efficient estimation of integrated volatility in presence of infinite
  variation jumps.
\newblock {\em Ann. Statist.}, 42(3):1029--1069, 2014.

\bibitem{LTWW94}
W.~Leland, M.~Taqqu, W.~Willinger, and D.~Wilson.
\newblock On the self-similar nature of {E}thernet traffic (extended version).
\newblock {\em IEEE/ACM Trans. Netw.}, 2(1):1--15, 1994.

\bibitem{LX16}
J.~Li and D.~Xiu.
\newblock Generalized method of integrated moments for high-frequency data.
\newblock {\em Econometrica}, 84(4):1613--1633, 2016.

\bibitem{Li15}
Y.~Li and P.~A. Mykland.
\newblock Rounding errors and volatility estimation.
\newblock {\em J. Financ. Econom.}, 13(2):478--504, 2015.

\bibitem{Lieberman2012a}
O.~Lieberman, R.~Rosemarin, and J.~Rousseau.
\newblock Asymptotic theory for maximum likelihood estimation of the memory
  parameter in stationary {{Gaussian}} processes.
\newblock {\em Econom. Theory}, 28(2):457--470, 2012.

\bibitem{MP22}
F.~Mies and M.~Podolskij.
\newblock Estimation of mixed fractional stable processes using high-frequency
  data.
\newblock {\em Ann. Statist.}, 51(5):1946--1964, 2023.

\bibitem{MRR02}
T.~Mikosch, S.~Resnick, H.~Rootz\'{e}n, and A.~Stegeman.
\newblock Is network traffic approximated by stable {L}\'{e}vy motion or
  fractional {B}rownian motion?
\newblock {\em Ann. Appl. Probab.}, 12(1):23--68, 2002.

\bibitem{Mishura08}
Y.~S. Mishura.
\newblock {\em Stochastic Calculus for Fractional {B}rownian Motion and Related
  Processes}, volume 1929 of {\em Lecture Notes in Mathematics}.
\newblock Springer-Verlag, Berlin, 2008.

\bibitem{MLS97}
F.~J. Molz, H.~H. Liu, and J.~Szulga.
\newblock Fractional {B}rownian motion and fractional {G}aussian noise in
  subsurface hydrology: {A} review, presentation of fundamental properties, and
  extensions.
\newblock {\em Water Resour. Res.}, 33(10):2273--2286, 1997.

\bibitem{NW94}
W.~K. Newey and K.~D. West.
\newblock Automatic lag selection in covariance matrix estimation.
\newblock {\em Rev. Econom. Stud.}, 61(4):631--653, 1994.

\bibitem{Nual}
D.~Nualart.
\newblock {\em The {M}alliavin Calculus and Related Topics}.
\newblock Probability and its Applications (New York). Springer-Verlag, Berlin,
  second edition, 2006.

\bibitem{pandey1998multifractal}
G.~Pandey, S.~Lovejoy, and D.~Schertzer.
\newblock Multifractal analysis of daily river flows including extremes for
  basins of five to two million square kilometres, one day to 75 years.
\newblock {\em J. Hydrol.}, 208(1-2):62--81, 1998.

\bibitem{Rosenbaum09}
M.~Rosenbaum.
\newblock Integrated volatility and round-off error.
\newblock {\em Bernoulli}, 15(3):687--720, 2009.

\bibitem{samorodnitsky1994stable}
G.~Samorodnitsky and M.~S. Taqqu.
\newblock {\em Stable {N}on-{G}aussian {R}andom {P}rocesses}.
\newblock Stochastic Modeling. Chapman \& Hall, New York, 1994.
\newblock Stochastic models with infinite variance.

\bibitem{Tsybakov2008}
A.~B. Tsybakov.
\newblock {\em Introduction to Nonparametric Estimation}.
\newblock Springer Series in Statistics. Springer, New York, 2009.
\newblock Revised and extended from the 2004 French original, Translated by
  Vladimir Zaiats.

\bibitem{vanZanten07}
H.~van Zanten.
\newblock When is a linear combination of independent f{B}m's equivalent to a
  single f{B}m?
\newblock {\em Stochastic Process. Appl.}, 117(1):57--70, 2007.

\bibitem{williams2004error}
S.~D. Williams, Y.~Bock, P.~Fang, P.~Jamason, R.~M. Nikolaidis,
  L.~Prawirodirdjo, M.~Miller, and D.~J. Johnson.
\newblock Error analysis of continuous {GPS} position time series.
\newblock {\em J. Geophys. Res.}, 109(B3):19 pp., 2004.

\bibitem{xu2017detecting}
C.~Xu.
\newblock Detecting periodic oscillations in astronomy data: {R}evisiting
  wavelet analysis with coloured and white noise.
\newblock {\em Mon. Not. R. Astron. Soc.}, 466(4):3827--3833, 2017.

\end{thebibliography}

\newpage

\begin{appendix}
	
	\setcounter{section}{1}

\section{Proof of Theorem~\ref{thm:CLT:mixedSM}}\label{sec:proof}
Except for Proposition~\ref{prop:CLT:1} below, we may and will assume that $M = 1$. Also, by standard localization results and since $g_0$ only contributes a  finite variation process with locally bounded density to $Y$, there is no loss of generality if we replace Assumption~\ref{Ass:B} by
\settheoremtag{(CLT'')}
\begin{Assumption} \label{Ass:Bprime} We have Assumption~\ref{Ass:B} with  $g_0\equiv0$. Moreover,  
	$$ \sup_{(\om,t)\in\Om\times[0,\infty)} \bigg\{ \lVert a_t(\om)\rVert +  \lVert \si_t(\om)\rVert + \lVert \rho_t(\om)\rVert + \lVert \si^{(0)}_t(\om)\rVert +  \lVert \wt \si_t(\om)\rVert  \bigg\}<\infty  $$
	and for every $p>0$, there is $K_p>0$ such that for all $s,t>0$,
	\beq\label{eq:Hoelder}\begin{split} \E[ \lVert \rho_t-\rho_s\rVert^p ]^{\frac 1p}&\leq K_p\lvert t-s\rvert^{\frac12},\qquad  \E[ \lVert \si^{(0)}_t-\si^{(0)}_s\rVert^p ]^{\frac 1p}\leq K_p\lvert t-s\rvert^{\ga},\\
		\E[ \lVert \wt\si_t-\wt\si_s\rVert^p ]^{\frac 1p}&\leq K_p\lvert t-s\rvert^{\eps}. \end{split}\eeq
\end{Assumption}

Under these assumptions, $g(t)=K_H^{-1} t^{H-1/2}\bone_{\{t>0\}}$ and $Y_t=X_0+A_t+X_t+Z_t$, where
\begin{equation} \label{dec:X:H:2}
	A_t = \int_{0}^{t} a_s \, \dd s, \qquad  
	X_t = \int_{0}^{t} \si_s \, \dd B_s, \qquad  Z_t = K_H^{-1} \int_{0}^{t} (t-s)^{H - \frac{1}{2}} \rho_s \, \dd B_s.
\end{equation}
Furthermore, we have $ \un{\Delta}^n_i  Y = \un\Delta^n_i A+\un{\Delta}^n_i X + \un{\Delta}^n_i Z$. For all $s \geq 0$ and $i, n \in \bbn$, define
\begin{equation} \label{def:kernel}
	\begin{split}
		\Delta^n_i g (s) & = g(i \Delta_n-s) - g((i - 1) \Delta_n-s) 
		\\ \uDeni g(s) & =  ( \Deni g (s), \Delta^n_{i + 1} g (s), \ldots, \Delta^n_{i + L - 1} g (s)  ),
	\end{split}
\end{equation}
such that we can rewrite 
\begin{equation}\label{eq:DeltaZ}
	\begin{split}
		\un{\Delta}^n_i Z & = \biggl(\int_0^{\infty} \Deni g(s) \rho_s \, \dd B_s, \int_0^{\infty} \Delta^n_{i + 1} g (s) \rho_s \, \dd B_s, \ldots, 
		\int_0^{\infty} \Delta^n_{i + L - 1} g (s) \rho_s \, \dd B_s \biggr) \\ & 
		= \int_0^{\infty} \rho_s \, \dd B_s \, \uDeni g(s)
	\end{split}
\end{equation}
in matrix notation. 
The proof of Theorem~\ref{thm:CLT:mixedSM} is  divided into  four steps.

\subsubsection*{Step 1: Truncation of fractional increments}
While increments of $X$ are given by stochastic integrals over disjoint intervals, \eqref{eq:DeltaZ} shows that all increments of $Z$ overlap. However, it can be shown that only the portion of the integral  closest to $i\Delta_n$ is asymptotically relevant.
\begin{Proposition}\label{prop:1} Under Assumption~\ref{Ass:Bprime}, if $\theta_n = [\Delta_n^{-\theta}]$ where $\theta$ is chosen such that
	\begin{equation} \label{theta:value}
		1 - \frac{1}{4 - 4H} = \frac{\frac{3}{2} - 2H}{2 - 2H} < \theta < \frac{1}{2}
	\end{equation}
	and if we define
	\begin{equation} \label{1st:trunc:CLT}
		\un{\Delta}^n_i Y^{\mathrm{tr}}= \un\Delta^n_i A+ \un{\Delta}^n_i X+ \int_{(i - \theta_n)\Delta_n}^{(i + L - 1) \Delta_n} \rho_s \, \dd B_s \, \uDeni g(s),
	\end{equation}
	then
	\beq\label{eq:step1} \Delta_n^{-\frac12}\biggl(V^n_f(Y,t)- 	\Delta_n \sum_{i = \theta_n + 1}^{[t/\Delta_n]-L+1} f \biggl( \frac{\un{\Delta}_i^n Y^{\mathrm{tr}}}{\sqrt{\Delta_n}} \biggr)\biggr)\stackrel{L^1}{\Longrightarrow}  0.\eeq
\end{Proposition}

\subsubsection*{Step 2: Centering and removing the fractional component}

In contrast to $\un\Delta^n_i Y$, two   truncated increments $\un{\Delta}^n_i Y^{\mathrm{tr}}$ and $\un{\Delta}^n_{i'} Y^{\mathrm{tr}}$ are defined on disjoint intervals as soon as  $\lvert i-i'\rvert\geq L+\theta_n$. By centering with the corresponding conditional expectations, we obtain a partial martingale structure, allowing us to show the following---in our opinion surprising---result: 
\begin{Proposition} \label{lem:approx:5} If $\theta$ is chosen according to \eqref{theta:value}, then 
	\begin{equation} \label{approx:before:CLT}   \begin{split} 
			& \DenOneHalf   \sum_{i = \theta_n + 1}^{[t/\Delta_n]-L+1}\biggl\{ f \biggl( \frac{\un{\Delta}_i^n Y^{\mathrm{tr}}}{\sqrt{\Delta_n}} \biggr) - 
			\bbe \biggl[f \biggl( \frac{\un{\Delta}_i^n Y^{\mathrm{tr}}}{\sqrt{\Delta_n}} \biggr) \mathrel{\Big|} \calf^n_{i - \theta_n} \biggr]\biggr\}  \\
			&\quad -\DenOneHalf
			\sum_{i = \theta_n + 1}^{[t/\Delta_n]-L+1} \biggl\{ f \biggl(\si_{(i- 1) \Den} \frac{\uDeni B}{\DenSquared}\biggr) - 
			\bbe \biggl[ f \biggl(\si_{(i- 1) \Den} \frac{\uDeni B}{\DenSquared} \biggr) \mathrel{\Big|} \calf^n_{i - 1} \biggr] \biggr\} \LeinsKonv 0.
		\end{split}\!\!
	\end{equation}
\end{Proposition}

We find this result surprising because upon centering by conditional expectations, the fractional component becomes negligible for all $H>\frac12$. Note that for $H\downarrow \frac12$, the local regularity of $Z$ becomes closer and closer to that of $X$. Since we are interested in convergence at a rate  of $1/\sqrt{\Delta_n}$, one would therefore expect that there is some critical value $H_0$ such that Proposition~\ref{lem:approx:5} holds for all $H>H_0$, while there is a bias due to the fractional part if $H<H_0$. 

\subsubsection*{Step 3: Martingale CLT}

After centering by conditional expectations and eliminating the fractional part, we obtain the following CLT:

\begin{Proposition} \label{prop:CLT:1}
	Under Assumption~\ref{Ass:Bprime},
	\begin{equation}
		\DenOneHalf \sum_{i = 1}^{[t/\Delta_n]-L+1} \biggl\{ f \biggl(\si_{(i- 1) \Den} \frac{\uDeni B}{\DenSquared}\biggr) -
		\bbe \biggl[ f \biggl(\si_{(i- 1) \Den} \frac{\uDeni B}{\DenSquared} \biggr) \mathrel{\Big|} \calf^n_{i - 1} \biggr] \biggr\} \stabKonv \calz,
	\end{equation}
	where $\mathcal{Z}$ is exactly as in \eqref{CLT:1}.
\end{Proposition}

\subsubsection*{Step 4: Convergence of conditional expectations}

The last step is to show that the conditional expectations 
$$ \Delta_n \sum_{i = \theta_n + 1}^{[t/\Delta_n]-L+1} \bbe \biggl[f \biggl( \frac{\un{\Delta}_i^n Y^{\mathrm{tr}}}{\sqrt{\Delta_n}} \biggr) \mathrel{\Big|} \calf^n_{i - \theta_n} \biggr]$$
that appear in \eqref{approx:before:CLT} converge to $V_f(Y,t)$, after removing the asymptotic bias term $\cala^{\prime n}_t$.
\begin{Proposition}\label{prop:4} Grant Assumption~\ref{Ass:Bprime}. If $\theta$ satisfies \eqref{theta:value}, then
	\begin{equation}  
		\mathtoolsset{multlined-width=0.9\displaywidth} \begin{multlined} 
			\DenMinOneHalf\biggl\{ \Den \sum_{i=\theta_n + 1}^{[t/\Delta_n]-L+1}  \bbe \biggl[f \biggl( \frac{\un{\Delta}_i^n Y^{\mathrm{tr}}}{\sqrt{\Delta_n}} \biggr) \mathrel{\Big|} \calf^n_{i - \theta_n} \biggr] -\int_{0}^{t} \mu_f (c(s)) \, \dd s -\cala^{\prime n}_t \biggr \} \LeinsKonv 0.
		\end{multlined}
	\end{equation}
\end{Proposition}

\begin{proof}[Proof of Theorem~\ref{thm:CLT:mixedSM}]The theorem immediately follows from Propositions~\ref{prop:1}--\ref{prop:4} and standard properties of stable convergence in law.
\end{proof}

Concerning the proofs of Propositions~\ref{prop:1}--\ref{prop:4}, notice that since the fractional part plays no role in Proposition~\ref{prop:CLT:1}, it is a special case of \cite[Theorem~11.2.1]{Chong20}. Moreover, the proof of  Proposition~\ref{prop:1} is very similar to that of \cite[Lemma C.1]{Chong21} (which was for the case $H<\frac12$), so we postpone it to the Appendix~\ref{appn}.

The main novelty is in the proof of Propositions~\ref{lem:approx:5} and \ref{prop:4}, which we detail in the following. Here, the fact that the martingale and the fractional components of $Y$ are driven by the same 
Brownian motion (instead of independent ones) starts to play a pivotal role. In what follows, we make repeated use of so-called \emph{size estimates} and \emph{martingale size estimates}. While the underlying ideas have existed since the beginnings of high-frequency statistics, we refer to \cite[Appendix~A]{Chong21}  for a more formal and more explicit version of these estimates. Another important estimate that we use many times and is true under Assumption~\ref{Ass:Bprime} is the following (cf.\ \cite[Equation~(B.3)]{Chong21}): for every $p>0$,
\beq\label{eq:away} \E\biggl[\biggl\lvert\int_0^{(i-\theta_n)\Den} \un\Delta^n_i g(s)\rho_s\,\dd B_s\biggr\rvert^p\biggr]^{\frac 1p} \lec \Den^{H+\theta(1-H)}. \eeq

\subsection{Proof of Proposition~\ref{lem:approx:5}}

Since $A$ is Lipschitz and both $\si$ and $\rho$ are $\frac12$-Hölder continuous in $L^p$, it is easy to see that replacing $	\un{\Delta}^n_i Y^{\mathrm{tr}}/\sqrt{\Den}$ by
\beq\label{eq:wtYtr} \begin{split}	\frac{\un{\Delta}^n_i \wt Y^{\mathrm{tr}}}{\sqrt{\Den}} &= \si_{(i-\theta_n)\Den} \frac{\un\Delta^n_i B }{\sqrt{\Den}}+\wt\xi^n_i\\ &=\si_{(i-\theta_n)\Den} \frac{\un\Delta^n_i B }{\sqrt{\Den}}+  \frac{1}{\sqrt{\Den}}\int_{(i - \theta_n)\Delta_n}^{(i + L - 1) \Delta_n} \rho_{(i-\theta_n)\Delta_n} \, \dd B_s \, \uDeni g(s)\end{split}\eeq
only leads an asymptotically negligible error. Indeed, since each term in
\begin{equation*} 
	\DenOneHalf   \sum_{i = \theta_n + 1}^{[t/\Delta_n]-L+1}\biggl\{ f \biggl( \frac{\un{\Delta}_i^n Y^{\mathrm{tr}}}{\sqrt{\Delta_n}} \biggr) - f \biggl( \frac{\un{\Delta}_i^n \wt Y^{\mathrm{tr}}}{\sqrt{\Delta_n}} \biggr)-
	\bbe \biggl[f \biggl( \frac{\un{\Delta}_i^n Y^{\mathrm{tr}}}{\sqrt{\Delta_n}} \biggr) -f \biggl( \frac{\un{\Delta}_i^n \wt Y^{\mathrm{tr}}}{\sqrt{\Delta_n}} \biggr)\mathrel{\Big|} \calf^n_{i - \theta_n} \biggr]\biggr\} 
\end{equation*}
is $\calf^n_{i+L}$-measurable with zero $\calf^n_{i-\theta_n}$-conditional expectation, a  size estimate shows that the previous display is of order $\sqrt{\theta_n}[\Den^{1/2}+(\theta_n\Den)^{1/2}]\leq 2\Den^{1/2-\theta}$, which tends to $0$ by \eqref{theta:value}.

The main work now consists in showing that we can drop $\wt \xi^n_i$ in \eqref{eq:wtYtr}. Then \eqref{approx:before:CLT} follows by moving the time point at which $\si$ is frozen from $(i-\theta_n)\Den$ to $(i-1)\Den$ (which, analogously to  what we just proved, is asymptotically negligible). The difficulty here is that a simple martingale size estimate is not sufficient: since $\wt\xi^n_i$ is of magnitude $\Den^{H-1/2}$, dropping it would lead to an error which is bounded by $\sqrt{\theta_n}\Den^{H-1/2}$. And this only goes to $0$ (for the smallest possible $\theta$ in \eqref{theta:value}) if $H>1-\frac{1}{2\sqrt{2}}\approx 0.65$. As a consequence, we need a more sophisticated argument.

We start by expanding
\begin{multline*} f \biggl( \frac{\un{\Delta}_i^n \wt Y^{\mathrm{tr}}}{\sqrt{\Delta_n}} \biggr) - f \biggl( \frac{\si_{(i-\theta_n)\Den}\un{\Delta}_i^n B}{\sqrt{\Delta_n}} \biggr) \\= \sum_{k=1}^{K} \sum_{\lvert \chi\rvert=k} \frac{1}{\chi!} \pd^\chi f\biggl( \frac{\si_{(i-\theta_n)\Den}\un{\Delta}_i^n B}{\sqrt{\Delta_n}} \biggr)(\wt \xi^n_i)^\chi + o_\bbp (\Den^{(H-\frac12)K} ) \end{multline*}
into a Taylor sum, where $K=N(H)+1$ and $N(H)$ is the same number as in Theorem~\ref{thm:CLT:mixedSM}. As $K\geq\frac{1}{2H-1}$, if we use a martingale size estimate, the $o_\bbp$-term contributes $\sqrt{\theta_n}\Den^{1/2}=\Den^{1/2-\theta}$ at most, which is negligible. For the remaining terms, we make the simplifying assumption that $d=L=1$; the subsequent arguments easily extend to any $d,L\in\N$ but require more cumbersome notation in the general case. To give the intuition behind our approach, consider $\chi=k=1$  in which case $(\wt \xi^n_i)^\chi=\Den^{-1/2}\int_{(i-\theta_n)\Den}^{i\Den} \Delta^n_{i} g(s)\rho_{(i-\theta_n)\Den}\,\dd B_s.$ The important observation is now the following: On the one hand, if $s$ is far from $i\Den$, say, $(i-\theta_n)\Den\leq s\leq (i-\theta^{(1)}_n)\Den$ for some $\theta^{(1)}_n = o(\theta_n)$, the corresponding part of the integral is of size $\Den^{H-1/2}(\theta^{(1)}_n)^{-(1-H)}$ by \eqref{eq:away}, which is smaller than  $\Den^{H-1/2}$. On the other hand, if $s$ is close to $i\Den$, the $\calf^n_{i-\theta_n}$-conditional expectation of the integral restricted to  $s\geq (i-\theta_n^{(1)})\Den$ agrees with its $\calf^n_{i-\theta_n^{(1)}}$-conditional expectation, which constitutes an improvement when we apply a martingale size estimate.

It turns out that in order to make this argument work in general, we need to choose not just one but a certain (potentially large but always finite) number of intermediate scales $\theta_n^{(q)} = [\Delta_n^{-\theta^{(q)}}]$ for $q=1, \ldots, Q-1$, where $Q \in \bbn$ and  $\theta^{(q)}$, $q=0, \ldots, Q$ are chosen such that $\theta = \theta^{(0)} > \theta^{(1)} > \cdots > \theta^{(Q-1)} > \theta^{(Q)} = 0$ and	
\begin{equation} \label{iter:CLT}
	\theta^{(q)} > \frac{1}{2-2H}\theta^{(q-1)}-\frac{H-\frac12}{1-H},\qquad q=1,\dots, Q.
\end{equation}
Earlier versions of this trick appear in \cite{BN11,Chong20,Chong20a2,Chong21, Corcuera13}.
By solving the underlying recurrence equation, the reader can check that such a choice exists.
We also let $\theta_n^{(Q)} = 0$.
By the multinomial theorem, for each $\chi=k\in\{1,\dots,K\}$, we have that
\begin{align*}
	(\wt\xi^n_i)^k=\sum_{j_1+\cdots+j_Q=k} {k\choose {j_1, \dots, j_Q}} \prod_{q=1}^Q \biggl( \int_{(i-\theta_n^{(q-1)})\Den}^{(i-\theta_n^{(q)})\Den} \Delta^n_i g(s)\rho_{(i-\theta_n)\Den}\,\dd B_s \biggr)^{j_q}.
\end{align*}
For fixed $k$ and $j_1,\dots,j_Q$, let $q_{\mathrm{min}}=\min\{q=1,\dots,Q: j_q\neq 0\}$. Then
\begin{align*}  &\E\biggl[\pd^\chi f\biggl( \frac{\si_{(i-\theta_n)\Den}\un{\Delta}_i^n B}{\sqrt{\Delta_n}} \biggr)\prod_{q=1}^Q \biggl( \int_{(i-\theta_n^{(q-1)})\Den}^{(i-\theta_n^{(q)})\Den} \Delta^n_i g(s)\rho_{(i-\theta_n)\Den}\,\dd B_s \biggr)^{j_q}\mathrel{\Big|} \calf^n_{i-\theta_n} \biggr]\\
	&=\E\biggl[\pd^\chi f\biggl( \frac{\si_{(i-\theta_n)\Den}\un{\Delta}_i^n B}{\sqrt{\Delta_n}} \biggr)\prod_{q=1}^Q \biggl( \int_{(i-\theta_n^{(q-1)})\Den}^{(i-\theta_n^{(q)})\Den} \Delta^n_i g(s)\rho_{(i-\theta_n)\Den}\,\dd B_s \biggr)^{j_q}\mathrel{\Big|} \calf^n_{i-\theta_n^{(q_{\mathrm{min}}-1)}} \biggr]. \end{align*}
Hence, by a martingale size estimate, its contribution to 
\begin{multline*} 
	\DenOneHalf   \sum_{i = \theta_n + 1}^{[t/\Delta_n]-L+1}\biggl\{   f \biggl( \frac{\un{\Delta}_i^n Y^{\mathrm{tr}}}{\sqrt{\Delta_n}} \biggr) - f \biggl( \frac{\si_{(i-\theta_n)\Den}\un{\Delta}_i^n B}{\sqrt{\Delta_n}} \biggr)\\
	-
	\bbe \biggl[ f \biggl( \frac{\un{\Delta}_i^n Y^{\mathrm{tr}}}{\sqrt{\Delta_n}} \biggr) - f \biggl( \frac{\si_{(i-\theta_n)\Den}\un{\Delta}_i^n B}{\sqrt{\Delta_n}} \biggr)\mathrel{\Big|} \calf^n_{i - \theta_n} \biggr]\biggr\} 
\end{multline*}
is of order $$\sqrt{\theta_n^{(q_{\mathrm{min}}-1)}}\Den^{(H-\frac12)k} \prod_{q=1}^Q \Den^{\theta^{(q)}j_q(1-H)}\leq \sqrt{\theta_n^{(q_{\mathrm{min}}-1)}}\Den^{H-\frac12}\Den^{\theta^{(q_{\mathrm{min}})}(1-H)},$$
which tends to zero by \eqref{iter:CLT}.\halmos

\subsection{Proof of Proposition~\ref{prop:4}}

Let us first remark that, in contrast to Proposition~\ref{lem:approx:5}, we \emph{cannot} drop the fractional component when we determine the asymptotic behavior of the conditional expectations in Proposition~\ref{prop:4}. In fact, the fractional part is precisely the reason why intermediate limits appear and have to be subtracted  to obtain convergence to $V_f(Y,t)$ at a rate of $1/\sqrt{\Den}$.

In a first step, we discretize both $\si$ and $\rho$ in $\un{\Delta}_i^n Y^{\mathrm{tr}}$. Recalling \eqref{iter:CLT}, we define
\begin{align*} 
	\xi_i^{n,\dd} &= \frac1{\sqrt{\Den}} \int_{(i - \theta_n)\Delta_n}^{(i + L - 1) \Delta_n} 
	\sum_{q=1}^{Q} \rho_{(i-\theta_n^{(q-1)})\Den} \ind_{((i-\theta_n^{(q-1)})\Den, (i-\theta_n^{(q)})\Den)}(s)  \, \dd B_s \, \uDeni g(s), \\ 
	\un{\Delta}^n_i Y^{\mathrm{dis}} &= \si_{(i - 1) \Den} \un{\Delta}^n_i B + \sqrt{\Den}\xi_i^{n,\dd}.
\end{align*}
\begin{Lemma}\label{lem:approx:6} Under the assumptions of Proposition~\ref{prop:4}, 
	\begin{equation}\label{eq:help}
		\DenOneHalf \sum_{i=\theta_n + 1}^{[t/\Delta_n]-L+1} \biggl\{ \bbe \biggl[f  \biggl(\frac{\un{\Delta}_i^n Y^{\mathrm{tr}}}{\sqrt{\Delta_n} }  \biggr) \mathrel{\Big|} \calf^n_{i - \theta_n} \biggr] - 
		\bbe \biggl[f  \biggl(\frac{\un{\Delta}_i^n Y^{\mathrm{dis}}}{\sqrt{\Delta_n} }  \biggr) \mathrel{\Big|} \calf^n_{i - \theta_n} \biggr] \biggr\} \LeinsKonv 0.
	\end{equation}	
\end{Lemma}

The proof is quite technical, but the dependence of the Brownian motions driving the martingale and the fractional part of $Y$ do not play any role here, which is why we defer it to the Appendix~\ref{appn}.

\begin{proof}[Proof of Proposition~\ref{prop:4}] By Lemma~\ref{lem:approx:6}, we are left to show that
	\begin{equation}\label{eq:remains}
		\DenOneHalf \sum_{i=\theta_n + 1}^{[t/\Delta_n]-L+1} \biggl\{ 
		\bbe \biggl[f  \biggl(\frac{\un{\Delta}_i^n Y^{\mathrm{dis}}}{\sqrt{\Delta_n} }  \biggr) \mathrel{\Big|} \calf^n_{i - \theta_n} \biggr]
		- V_f(Y,t)-\cala^{\prime n}_t \biggr\} \LeinsKonv 0.
	\end{equation}
	To this end, we recall \eqref{mat:lim:thms} and define $v_0^{n,i}, v^{n,i} \in (\bbr^{d \times L})^2$ by
	\begin{equation} \label{cov:mat:v0}
		\begin{split}
			&	(v_0^{n,i})_{k\ell,k\ell'}  = c((i-1)\Den)_{k\ell, k'\ell'} \\
			&\qquad+\mathtoolsset{multlined-width=0.8\displaywidth} \begin{multlined}[t]   (\si_{(i-1)\Den}   \rho_{(i-\theta_n^{(Q-1)})\Den}^T\bone_{\{\ell\leq \ell'\}}+  \rho_{(i-\theta_n^{(Q-1)})\Den}\si_{(i-1)\Den}^T\bone_{\{\ell'\leq \ell\}})_{k k'} \\
				\times
				\int_{(i + \ell \wedge \ell' - 2 )\Delta_n}^{(i + \ell \wedge \ell' - 1 ) \Delta_n} \frac{\Delta^n_{i + \ell \vee \ell' - 1} g(s)}{\Den}  \,  \dd s \end{multlined}\\ 
			& \qquad +
			\sum_{q=1}^{Q} 
			( \rho_{(i-\theta_n^{(q-1)})\Den} \rho_{(i-\theta_n^{(q-1)})\Den}^T  )_{k k'}
			\int_{(i-\theta_n^{(q-1)})\Den}^{(i-\theta_n^{(q)})\Den} \frac{\Delta^n_{i + \ell - 1} g(s) \Delta^n_{i + \ell' - 1} g(s)}{\Den} 
			\,  \dd s
		\end{split}\!
	\end{equation}
	and
	\begin{equation} \label{eq:vni}
		(v^{n,i})  = c((i-1)\Den) + \la((i-1)\Den) \Den^{H-\frac12} +\pi((i-1)\Den) \Den^{2H-1},
	\end{equation}
	respectively.
	Observe that $v_0^{n,i}$ is the covariance matrix of $\un{\Delta}^n_i Y^{\mathrm{dis}}/\sqrt{\Den}$ if  $\si$ and $\rho$ are deterministic. We separate the proof of \eqref{eq:remains} into four parts:
	\begin{align}
		&	\DenOneHalf \sum_{i=\theta_n + 1}^{[t/\Delta_n]-L+1} \biggl\{ 
		\bbe \biggl[f  \biggl(\frac{\un{\Delta}_i^n Y^{\mathrm{dis}}}{\sqrt{\Delta_n} }  \biggr) \mathrel{\Big|} \calf^n_{i - \theta_n} \biggr]
		- \mu_f \Bigl( \bbe  [ v_0^{n,i} \mid \calf^n_{i - \theta_n}  ] \Bigr) \biggr\} \LeinsKonv 0.\label{eq:part1} \\
		&		\DenOneHalf \sum_{i=\theta_n + 1}^{[t/\Delta_n]-L+1} \Bigl\{  
		\mu_f \Bigl( \bbe  [ v_0^{n,i} \mid \calf^n_{i - \theta_n}  ] \Bigr)-\mu_f ( v_0^{n,i}) \Bigr\} \LeinsKonv 0,\label{eq:part2} \\
		&	 	\DenOneHalf \sum_{i=\theta_n + 1}^{[t/\Delta_n]-L+1}  \{  
		\mu_f ( v_0^{n,i})-\mu_f ( v^{n,i})  \} \LeinsKonv 0,\label{eq:part3}\\
		& \DenOneHalf \sum_{i=\theta_n + 1}^{[t/\Delta_n]-L+1}  \{ \mu_f ( v^{n,i})-V_f(Y,t)-\cala^{\prime n}_t\}\LeinsKonv 0.\label{eq:part4}
	\end{align}
	
	Concerning the first part,
	by a change of variable and \eqref{est:fBm:kernel:trunc:2},  the second term on the right-hand side of \eqref{cov:mat:v0} equals
	\begin{align*}
		&(\si_{(i-1)\Den}   \rho_{(i-\theta_n^{(Q-1)})\Den}^T\bone_{\{\ell\leq \ell'\}}+  \rho_{(i-\theta_n^{(Q-1)})\Den}\si_{(i-1)\Den}^T\bone_{\{\ell'\leq \ell\}})_{k k'} \int_{0}^{\Delta_n} \frac{\Delta^n_{\vert \ell - \ell' \vert +1} g(s)}{\Den}  \,  \dd s\\
		&= \frac1{2^{\bone_{\{\ell=\ell'\}}}}(\si_{(i-1)\Den}   \rho_{(i-\theta_n^{(Q-1)})\Den}^T\bone_{\{\ell\leq \ell'\}}+  \rho_{(i-\theta_n^{(Q-1)})\Den}\si_{(i-1)\Den}^T\bone_{\{\ell'\leq \ell\}})_{k k'}   \Den^{H-\frac12} \Phi^H_{\vert \ell - \ell' \vert}.
	\end{align*}
	With this in mind, the remaining proof of \eqref{eq:part1} is completely analogous to that of \cite[Lemma~D.2]{Chong21} if, in the notation of the reference, we define
	\begin{align*}
		&\mathbb{Y}^{n,r}_i=\frac1{\sqrt{\Den}} \int_{(i - \theta_n)\Delta_n}^{(i + L - 1) \Delta_n} 
		\sum_{q=1}^{r} \rho_{(i-\theta_n^{(q-1)})\Den} \ind_{((i-\theta_n^{(q-1)})\Den, (i-\theta_n^{(q)})\Den)}(s)  \, \dd B_s \, \uDeni g(s),\\
		&	(	\Upsilon^{n,r}_i)_{k\ell,k'\ell'}  = c((i-1)\Den)_{k\ell, k'\ell'}+ \Bigl(\si_{(i-1)\Den}   \rho_{(i-\theta_n^{(Q-1)})\Den}^T\bone_{\{\ell\leq \ell'\}}\\
		&\qquad    +  \rho_{(i-\theta_n^{(Q-1)})\Den}\si_{(i-1)\Den}^T\bone_{\{\ell'\leq \ell\}}\Bigr)_{k k'}  
		\int_{(i + \ell \wedge \ell' - 2 )\Delta_n}^{(i + \ell \wedge \ell' - 1 ) \Delta_n} \frac{\Delta^n_{i + \ell \vee \ell' - 1} g(s)}{\Den}  \,  \dd s \\ & \qquad +
		\sum_{q=r+1}^{Q} 
		( \rho_{(i-\theta_n^{(q-1)})\Den} \rho_{(i-\theta_n^{(q-1)})\Den}^T  )_{k k'}
		\int_{(i-\theta_n^{(q-1)})\Den}^{(i-\theta_n^{(q)})\Den} \frac{\Delta^n_{i + \ell - 1} g(s) \Delta^n_{i + \ell' - 1} g(s)}{\Den} 
		\,  \dd s
	\end{align*}
	and realize upon using \eqref{eq:away} that for any $p>0$,
	\begin{align*}
		\E[\lvert \mathbb{Y}^{n,r}_i \rvert^p]^{\frac1p}&\lec \Den^{\theta^{(r)}(1-H)}\Den^{H-\frac12},\\
		\E\Bigl[\Bigl\lvert \E[\Upsilon^{n,r}_i \mid \calf^n_{i-\theta_n^{(r)}}] - \E[\Upsilon^{n,r}_i \mid \calf^n_{i-\theta_n^{(r-1)}}]\Bigr\rvert^p\Bigr]^{\frac1p}&\lec (\theta_n^{(r-1)}\Den)^{\frac12}\Den^{H-\frac12}.
	\end{align*}
	(For the last inequality, first note that the $\mathrm{d}s$-integrals in $\mathbb{Y}^{n,r}_i$ are $O(\Den^{H-1/2})$. Then observe that the two conditional expectations will become equal if we shift time in the coefficients defining $\mathbb{Y}^{n,r}_i$ to $i-\theta_n^{(r-1)}\Den$, which comes with an error of $(\theta_n^{(r-1)}\Den)^{1/2}$.)
	Similarly, apart from obvious modifications, the proof of \cite[(D.10)]{Chong21} can be reproduced almost identically to show \eqref{eq:part2}. 
	For  \eqref{eq:part3}, the proof in our setting (with $H>\frac12$) is, in fact, much simpler than the analogous statement in \cite[(D.11)]{Chong21}. Indeed, by \eqref{est:fBm:kernel:2}, \eqref{est:fBm:kernel:trunc:2} and \eqref{est:fBm:kernel:trunc},   $v_0^{n,i} - v^{n,i}$ is of size $(\theta_n^{(Q-1)}\Den)^{1/2}\Den^{H-1/2}+\sum_{q=1}^Q (\theta_n^{(q-1)}\Den)^{1/2}\Den^{\theta^{(q)}(2-2H)}\Den^{2H-1}+\Den^{\theta(2-2H)}\Den^{2H-1}$, which is $o(\sqrt{\Den})$ by \eqref{iter:CLT} and \eqref{theta:value} and thus yields \eqref{eq:part3} by the mean-value theorem.	
	
	Concerning \eqref{eq:part4}, we recall the multi-index notations introduced in \eqref{mult:ind:1} and \eqref{mult:ind:2}. We then express the difference 
	$\mu_f ( v^{n,i} ) - \mu_f ( c((i-1)\Den) )$ in a Taylor expansion of order $N(H)$ (which is possible since $f$ is $2(N(H)+1)$-times continuously differentiable by assumption, see \cite[(D.45), (D.46)]{Chong20}) to  obtain
	\begin{equation*} 
		\DenOneHalf \sum_{i=\theta_n + 1}^{[t/\Delta_n]-L+1}  \Bigl\{  \mu_f ( v^{n,i} ) - \mu_f ( c((i-1)\Den) ) -\cala^{\prime n}_i \Bigr \} = K_1^{n}(t) + K_2^{n}(t),
	\end{equation*}
	where
	\begin{equation*}
		\cala^{\prime n}_i=
		\sum_{j=1}^{  {N(H)} } 
		\sum_{\vert \chi \vert = j} \frac{1}{\chi!}  \pd^{\chi} \mu_f  (c((i-1)\Den) ) \Bigl( \Den^{H-\frac12} \la((i-1)\Den) + \Den^{2H-1} \pi((i-1)\Den) \Bigr)^{\chi}
	\end{equation*}
	and
	\begin{equation*}
		K^{n}(t)  = \DenOneHalf \sum_{i=\theta_n + 1}^{[t/\Delta_n]-L+1}
		\sum_{\vert \chi \vert = N(H) + 1}  \frac{1}{\chi!}\pd^{\chi} \mu_f( \chi^n_i ) (v^{n,i} - c((i-1)\Den) )^{\chi} 
	\end{equation*}
	for some  $\chi^n_i$ between $v^{n,i}$ and $c((i-1)\Den)$.

	By simple size estimates and the definition of $N(H)$ in \eqref{eq:NH},
	\begin{equation*}
		\bbe \biggl[ \sup_{t\in[0,T]} {\lvert K^{n}(t) \rvert} \biggr] \lesssim 
		\Den^{-\frac12}  \Den^{(N(H) + 1)(H-\frac12)} \lra 0 
	\end{equation*}	
	as $n\to\infty$. Our last step is to remove the discretization in time, that is, to show that
	\begin{equation*} 
		\Den^{-\frac12}\biggl\{ \Den \sum_{i=\theta_n + 1}^{[t/\Delta_n]-L+1}  \Bigl \{  \mu_f ( c((i-1)\Den) ) + \cala^{\prime n}_i  \Bigr\} -V_f(Y,t)-\cala^{\prime n}_t\biggr\} \LeinsKonv 0.
	\end{equation*}
	To see this, observe that $\Den^{-1/2}\{\Den\sum_{i=\theta_n+1}^{[t/\Den]-L+1} \mu_f(c((i-1)\Den))-V_f(Y,t)\} \LeinsKonv 0$ by \cite[Lemma 11.2.7]{Jacod12}. At the same time, note that changing $c((i-1)\Den)$, $\la((i-1)\Den)$ and $\pi((i-1)\Den)$ to $c(s)$, $\la(s)$ and $\pi(s)$ for $s\in((i-1)\Den,i\Den)$ induces an error of order $\sqrt{\Den}$, respectively, by Assumption \ref{Ass:Bprime}. Because there is at least a factor of $\Den^{H-1/2}$ present in $\cala^{\prime n}_i$, the mean-value theorem shows that 
	\begin{equation*}\bbe \biggl[ \sup_{t\in[0,T]} {\biggl\lvert \Den^{-\frac12}\biggl(\Den\sum_{i=\theta_n+1}^{[t/\Den]-L+1} \cala^{\prime n}_i-\cala^{\prime n}_t\biggr)\biggr\rvert} \biggr] \lec \Den^{-\frac12}\Den^{\frac12+(H-\frac12)}\to0.\qedhere\end{equation*}
\end{proof}

\section{Technical results for the proof of Theorem~\ref{thm:CLT:mixedSM}}\label{appn}

We start with some estimates on the kernel $g$.
\begin{Lemma} \label{lem:est:fBm:kernel}
	Recall the definitions in \eqref{kernel:g}, \eqref{const:C:H}, \eqref{nu:La} and \eqref{num:Ga} and assume that $H\in(0,1)\setminus\{\frac12\}$ and $g_0\equiv0$.
	\begin{enumerate}
		\item For any $k, n \in \bbn$,
		\begin{equation}  \label{est:fBm:kernel:1}
			\int_0^{\infty} \Delta^n_k g(t)^2 \, \dd t  = K_H^{-2} \left\{ \frac{1}{2H} + \int_{1}^{k} \left( r^{H -\frac{1}{2}} - (r-1)^{H -\frac{1}{2}} \right)^2 \, \dd r \right\} \Delta_n^{2H}.
		\end{equation}
		\item For any $k, \ell, n \in \bbn$ with $k < \ell$,
		\begin{equation} \label{est:fBm:kernel:2}
			\int_{- \infty}^{\infty} \Delta^n_k g(t) \Delta^n_\ell g(t) \, \dd t = 
			\Den^{2H} \Ga^H_{\ell - k}.
		\end{equation}
		\item 
		For $r \geq 0$,
		\begin{equation} \label{est:fBm:kernel:trunc:2}
			\int_{0}^{\Den} \frac{\Delta^n_{r + 1} g(s)}{\Den^{2H\wedge 1}} \, \dd s
			= \Den^{\lvert H-\frac12\rvert} 2^{-\bone_{\{r=0\}}}\Phi^H_r 
		\end{equation}
		\item For any $\theta \in (0,1)$, setting $\theta_n = [\Delta_n^{-\theta}]$, we have for any $i > \theta_n$ and $r \in \bbn$,
		\begin{equation} \label{est:fBm:kernel:trunc}
			\int_{-\infty}^{(i - \theta_n) \Den} \Delta^n_i g(s) \Delta^n_{i + r} g(s) \, \dd s \lesssim \Den^{2H+\theta(2-2H)}.
		\end{equation}
	\end{enumerate}
\end{Lemma}

\begin{proof} The third statement follows from a direct calculation. All other statements follow from \cite[Lemma~B.1]{Chong21} (the proof of which applies to both $H<\frac12$ and $H>\frac12$).
\end{proof}

\begin{proof}
[Proof of Proposition~\ref{prop:1}]	Note that the left-hand side of \eqref{1st:trunc:CLT} is given by
\begin{equation*} 
	\Delta_n^{\frac{1}{2}} \sum_{i = 1}^{\theta_n}  f \biggl( \frac{\un{\Delta}_i^n  Y}{\sqrt{\Delta_n}} \biggr) +
	\DenOneHalf \sum_{i = \theta_n + 1}^{[t/\Delta_n]-L+1} \Bigl\{\de_i^{n} -\bbe [ \de_i^{n} \mid \calf^n_{i - \theta_n}]\Bigr\} + \DenOneHalf \sum_{i = \theta_n + 1}^{[t/\Delta_n]-L+1} \bbe [ \de_i^{n} \mid \calf^n_{i - \theta_n}],
	\end{equation*}
where
$\delta^n_i = f  (  {\un{\Delta}_i^n  Y}/{\sqrt{\Delta_n}}  ) - 
	f  (  {\un{\Delta}_i^n Y^{\mathrm{tr}}}/{\sqrt{\Delta_n}}  )$.
Clearly, the first term in the previous display is of size $\Delta_n^{1/2-\theta}$ and hence negligible by \eqref{theta:value}.  Regarding the second term, note that	 $\vert f(z) - f(z') \vert \lesssim (1 +P(z) + P(z')) \Vert z - z' \Vert$ by Assumption~\ref{Ass:B}, where $P$ is a polynomial. 
So by the contraction property of conditional expectations and  a size estimate,
	\begin{equation*}
	\bbe \Bigl[  (\de_i^{n} -\bbe [ \de_i^{n} \mid \calf^n_{i - \theta_n}] )^2 \Bigr] \leq 4 \bbe  [ (\de_i^{n})^2  ]\lec \Den^{-1} \int_0^{(i-\theta_n)\Den}\lVert \un\Delta^n_i g(s)\rVert^2\,\dd s
	  \lesssim \Den^{2H-1+2\theta(1-H)},
	\end{equation*}	
where the last step follows from \cite[(B.3)]{Chong21}.
Moreover,  $\ov{\de}_i^{n}$ is $\calf^n_{i + L -1}$-measurable with vanishing $\calf^n_{i - \theta_n}$-conditional expectation. Therefore,  a martingale size estimate shows that 
	\begin{equation*}
	\bbe \biggl[\sup_{t\in[0, T]} \biggl \lvert \DenOneHalf \sum_{i = \theta_n + 1}^{[t/\Delta_n]-L+1} \Bigl\{\de_i^{n} -\bbe [ \de_i^{n} \mid \calf^n_{i - \theta_n}]\Bigr\}  \biggr \rvert  \biggr] 
	\lesssim  \theta_n^{\frac12} \Den^{H-\frac12+\theta(1-H)} \leq 
	\Den^{(H-\frac12) (1 - \theta  )}\to0.
	\end{equation*}
	
	Next, define
	\begin{align*}
\zeta^n_i &= \frac{\un{\Delta}_i^n  Y - \un{\Delta}_i^n Y^{\mathrm{tr}}} { \sqrt{\Delta_n}} =	\frac1{\sqrt{\Den}}\int_{0}^{(i - \theta_n)\Den}  \rho_s \, \dd W_s \, \uDeni g(s),\\
	\xi^n_i & = \si_{(i- \theta_n) \Den} \frac{\uDeni B }{\sqrt{\Den}}+ \frac{1}{\sqrt{\Den}}\int_{(i - \theta_n) \Den}^{(i + L - 1) \Den} \rho_{(i - \theta_n) \Den} \, \dd W_s \, \uDeni g(s).
	\end{align*}
	Then Taylor's theorem implies that
	\begin{equation*}
	\begin{split}
	\de_i^{n} & = \sum_{\al}\pd_\al f \biggl( \frac{\uDeni Y^{\mathrm{tr}}}{\sqrt{\Delta_n}} \biggr) (\zeta^n_i)_{\al} + 
	\frac{1}{2} \sum_{\al, \beta}  \pd^2_{\al\beta} f ( \eta_i^{n,1} )  (\zeta^n_i)_{\al} (\zeta^n_i)_{\beta} \\ & =
	\sum_{\al} \pd_\al f  (  \xi^n_i   ) (\zeta^n_i)_{\al} + 
	\frac{1}{2} \sum_{\al, \beta} \pd^2_{\al\beta} f (\eta_i^{n,2}) 
	(\zeta^n_i)_{\al} \biggl( \frac{\uDeni Y^{\mathrm{tr}}}{\sqrt{\Delta_n}}-\xi^n_i\biggr)_{\beta} \\
 &\quad+	
	\frac{1}{2} \sum_{\al, \beta} \pd^2_{\al\beta} f (\eta_i^{n,1}) (\zeta^n_i)_{\al} (\zeta^n_i)_{\beta}, 
	\end{split}
	\end{equation*}
	where $\al, \beta \in \{ 1,\ldots,d \} \times \{1, \ldots, L\}$ and $\eta_i^{n,1}$ (resp., $\eta_i^{n,2}$) is some point on the line between 
	$\uDeni Y / \DenSquared $ and $\uDeni Y^{\mathrm{tr}} / \DenSquared $ (resp., $\un\Delta^n_i Y^{\mathrm{tr}} / \DenSquared $ and $\xi^n_i  $). If $\de_i^{n,1}$, $\de_i^{n,2}$ and $\de_i^{n,3}$ denote the last three terms in the previous display, then
	\begin{equation*}
	\DenOneHalf \sum_{i=\theta_n + 1}^{[t/\Delta_n]-L+1} \mathbb{E}  [ \de_i^{n} \mid \mathcal{F}^n_{i - \theta_n}  ] = \sum_{j=1}^3 D_j^n(t) \quad \textrm{with} \quad 
	D_j^n(t) = \DenOneHalf \sum_{i=\theta_n + 1}^{[t/\Delta_n]-L+1} \mathbb{E}  [ \de_i^{n,j} \mid \mathcal{F}^n_{i - \theta_n}  ].
	\end{equation*}
	Since $\zeta^n_i$ is $\calf^n_{i - \theta_n}$-measurable,  
$
	\mathbb{E}  [ \pd_\al f  ( \xi^n_i) (\zeta^n_i)_{\al} \mid \mathcal{F}^n_{i - \theta_n}  ]  =
	(\zeta^n_i)_{\al} \mathbb{E}  [ \pd_\al f  (  \xi^n_i ) \mid \mathcal{F}^n_{i - \theta_n}  ] = 0
$
	because $\xi^n_i$ has a centered normal distribution given $\mathcal{F}^n_{i - \theta_n}$ and $f$ has odd first partial derivatives (as $f$ is even). It follows that $D_1^n(t) \equiv 0$. For the two remaining terms, we combine  size estimates with \cite[(B.3)]{Chong21} and the Hölder properties of $\sigma$ and $\rho$ to obtain
	\begin{equation*}
	\begin{split}
	\bbe \biggl[ \sup_{t\in[0, T]}   \vert D^n_2(t)   \vert \biggr] & \lesssim 
	\Den^{-\frac12+(H-\frac12)+\theta(1-H)} 
	 \Bigl[\Den^{\frac12}+ (\theta_n \Den)^{\frac{1}{2}}\Bigr] \leq 2 \Den^{(H-\frac12) (1 - \theta  )} \to 0,  \\
	\bbe \biggl[ \sup_{t\in[0, T]}   \vert D^n_3(t)   \vert \biggr] & \lesssim 
	\DenMinOneHalf  \Den^{2(H-\frac12+\theta(1-H))}  = 
	\Den^{- \frac32+2H+ \theta (2-2H)}\to 0,
	\end{split}
	\end{equation*}
where the last step follows from \eqref{theta:value}.
\end{proof}

\begin{proof}[Proof of Lemma~\ref{lem:approx:6}]
	The left-hand side of \eqref{eq:help} is given by $H^n_1(t) + H^n_2(t)$ where
	\begin{equation*}
	\begin{split}
	H^n_1(t) & = \Delta_n^{\frac{1}{2}} \sum_{i = \theta_n + 1}^{[t/\Delta_n]-L+1} \sum_{\al} 
	\mathbb{E} \biggl[\partial_\al f\biggl(\frac{\un{\Delta}_i^n Y^{\mathrm{dis}}}{\sqrt{\Delta_n} }  \biggr)(\kappa_i^n)_{\al} \mathrel{\Big \vert} \mathcal{F}^n_{i - \theta_n} \biggr],\\
	H^n_2(t) & = \Delta_n^{\frac{1}{2}}  \sum_{i = \theta_n + 1}^{[t/\Delta_n]-L+1} \sum_{\al, \beta}  \frac{1}{2^{\bone_{\{\al=\beta\}}}}
	\mathbb{E} \Bigl[ \partial^2_{\al\beta}f  (\xi_i^{\prime n,\dd}   ) (\kappa_i^n)_{\al} (\kappa_i^n)_{\beta} \mathrel{\big \vert} \mathcal{F}^n_{i - \theta_n} \Bigr].	
	\end{split}
	\end{equation*}
 Here $\kappa_i^n =(\un{\Delta}_i^n Y^{\mathrm{tr}}- \un{\Delta}_i^n Y^{\mathrm{dis}})/\DenSquared$, $\xi_i^{\prime n,\dd}$ is some point on the line between $\un{\Delta}_i^n Y^{\mathrm{tr}}/\DenSquared$ and $\un{\Delta}_i^n Y^{\mathrm{dis}}/\DenSquared$, and $\al$ and $\beta$ run through $\{1, \ldots, N\} \times \{1, \ldots, L\}$. 
	
	For $\al = (k,\ell)$, we have
	\begin{equation} \label{dec:lem:approx:6}
 \begin{split} 
	(\kappa^n_i)_{\al} &= \frac{\Delta^n_{i + \ell -1} A^k}{\DenSquared} 
	+ \frac{1}{\DenSquared} \int_{(i + \ell - 2) \Delta_n}^{(i + \ell - 1) \Delta_n} \sum_{\ell' = 1}^{d'}  ( \si_s^{k \ell'} - \si^{k \ell'}_{(i - 1) \Den}  ) \, \dd B_s^{\ell'} \\ 
&\quad	+ \frac{1}{\sqrt{\Den}}  \sum_{q=1}^{Q} \int_{(i - \theta^{(q-1)}_n) \Den}^{(i-\theta^{(q)}_n) \Den}  {\Delta^n_{i + \ell -1} g(s)} 
  \sum_{\ell' = 1}^{d'}
	 ( \rho_s^{k \ell'} - \rho^{k \ell'}_{(i-\theta_n^{(q-1)})\Den}  )  \, \dd B^{\ell'}_s.
	\end{split}
	\end{equation}
Since $\theta^{(q)}<\frac12$ for all $q$, we have
$
	\bbe  [ \sup_{t \in[0, T]} \bv H^n_2(t) \bv  ] \lesssim \Den^{-1/2}  ( 2\Den^{1/2}   + \Den^{H-1/4})^2 \to 0
$
by a size estimate.
Splitting $H^n_1(t) = H^n_{11}(t) + H^n_{12}(t) + H^n_{13}(t)$ into three parts according to   \eqref{dec:lem:approx:6}, we can use another size estimate to show that $H^n_{13}(t)$ is of order $\Den^{H-1}\sum_{q=1}^{Q} (\theta_n^{(q-1)}\Den)^{1/2}\Den^{\theta^{(q)}(1-H)}$, which tends to $0$ by \eqref{iter:CLT}.
	In order to show that $H^n_{11}(t)$ and $H^n_{12}(t)$ are asymptotically negligible, we need further arguments. 

Concerning $H^n_{11}(t)$, we decompose, for fixed $\al = (k,\ell)$,  
	\begin{equation} \label{NR3}
	\begin{split}
	& \Delta_n^{\frac{1}{2}} \sum_{i = \theta_n + 1}^{[t/\Delta_n]-L+1}
	\mathbb{E} \biggl[ \partial_\al f\biggl(\frac{\un{\Delta}_i^n Y^{\mathrm{dis}}}{\sqrt{\Delta_n} }  \biggr)  \frac{\Delta^n_{i + \ell -1} A^k}{\DenSquared}  \mathrel{\Big \vert} \mathcal{F}^n_{i - \theta_n} \biggr] \\ & \qquad = 
	\Delta_n^{\frac{1}{2}} \sum_{i = \theta_n + 1}^{[t/\Delta_n]-L+1}
	\mathbb{E} \biggl[ \partial_\al f\biggl(\frac{\un{\Delta}_i^n Y^{\mathrm{dis}}}{\sqrt{\Delta_n} }  \biggr)
	\frac{1}{\DenSquared} \int_{(i + \ell - 2) \Delta_n}^{(i + \ell - 1) \Delta_n}  (a^k_s - a^k_{(i - \theta_n) \Delta_n} ) \, \dd s
	\mathrel{\Big \vert} \mathcal{F}^n_{i - \theta_n} \biggr] \!\!\!\!\\  &\qquad\quad+ 
	\Delta_n^{\frac{1}{2}} \sum_{i = \theta_n + 1}^{[t/\Delta_n]-L+1}
	\mathbb{E} \biggl[ \partial_\al f\biggl(\frac{\un{\Delta}_i^n Y^{\mathrm{dis}}}{\sqrt{\Delta_n} }  \biggr)
	\DenSquared a^k_{(i - \theta_n) \Delta_n}
	\mathrel{\Big \vert}  \mathcal{F}^n_{i - \theta_n} \biggr].
	\end{split}\raisetag{-3.75\baselineskip}
	\end{equation}	
	By the contraction property of conditional expectations and  Hölder's inequality, the uniform $L^1$-norm up to  time $T$ of the first term on the right-hand side is bounded  by
	\begin{align*}
	& C 
	\sum_{i = \theta_n + 1}^{[T/\Delta_n]-L+1}
	\mathbb{E} \biggl[ \biggl\lvert \int_{(i + \ell - 2) \Delta_n}^{(i + \ell - 1) \Delta_n}  (a^k_s - a^k_{(i - \theta_n) \Delta_n} ) \, \dd s \biggr\rvert^p \biggr]^{\frac{1}{p}}\\
	&\qquad\lec\Delta_n^{1 - \frac{1}{p}}  \sum_{i = \theta_n + 1}^{[T/\Delta_n]-L+1}
	\mathbb{E} \biggl[  \int_{(i + \ell - 2) \Delta_n}^{(i + \ell - 1) \Delta_n} \lvert a^k_s - a^k_{(i - \theta_n) \Delta_n} \rvert^p \, \dd s \biggr]^{\frac{1}{p}}
	\\ & \qquad\lec 
	\biggl(\sum_{i = \theta_n + 1}^{[T/\Delta_n]-L+1} \mathbb{E} \biggl[ \int_{(i + \ell - 2) \Delta_n}^{(i + \ell - 1) \Delta_n} \lvert a^k_s - a^k_{(i - \theta_n) \Delta_n} \rvert^p \, \dd s \biggr]\biggr)^{\frac{1}{p}} \\ & \qquad= 
	  \mathbb{E} \biggl[ \int_{ (\theta_n+\ell - 1) \Delta_n}^{([T/\Delta_n]-L + \ell ) \Delta_n} 
	  \lvert a^k_s - a^k_{\left( [  {s}/{\Den}  ] - (\ell - 2 + \theta_n) \right) \Delta_n}    \rvert^p \, \dd s \biggr]^{\frac{1}{p}} 
	\end{align*}
for some $p>2$.
Since $a$ is càdlàg and bounded, the last integral above vanishes by dominated convergence as $\nto$.

Next, we use Taylor's theorem to write the second term on the right-hand side of \eqref{NR3} as 
	\begin{multline*} 
	\Delta_n  \sum_{i = \theta_n + 1}^{[t/\Delta_n]-L+1} a^k_{(i - \theta_n) \Delta_n}\biggl\{
	\mathbb{E} \biggl[ \partial_\al f \biggl(\frac{\si_{(i-1) \Den}\uDeni B}{\sqrt{\Delta_n} }\biggr)  
	\mathrel{\Big \vert} \mathcal{F}^n_{i - \theta_n} \biggr] \\ +  \sum_{\beta}
	\mathbb{E} \Big[ \pd^2_{\al\beta}  ( \xi''^{n,\dd}_i  ) \xi^{n,\dd}_i  	\mathrel{\Big \vert}  \mathcal{F}^n_{i - \theta_n} \Big]\biggr\},
	\end{multline*}
	where $\xi''^{n,\dd}_i$ is a point   between $\un{\Delta}^n_i Y^{\mathrm{dis}} / \sqrt{\Delta_n}$ and $\si_{(i-1) \Den}  \uDeni B/\sqrt{\Delta_n}$. A simple size estimate shows that the second part vanishes as $n\to\infty$. By first conditioning on $\calf^n_{i-1}$, we further see that the first conditional expectation is identically zero, because $\si_{(i-1) \Den} \uDeni B / \sqrt{\Delta_n}$ has a centered normal distribution given $\calf^n_{i-1}$ and $\pd_\al f$ is an odd function. 
	
Thus,	it remains to analyze $H^n_{12}(t)$. To keep the presentation simple, and since the general case does not involve any additional arguments, we assume $d=d'=k=\ell=L=1$. Then
	\begin{equation*}
	\si_s  - \si_{(i - 1) \Den} =  ( \si_s^{(0)} - \si^{(0)}_{(i - 1) \Den}   ) +
	\int_{(i-1) \Den}^{s} \wt{b}_r \, \dd r + \int_{(i-1) \Den}^{s}   \wt{\si}_r \, \dd B_r.
	\end{equation*}
	We then split $H^n_{12}(t)$ into $H^{n,1}_{12}(t) + H^{n,2}_{12}(t) + H^{n,3}_{12}(t)$ where
	\begin{align*}
	H^{n,1}_{12}(t) & =
	\Delta_n^{\frac{1}{2}} \sum_{i = \theta_n + 1}^{[t/\Delta_n]} \sum_{k=1}^d  
	\mathbb{E} \biggl[ f'\biggl(\frac{ {\Delta}^n_i Y^{\mathrm{dis}}}{\sqrt{\Delta_n} }\biggr) 
	\frac{1}{\DenSquared} \int_{(i -1) \Delta_n}^{i \Delta_n}  ( \si_s^{(0) } - \si^{(0) }_{(i - 1) \Den}  ) \, \dd B_s 
\mathrel{\bigg \vert}\mathcal{F}^n_{i - \theta_n} \biggr], \\
	H^{n,2}_{12}(t) & =
	\Delta_n^{\frac{1}{2}} \sum_{i = \theta_n + 1}^{[t/\Delta_n]}  
	\mathbb{E} \biggl[ f' \biggl(\frac{ {\Delta}^n_i Y^{\mathrm{dis}}}{\sqrt{\Delta_n} }\biggr) 
	\frac{1}{\DenSquared} \int_{(i -1) \Delta_n}^{i \Delta_n}  \biggl( 
	\int_{(i - 1) \Den}^{s} \wt{b}_r  \, \dd r
	\biggr) \, \dd B_s 
	\mathrel{\bigg \vert} \mathcal{F}^n_{i - \theta_n} \biggr], \\
	H^{n,3}_{12}(t) & =
	\Delta_n^{\frac{1}{2}} \sum_{i = \theta_n + 1}^{[t/\Delta_n]}  
	\mathbb{E} \biggl [ f' \biggl(\frac{ {\Delta}^n_i Y^{\mathrm{dis}}}{\sqrt{\Delta_n} }\biggr) 
	\frac{1}{\DenSquared} \int_{(i -1) \Delta_n}^{i\Delta_n}   \biggl( 
	\int_{(i - 1) \Den}^{s}   \wt{\si}_r  \, \dd B_r	\biggr) \, \dd B_s 
	\mathrel{\bigg \vert} \mathcal{F}^n_{i - \theta_n} \biggr ].
	\end{align*}
Given Assumption~\ref{Ass:Bprime},
	a size estimate shows that $H^{n,1}_{12}(t)$ and $H^{n,2}_{12}(t)$ are  of size $\Den^{\ga -  {1}/{2}}$ and $\Den^{1/2}$, respectively, both of which go to $0$ as $\nto$ (since $\ga > \frac12$). 
	Thus, the last term to be considered is
	 $ H^{n,3}_{12}(t)$, for which we must resolve to a last decomposition. We write $ H^{n,3}_{12}(t) = H^{n,3,1}_{12}(t) + H^{n,3,2}_{12}(t) + H^{n,3,3}_{12}(t)$ with
	\begin{align*}
	H^{n,3,1}_{12}(t) & = \mathtoolsset{multlined-width=0.8\displaywidth} \begin{multlined}[t] \Delta_n^{\frac{1}{2}} \sum_{i = \theta_n + 1}^{[t/\Delta_n]} 
	\mathbb{E} \biggl [ f'\biggl(\frac{{\Delta}^n_i Y^{\mathrm{dis}}}{\sqrt{\Delta_n} }\biggr)  \\
\times	\frac{1}{\DenSquared} \int_{(i -1) \Delta_n}^{i \Delta_n}  
	\int_{(i - 1) \Den}^{s}   (\wt{\si}_r  - \wt{\si}_{(i - 1) \Den} ) \, \dd B_r  \, \dd B_s 
\mathrel{\Big \vert} \mathcal{F}^n_{i - \theta_n} \biggr ],\end{multlined} \\
	H^{n,3,2}_{12}(t) & = \mathtoolsset{multlined-width=0.8\displaywidth} \begin{multlined}[t] \Delta_n^{\frac{1}{2}} \sum_{i = \theta_n + 1}^{[t/\Delta_n]} 
	\mathbb{E} \bigg [ \biggl\{f' \biggl(\frac{{\Delta}^n_i Y^{\mathrm{dis}}}{\sqrt{\Delta_n} }\biggr) - 
f' \biggl(\frac{ \si_{(i-1)\Den}  {\Delta}^n_i B}{\sqrt{\Delta_n} }\biggr) \biggr\} \\  
\times	\frac{1}{\DenSquared} \int_{(i -1) \Delta_n}^{i \Delta_n}  
	\int_{(i - 1) \Den}^{s}  \wt{\si}_{(i - 1) \Den} \, \dd B_r	  \, \dd B_s 
	\mathrel{\Big \vert} \mathcal{F}^n_{i - \theta_n} \bigg ], \end{multlined}\\ 
	H^{n,3,3}_{12}(t) & = \mathtoolsset{multlined-width=0.8\displaywidth} \begin{multlined}[t] \Delta_n^{\frac{1}{2}} \sum_{i = \theta_n + 1}^{[t/\Delta_n]}  
	\mathbb{E} \bigg [ f' \biggl( \frac{ \si_{(i-1)\Den} {\Delta}^n_i B}{\sqrt{\Delta_n} }\biggr) \\
\times	\frac{1}{\DenSquared} \int_{(i -1) \Delta_n}^{i \Delta_n} 
	\int_{(i - 1) \Den}^{s}  \wt{\si}_{(i - 1) \Den} \, \dd B_r  \, \dd B_s 
	\mathrel{\Big \vert} \mathcal{F}^n_{i - \theta_n} \bigg ].\end{multlined}
	\end{align*}

By the Hölder properties of $\wt \si$, we can use a size estimate to see that $H^{n,3,1}_{12}(t)$ is of size $\Den^{-1/2} \Den^{ {1}/{2} + \eps'} = \Den^{\eps'}$, which vanishes for large $n$. Similarly,
since 
	$( {\Delta}^n_i Y^{\mathrm{dis}} - \si_{(i-1)\Den}  {\Delta}^n_i B)/\sqrt{\Delta_n} = \xi^{n,\dd}_i$, we can use the mean-value theorem and a size estimate  to see that $H^{n,3,2}_{12}(t)$ is of size $\Den^{-1/2}\Den^{H-1/2}\Den^{1/2}$, which goes as $\nto$ by \eqref{iter:CLT}.	Finally, if we first condition on $\calf^n_{i-1}$ in $H^{n,3,3}_{12}(t)$, then because $f$ is even and $\si_{(i-1)\Den}  {\Delta}^n_i B/ \sqrt{\Delta_n}$ is centered Gaussian given $\calf^n_{i - 1}$, it follows that 
	$f' (\si_{(i-1)\Den}  {\Delta}^n_i B/ \sqrt{\Delta_n})$ belongs to the direct sum of all Wiener chaoses of odd orders. At the same time, the double stochastic integral  in $H^{n,3,3}_{12}(t)$ belongs to the second Wiener chaos (see  \cite[Proposition 1.1.4]{Nual}). Since Wiener chaoses are mutually orthogonal, conditioning on $\calf^n_{i - 1}$ shows that $H^{n,3,3}_{12}(t)\equiv0$. 
\end{proof}
 
\section{Proof of Theorem~\ref{thm:CLT2}}\label{app:B}

If $\la(s)\equiv0$, Theorem~\ref{thm:CLT2} is exactly \cite[Theorem~3.1]{Chong21}. If $\la(s)\not\equiv0$, there are only two places in the proof that require modifications: one is the proof of Lemma~C.2 and the other is the statement of Lemma~C.3 in the reference. Concerning the latter, the only difference is that we now have an additional term $\la((i-1)\Den)_{k\ell,k'\ell'}\Den^{1/2-H}$ in \cite[Equation~(C.7)]{Chong21}. Note that the same term (with $\frac12-H$ replaced by $H-\frac12$) appeared in  \eqref{eq:vni} in the case where $H>\frac12$. Thus, a straightforward combination of the proof of Proposition~\ref{prop:4} with the proof of \cite[Lemma~C.3]{Chong21} yields an ``update'' of the latter that covers the case $\la(s)\not\equiv0$.

In the following, we detail how to prove the statement of \cite[Lemma~C.2]{Chong21} if $\la(s)\not\equiv0$. As shown at the beginning of the proof of \cite[Lemma~C.2]{Chong21}, it suffices to prove $\ov U^n\LeinsKonv0$, where
\beq\label{eq:Unt}	\begin{split} \ov U^n(t)&= \DenOneHalf \sum_{i=\theta_n + 1}^{[t/\Den]-L+1} \Bigg\{ 
	f \bigg( \frac{\si_{(i-1)\Den} \un\Delta^n_i B + \xi^{n,\di}_i}{\Den^H} \bigg) - f \bigg( \frac{ \xi^{n,\di}_i }{\Den^H} \bigg) \\
&\quad	- 	\bbe \bigg[f \bigg( \frac{\si_{(i-1)\Den} \un\Delta^n_i B + \xi^{n,\di}_i}{\Den^H} \bigg) - f \bigg( \frac{ \xi^{n,\di}_i }{\Den^H} \bigg) \mathrel{\Big\vert}  \calf^n_{i - \theta_n} \bigg] \Bigg\} \end{split}\eeq
and 
	$ \xi^{n,\di}_i=\int_{(i-\theta_n)\Den}^{(i+L-1)\Den} \rho_{(i-\theta_n)\Den} \,\dd B_s\,\un\Delta^n_i g(s)$. We further decompose $ \xi^{n,\di}_i= \xi^{n,1}_i+ \xi^{n,2}_i$, where
	$$ \xi^{n,1}_i= \int_{(i-\theta_n)\Den}^{(i-1)\Den} \rho_{(i-\theta_n)\Den} \,\dd B_s\,\un\Delta^n_i g(s),\qquad \xi^{n,2}_i= \int_{(i-1)\Den}^{(i+L-1)\Den} \rho_{(i-\theta_n)\Den} \,\dd B_s\,\un\Delta^n_i g(s). $$

Since $(N(H)+1)(\frac12-H)>\frac12$ by \eqref{eq:NH}, applying Taylor's theorem twice, we obtain
\begin{align*}
	&	f \bigg( \frac{\si_{(i-1)\Den} \un\Delta^n_i B + \xi^{n,\di}_i}{\Den^H} \bigg) - f \bigg( \frac{ \xi^{n,\di}_i }{\Den^H} \bigg) \\
	&\qquad=\sum_{j=1}^{N(H)} \sum_{\lvert \chi\rvert =j}\frac1{\chi!} \partial^\chi f\bigg( \frac{ \xi^{n,\di}_i }{\Den^H} \bigg) \bigg( \frac{\si_{(i-1)\Den} \un\Delta^n_i B }{\Den^H} \bigg)^\chi + o_\bbp(\Den^{\frac12})\\
	&\qquad=\sum_{j,j'=1}^{N(H)} \sum_{\lvert \chi\rvert =j}\sum_{\lvert \chi\rvert =j'}\frac1{\chi!\chi'!} \partial^{\chi+\chi'} f\bigg( \frac{ \xi^{n,1}_i }{\Den^H} \bigg)\bigg( \frac{ \xi^{n,2}_i }{\Den^H} \bigg)^{\chi'} \bigg( \frac{\si_{(i-1)\Den} \un\Delta^n_i B }{\Den^H} \bigg)^\chi + o_\bbp(\Den^{\frac12}).
\end{align*}
For each $i$, if we condition on $\calf^n_{i-1}$, It\^o's representation theorem implies that
\beq\label{eq:Ito}\begin{split} &\bigg( \frac{ \xi^{n,2}_i }{\Den^H} \bigg)^{\chi'} \bigg( \frac{\si_{(i-1)\Den} \un\Delta^n_i B }{\Den^H} \bigg)^\chi \\ &\qquad =\E\biggl[\bigg( \frac{ \xi^{n,2}_i }{\Den^H} \bigg)^{\chi'} \bigg( \frac{\si_{(i-1)\Den} \un\Delta^n_i B }{\Den^H} \bigg)^\chi \mathrel{\Big|}\calf^n_{i-1}\biggr]+\int_{(i-1)\Den}^{(i+L-1)\Den} \zeta^{n,\chi,\chi'}_i(s)\,\dd B_s \end{split}\eeq
for some process $\zeta^{n,\chi,\chi'}_i(s)$ that is $O_\bbp(\Den^{-H})$, uniformly in $\chi$, $\chi'$, $i$ and $s$, because $\lvert \chi\rvert\geq1$. The stochastic integral in the last line is $\calf^n_{i+L-1}$-measurable with a zero $\calf^n_{i-1}$-conditional expectation. Therefore, by a martingale size estimate, its contribution to $\ov U^n$ is of magnitude $\Den^{1/2-H}$, which is negligible. Therefore, it remains to consider the conditional expectation in \eqref{eq:Ito}. Note that 
\begin{align*}
	&\E\biggl[\bigg( \frac{ \xi^{n,2}_i }{\Den^H} \bigg)^{\chi'} \bigg( \frac{\si_{(i-1)\Den} \un\Delta^n_i B }{\Den^H} \bigg)^\chi \mathrel{\Big|}\calf^n_{i-1}\biggr]\\
	 &\qquad= \E\biggl[\bigg( \frac{ \xi^{n,2}_i }{\Den^H} \bigg)^{\chi'} \bigg( \frac{\si_{(i-\theta_n)\Den} \un\Delta^n_i B }{\Den^H} \bigg)^\chi \mathrel{\Big|}\calf^n_{i-1}\biggr] + O_\bbp({(\theta_n\Den)^{\frac12}})\\
	 &\qquad= \E\biggl[\bigg( \frac{ \xi^{n,2}_i }{\Den^H} \bigg)^{\chi'} \bigg( \frac{\si_{(i-\theta_n)\Den} \un\Delta^n_i B }{\Den^H} \bigg)^\chi \mathrel{\Big|}\calf^n_{i-\theta_n}\biggr] + O_\bbp({(\theta_n\Den)^{\frac12}}),
\end{align*}
with exact equality between the last two conditional expectations. Clearly, the last conditional expectation does not contribute to $\ov U^n$, since it is removed by the second line of \eqref{eq:Unt}. So the proof is complete upon noticing that the contribution of the $O_\bbp({(\theta_n\Den)^{1/2}})$-term to $\ov U^n$ is negligible by a martingale size estimate. \halmos

\section{Additional simulation results}\label{app:sim}

In this appendix, we report additional simulation results in a setting where $\si$ and $\rho$ are fixed (instead of the signal-to-noise ratio). Everything  is chosen and defined in the same way as in Section~\ref{sec:sim} with the following exceptions: We take $\si=0.02$ and $\rho=0.02/\sqrt{10}$ if $H<\frac12$ and $\si=0.02/\sqrt{10}$ and $\rho=0.02$ if $H>\frac12$. Moreover, as the noise can become small now as $H\to1$, we do not set $\mathrm{score}(H)=\infty$ if in the estimated model, noise accounts for less than 1\% of the return variance. Tables~\ref{tab:5}--\ref{tab:8} summarize the results in this new setup.

The results for $H_n$ are similar to those in Table~\ref{tab:1}, except that for $\la=0.5$ and $\la=0.9$, the bias is  smaller now if $H<\frac12$ and bigger if $H>\frac12$. This makes sense because with fixed $\rho$, the noise accounts for a higher (resp., lower) percentage of the total return variance if $H$ is small (resp., large), making inference of $H$ easier (resp., more difficult). The situation is different for $C_n$: since $\rho$ is fixed, estimation of volatility faces two challenges at the same time. On the one hand, if $H$ is small, volatility can no longer be estimated consistently according to our theory. On the other hand, if $H$ gets closer to $\frac12$, separating the fractional and the Brownian components becomes increasingly difficult (see the discussion at the end of Section~\ref{sec:lb}). This is why in Table~\ref{tab:6}, we can see that the performance of $C^n$ worsens both as $H$ gets close to $0$ and as $H$ gets close to $\frac12$, with this effect being stronger for positive values of $\lambda$. Regarding estimation of $\lambda$ itself, we can see as in Section~\ref{sec:sim} that estimation is easier for negative values of $\la$ but now also if $H$ is small, leading to a good performance on the upper-left part of Table~\ref{tab:7}. Finally, as we can see from Table~\ref{tab:8}, the estimator $\Pi_n$ performs well if $H$ is small or larger but close to $\frac12$, with better results for negative values of $\lambda$.

\begin{table*}
	\caption{Median and interquartile range  of $H_n$ based on 1{,}000 simulated paths.}
	\label{tab:5}
	\begin{tabular}{@{}cccccc@{}}
		& \multicolumn{5}{c}{$\la$}\\
		\hline 
		$H$ & \multicolumn{1}{c}{$-0.9$}&\multicolumn{1}{c}{$-0.5$}
		& \multicolumn{1}{c}{0} & \multicolumn{1}{c}{$0.5$}&\multicolumn{1}{c}{$0.9$}
		\\
		\hline
		$0.1$  & 0.1001 &  						0.1001&   					0.1001   & 					0.1002& 					0.1010  \\
		~ &  [0.0979, 0.1031]  & [0.0968, 0.1034] & [0.0966,  0.1039]& [0.0977, 0.1042]& [0.0993,   0.1044] \smallskip \\
		$0.2$    & 0.2000 &						0.2000   & 					0.2000    &					0.2001&  				0.2022  \\
		~ &[0.1956,     0.2046]&[0.1939,  0.2058]&[0.1916,  0.2076]&[0.1933,  0.2095]&[0.1982,   0.2112] \smallskip \\
		$0.3$  &   0.3000   &   				0.2999   &   				0.2999  &    					0.3004  &  					0.3064   \\
		~& [0.2948,  0.3049]  &[0.2867,  0.3115]  & [0.2710, 0.3208]  & [0.2835, 0.3309] & [0.2967, 0.3383] \smallskip \\
		$0.4$  &   0.3996  &  					0.3958   &    				0.3979  &  					0.4218 &   						0.4317 \\
		~& [0.3787,  0.4200]  & [0.2733, 0.4641] & [0.3516, 0.4916]  & [0.3806,  0.5053]  & [0.3955,  0.5111]\smallskip \\
		$0.5$ &   &  & 0.5000 &  & \\
		~&   &  & [0.5000,0.5000] &  &\smallskip \\
		$0.6$  & 0.5996 &  						0.5970&   					0.5948   & 					0.5838& 					0.5732  \\
		~ &  [0.5880, 0.6119]  & [0.5639, 0.6558] & [0.5394,  0.6306]& [0.5223, 0.6113]& [0.5116,   0.6010] \smallskip \\
		$0.7$    & 0.7001 &						0.6985   & 					0.6969    &					0.6918&  				0.6802  \\
		~ &[0.6901,     0.7099]&[0.6762,  0.7244]&[0.6532,  0.7783]&[0.6334,  0.7323]&[0.6205,   0.7047] \smallskip \\
		$0.8$  &   0.8010   &   				0.7997   &   				0.7974  &    					0.7735  &  					0.7534   \\
		~& [0.7691,  0.8264]  &[0.7509,  0.8441]  & [0.7221, 0.9220]  & [0.6934, 0.8178] & [0.6791, 0.7879] \smallskip \\
		$0.9$  &   0.8911  &  					0.9002   &    				0.7437  &  					0.8160 &   						0.7956 \\
		~& [0.8527,  0.9219]  & [0.7834, 0.9734] & [0.5000, 0.9780]  & [0.7006,  0.8887]  & [0.6963,  0.8558]\smallskip \\
		\hline
	\end{tabular}
\end{table*}

\begin{table*}
	\caption{Median and interquartile range  of $C^{[19,20]}_n/\si^2$ based on 1{,}000 simulated paths.}
	\label{tab:6}
	\begin{tabular}{@{}cccccc@{}}
		& \multicolumn{5}{c}{$\la$}\\
		\hline 
		$H$ & \multicolumn{1}{c}{$-0.9$}&\multicolumn{1}{c}{$-0.5$}
		& \multicolumn{1}{c}{0} & \multicolumn{1}{c}{$0.5$}&\multicolumn{1}{c}{$0.9$}
		\\
		\hline
		$0.1$  & 1.1205 &  						1.0224&   					1.0854   & 					1.2041& 					1.4105  \\
		~ &  [0.6911, 1.8850]  & [0.2629, 1.9885] & [0.2288,  2.1059]& [0.5416, 2.3145]& [0.8991,   2.5605] \smallskip \\
		$0.2$    & 0.9998 &						1.0129   & 					1.0062    &					1.0993&  				1.2680  \\
		~ &[0.7152,     1.3542]&[0.5855,  1.4849]&[0.4822,  1.6647]&[0.6873,  1.8406]&[0.9197,   2.0637] \smallskip \\
		$0.3$  &   0.9964   &   				0.9950   &   				1.0260  &    					1.1928  &  					1.3398   \\
		~& [0.8904,  1.1313]  &[0.7486,  1.3184]  & [0.7168, 1.5686]  & [0.8755, 1.8944] & [0.9798, 2.1662] \smallskip \\
		$0.4$  &   0.9796  &  					0.9339   &    				1.3600  &  					2.0031 &   						2.6196 \\
		~& [0.7744,  1.3489]  & [0.7267, 3.4177] & [0.9757, 7.7919]  & [1.1330,  12.5047]  & [1.2537,  15.6350]\smallskip \\
		$0.5$ &   &  & 1.0006 &  & \\
		~&   &  & [0.9942,1.0070] &  &\smallskip \\
		$0.6$  & 1.0247 &  						1.0804&   					1.5038   & 					2.3624& 					3.3730  \\
		~ &  [0.6932, 1.4451]  & [0.5475, 3.4861] & [0.8567,  11.9466]& [1.0950, 31.4528]& [1.2548,   40.4074] \smallskip \\
		$0.7$    & 0.9995 &						1.0032   & 					1.0183    &					1.0879&  				1.1341  \\
		~ &[0.9527,     1.0548]&[0.9245,  1.1143]&[0.9502,  1.1973]&[0.9906,  1.3255]&[0.9882,   1.4884] \smallskip \\
		$0.8$  &   1.0040   &   				1.0011   &   				1.0083  &    					1.0243  &  					1.0237   \\
		~& [0.9734,  1.0394]  &[0.9692,  1.0436]  & [0.9861, 1.0540]  & [0.9828, 1.0821] & [0.9742, 1.1045] \smallskip \\
		$0.9$  &   1.0078  &  					1.0038   &    				1.0055  &  					1.0128 &   						1.0119 \\
		~& [0.9930,  1.0252]  & [0.9871, 1.0290] & [0.9953, 1.0206]  & [0.9888,  1.0533]  & [0.9830,  1.0619]\smallskip \\
		\hline
	\end{tabular}
\end{table*}

\begin{table*}
	\caption{Median and interquartile range  of $\La_n/\sqrt{C_n\Pi_n}$ (defined as $0$ if $\Pi_n=0$) based on 1{,}000 simulated paths. }
	\label{tab:7}
	\begin{tabular}{@{}cccccc@{}}
		& \multicolumn{5}{c}{$\la$}\\
		\hline 
		$H$ & \multicolumn{1}{c}{$-0.9$}&\multicolumn{1}{c}{$-0.5$}
		& \multicolumn{1}{c}{0} & \multicolumn{1}{c}{$0.5$}&\multicolumn{1}{c}{$0.9$}
		\\
		\hline
		$0.1$  & -0.9130 &  						-0.5415&   					-0.0022   & 					0.4947& 					0.8934  \\
		~ &  [-0.9998, -0.8791]  & [-0.6219, -0.4669] & [-0.2094,  0.4411]& [0.1439, 1.0000]& [0.4077,   1.0000] \smallskip \\
		$0.2$    & -0.9055 &						-0.4986   & 					0.0004    &					0.4948&  				0.8875  \\
		~ &[-0.9181,     -0.8934]&[-0.5836,  -0.3861]&[-0.2166,  0.3359]&[0.1079,  0.9999]&[0.3431,   1.0000] \smallskip \\
		$0.3$  &   -0.8994   &   				-0.5003   &   				-0.0085  &    					0.4986  &  					0.9108   \\
		~& [-0.9122,  -0.8860]  &[-0.6107,  -0.3427]  & [-0.3566, 0.6503]  & [-0.2141, 0.9998] & [-0.1568, 0.9999] \smallskip \\
		$0.4$  &   -0.8991  &  					-0.4771   &    				0.0549  &  					0.5632 &   						0.5502 \\
		~& [-0.9251,  -0.8742]  & [-0.9071, 0.2298] & [-0.9277, 0.7623]  & [-0.9263,  0.9942]  & [-0.9356,  0.9974]\smallskip \\
		$0.6$  & -0.9013 &  						-0.5357&   					-0.1335   & 					0.1554& 					0.2608  \\
		~ &  [-0.9310, -0.8597]  & [-0.8635, 0.5909] & [-0.9319,  0.9977]& [-0.9630, 0.9981]& [-0.9626,   0.9982] \smallskip \\
		$0.7$    & -0.9010 &						-0.5013   & 					-0.0178    &					0.3951&  				0.6199  \\
		~ &[-0.9029,     -0.8989]&[-0.5617,  -0.4484]&[-0.3057,  0.5923]&[-0.2607,  0.9602]&[-0.3045,   1.0000] \smallskip \\
		$0.8$  &   -0.9147   &   				-0.5159   &   				-0.0105  &    					0.3719  &  					0.5060   \\
		~& [-0.9453,  -0.8897]  &[-0.5441,  -0.4991]  & [-0.1827, 0.4771]  & [-0.0893, 0.6337] & [-0.0772, 0.7574] \smallskip \\
		$0.9$  &   -0.9988  &  					-0.8723   &    				0.0000  &  					0.2020 &   						0.2100 \\
		~& [-0.9999,  -0.9503]  & [-1.0000, -0.4769] & [-0.1706, 0.1241]  & [-0.1490,  0.4384]  & [-0.1413,  0.5077]\smallskip \\
		\hline
	\end{tabular}
\end{table*}

\begin{table*}
	\caption{Median and interquartile range  of $\Pi_n/(20\rho^2)$ based on 1{,}000 simulated paths.}
	\label{tab:8}
	\begin{tabular}{@{}cccccc@{}}
		& \multicolumn{5}{c}{$\la$}\\
		\hline 
		$H$ & \multicolumn{1}{c}{$-0.9$}&\multicolumn{1}{c}{$-0.5$}
		& \multicolumn{1}{c}{0} & \multicolumn{1}{c}{$0.5$}&\multicolumn{1}{c}{$0.9$}
		\\
		\hline
		$0.1$  & 1.0034 &  						1.0015& 					  1.0006   & 				1.0023& 					1.0243  \\
		~ &  [0.9521, 1.0769]  & [0.9257, 1.0872] & [0.9206,  1.0968]& [0.9454, 1.1051]& [0.9825,   1.1117] \smallskip \\
		$0.2$    & 1.0027 &						0.9974   & 					1.0006    &				1.0018&  				1.0565  \\
		~ &[0.8924,     1.1230]&[0.8500,  1.1664]&[0.8002,  1.2292]&[0.8289,  1.2977]&[0.9521,   1.3675] \smallskip \\
		$0.3$  &   1.0023   &   			1.0005   &   				0.9923  &    				0.9878  &  					1.1980   \\
		~& [0.8586,  1.1562]  &[0.6625,  1.4375]  & [0.3942, 1.9990]  & [0.5689, 2.8596] & [0.8900, 3.8269] \smallskip \\
		$0.4$  &   0.9888  &  					0.8557   &    				0.7420  &  				1.9439 &   						4.5755 \\
		~& [0.4800,  2.1476]  & [0.0096, 28.9549] & [0.0986, 109.5318]  & [0.3489,  160.2848]  & [0.8019,  196.3857]\smallskip \\
		$0.6$  & 1.0066 &  						1.0344& 					  1.0434   & 				0.9439& 					1.0730  \\
		~ &  [0.9526, 1.0821]  & [0.7128, 1.6410] & [0.6644,  3.0806]& [0.8191, 4.0318]& [0.9700,   4.2804] \smallskip \\
		$0.7$    & 1.0017 &						0.9969   & 					1.0364    &				1.0590&  				1.1455  \\
		~ &[0.9178,     1.0935]&[0.8803,  1.1761]&[0.9405,  1.1606]&[0.9048,  1.2191]&[0.9799,   1.4538] \smallskip \\
		$0.8$  &   0.9999   &   			0.9977   &   				0.9899  &    				0.8877  &  					0.9944   \\
		~& [0.6222,  1.4662]  &[0.5992,  1.6541]  & [0.6365, 1.8533]  & [0.7572, 1.0697] & [0.8877, 1.1227] \smallskip \\
		$0.9$  &   0.7287  &  					0.9678   &    				0.3615  &  				0.6399 &   						0.7288 \\
		~& [0.3366,  1.5829]  & [0.1343, 2.4936] & [0.0000, 1.7291]  & [0.4181,  1.0834]  & [0.5455,  1.0401]\smallskip \\
		\hline
	\end{tabular}
\end{table*}

\end{appendix}

\end{document}